\documentclass[pdflatex,sn-mathphys-num,referee]{sn-jnl}

\usepackage{graphicx}%
\usepackage{multirow}%
\usepackage{amsmath,amssymb,amsfonts}%
\usepackage{amsthm}%
\usepackage{mathrsfs}%
\usepackage[title]{appendix}%
\usepackage{xcolor}%
\usepackage{textcomp}%
\usepackage{manyfoot}%
\usepackage{booktabs}%
\usepackage{algorithm}%
\usepackage{algorithmicx}%
\usepackage{algpseudocode}%
\usepackage{listings}%
\usepackage{comment}
\usepackage{makecell}
\usepackage{adjustbox}
\usepackage{etoolbox}

\makeatletter
\AtBeginDocument{%
  \AtBeginEnvironment{appendices}{%
  }%
}
\makeatother

\theoremstyle{thmstyleone}%
\newtheorem{theorem}{Theorem}
\newtheorem{proposition}[theorem]{Proposition}%
\newtheorem{lemma}[theorem]{Lemma}

\theoremstyle{thmstyletwo}%

\theoremstyle{thmstylethree}%

\newcommand{\gap}{\vspace{0.1in}}

\newcommand{\epc}{\hspace{1pc}}

\newcommand{\onebld}{\mathbf{1}}

\newcommand{\bs}{\boldsymbol}
\newcommand{\wh}{\widehat}

\usepackage{geometry}
\begin{document}

\title[Article Title]{Solving Constrained Affine Heaviside Composite 
Optimization Problems by a Progressive IP Approach}


\author[1]{\fnm{Ke} \sur{Zheng}}\email{zhengk24@mails.tsinghua.edu.cn}

\author*[1]{\fnm{Junyi} \sur{Liu}}\email{
junyiliu@tsinghua.edu.cn}

\author[1]{\fnm{Yurui} \sur{Wang}}\email{wang-yr25@mails.tsinghua.edu.cn}

\author[2]{\fnm{Jong-Shi} \sur{Pang}}\email{
jongship@usc.edu}

\affil[1]{\orgdiv{Department of Industrial Engineering}, \orgname{Tsinghua University}, \orgaddress{\city{Beijing}, \postcode{100084}, \country{China}}}

\affil[2]{\orgdiv{The Daniel J. Epstein Department of Industrial and Systems Engineering}, \orgname{University of Southern California}, \orgaddress{\city{Los Angeles}, \state{CA}, \postcode{90007}, \country{USA}}}
 

\abstract{This paper discusses the computational resolution and presents
numerical results for solving affine
combinations of Heaviside composite optimization problems (abbreviated
as A-HSCOPs) by a progressive integer programming (abbreviated as PIP)
method.  The characteristics of these problems are
that the Heaviside functions, which appear in the objective and define the constraints, are discontinuous, and their mixed-signed combinations result
in the overall objective lacking the matching semicontinuity needed
for the optimization and in the feasible set being not necessarily closed.
Added
to these challenging properties is the nondifferentiability of the inner
functions in the composition.  In this paper, we propose resolutions to
all these challenges by first an approximation to remedy the lack of 
semicontinuity in the objective and closedness in the constraints, followed
by a progressive integer programming approach with successive decomposition 
to handle the intrinsically
discrete nature of the Heaviside function.  Convergence to the local 
optimizers of the given Heaviside 
optimization problem is established.  The effectiveness of the overall
solution strategy is supported by extensive computational experiments
on the score-based and tree-based multiclass classification problems with precision constraints.}




\maketitle

\section{Introduction}\label{sec:Intro}
We consider a class of discontinuous maximization problem that lacks upper 
semicontinuity (u.s.c.) in the objective and closedness in the feasible set,
both defined by affine combinations of Heaviside composite functions.  These
unusual and challenging characteristics are caused by a univariate Heaviside
function, which is the indicator of an open or closed interval 
$( 0, \infty)$ or $[ 0, \infty)$, which we call an open or closed 
Heaviside function, respectively.  By its name, a Heaviside composite 
function is the composition of a Heaviside function (the outer function) with 
a continuous multivariate function
(the inner function) that may be nonconvex and nondifferentiable (for 
instance, piecewise affine). The Heaviside composite optimization problem 
was first studied in \cite{CuiLiuPang23-piecewise} with positive affine 
combinations of Heaviside composite functions that ensure upper semicontinuity
of the objective function (to be maximized) and closedness of the feasible 
set.  Subsequent to \cite{CuiLiuPang23-piecewise}, 
the follow-up paper \cite{HanCuiPang24} considers affine 
combinations with mixed-signed coefficients; the resulting (maximization)
problem, which we term an affine Heaviside composite optimization problem
and abbreviate as A-HSCOP, thus has a non-u.s.c.\ objective and a non-closed 
feasible set, it therefore violates the two basic requirements in the 
classical 
Bolzano-Weierstrass theorem: namely, functional semicontinuity and 
set-theoretic closedness.  Lacking these two essential properties poses 
critical challenges to the A-HSCOP in both theory and computation. This may 
explain why this class of discontinuous problems have not been broadly 
studied, in spite of their wide applications, which include policy learning 
\cite{liu2026offline} in economics, hierarchical variable/constraint 
selection \cite{BienTaylorTibshirani2013} in statistics, decision trees 
\cite{BertsimasDunn17} in classification, and precision constraints 
in machine learning, to name a few.   

In the present paper, we focus on the A-HSCOP with mixed-signed combinations 
of the Heaviside functions.  
Lifting of the problem was the principle tool employed in \cite{HanCuiPang24}
to resolve the lack of the two basic properties. The drawback
of the resulting lifted problem is that it 
is highly nonconvex and nondifferentiable and quite complex to be solved in 
practice, even by the growing pool of numerical methods based on surrogation 
\cite{CuiPang2021,PangPeng2025}. Another approach for resolving the 
A-HSCOP is to construct continuous approximations to the 
Heaviside composite functions, including the piecewise approximations 
\cite{CuiLiuPang22-ccp,CuiLiuPang23-piecewise}, integration-based
convolution \cite{ENWets95}, and smoothing approaches 
\cite{QiCuiLiuPang19,Chen12}. While the continuous approximation usually 
leads to nonconvex problems, it provides a viable approach for algorithms 
to be developed given the advances in nonconvex optimization. As shown in 
\cite{QiCuiLiuPang19} the smoothing approach works arguably well for 
unconstrained treatment learning problems. However, when incorporating 
context-related Heaviside-based constraints, it is easy to get stuck at 
solutions with inferior quality. Moreover, there is insufficient 
understanding of the limiting properties of the computed solution sequence 
reference to the original discontinuous problem.  The 
article \cite{CuiLiuPang23-piecewise} is one of a few that demonstrates that 
the limit point of Bouligand-stationary solutions to the pointwise 
approximated problems is a pseudo-stationary point, which is a very weak
kind of stationary solutions.

Arguably, one obvious
way to avoid the semicontinuity/closedness issue is to begin with the 
modeling;
that is, create a mathematical model with favorable semicontinuity 
of the objective function and needed closedness of the constraint set 
by using the closed Heaviside function when multiple such functions are 
positively combined and using 
open Heaviside functions when they are negatively combined.  However, 
as shown by our numerical results, this mathematically dictated modeling
is at the expense of sacrificing model faithfulness; the consequence
of such loosely constructed models is that they diminish the 
ultimate performance of the models' intended goals.

\subsection{Problem formulation and an algorithmic sketch}

In this paper, we address an A-HSCOP where the coefficients are of
mixed signs and the inner functions are piecewise affine.  The formal 
mathematical 
formulation is as follows: 
\begin{equation} \label{eq:A_HSCOP}
\begin{array}{ll}
\underset{x\in P}{\textbf{maximize}} &
\theta_{\text{AHS}}(x) \triangleq c(x) + \displaystyle{
\sum_{j=1}^{J_0}
} \, \psi_{0j} \, \mathbf{1}_{[ \, 0,\infty )}(\phi_{0j}(x)) \\ [0.1in]
\textbf{subject to} 
& A_{ i \bullet } x + \displaystyle{
\sum_{j=1}^{J_i}
} \, \psi_{ij} \, \onebld_{\left[ \, 0 , \infty \right)}(\phi_{ij}(x))  
\, \geq \, \eta_i, \qquad  \mbox{for all $i = 1, \cdots, I$}
\end{array}
\end{equation}
where $P$ is a closed convex set (e.g.\ a polyhedron), the stand-alone 
function $c(x)$ has favorable properties (e.g., linear), $A_{i \bullet}$ is
the $i$-th row of an $I \times n$ matrix $A$,
and each 
$\phi_{ij}$ is the sum of a convex and a concave piecewise affine (PA)
function:
\begin{equation} \label{eq:PA phi}
\phi_{ij}(x) \, = \, \underbrace{\displaystyle{
\max_{1 \leq k \leq K_{ij}}
} \, \left[ \, ( a_{ij}^k )^{\top}x + \alpha_{ij}^k \, \right]}_{
\mbox{denoted $\phi_{ij}^{\rm cvx}(x)$}} 
+ \underbrace{\displaystyle{
\min_{1 \leq \ell \leq L_{ij}}
} \, \left[ \, ( b_{ij}^{\, \ell} )^{\top}x + \beta_{ij}^{\, \ell} 
\, \right]}_{\mbox{denoted $\phi_{ij}^{\rm cve}(x)$}}, 
\epc x \, \in \, \mathbb{R}^n
\end{equation}
for some constant vectors $a_{ij}^k$ and $b_{ij}^{\, \ell}$
and scalars $\alpha_{ij}^k$ and $\beta_{ij}^{\, \ell}$; thus $\phi_{ij}$
is a difference-of-convex (dc) function.  As formulated,
the above-mentioned computational difficulties become apparent: 
the objective function may not be upper semicontinuous and
the feasible set may not be closed.  The difference-of-convexity of
the inner piecewise functions $\phi_{ij}$ adds further challenges to
the practical solution of the problem.

In this paper, we develop an iterative decomposed shrinkage algorithm (IDSA) 
with the progressive integer programming (PIP) method as the principal 
computational tool for efficiently computing a locally optimal solution 
of the A-HSCOP. 
The PIP idea originates from the computational study of the A-HSCOP with 
positive coefficients for policy learning by \cite{YueFang23}. An integer 
programming approach is natural due to the 
combined continuous-discrete features in the Heaviside 
composite functions.  Moreover, the IP approach can easily preserve 
the linearity structure of the problem to a large extent, unlike the 
smoothing approach \cite{CuiLiuPang23-piecewise} or a lifting approach
\cite{HanCuiPang24} that lacks this advantage.  The IP approach has 
the added benefit that a globally optimal solution can in principle be 
computed if desired, provided that there is no computational budget or time
restrictions. One goal in our computational study is to demonstrate
that in situations where globally optimal solutions can be computed by 
a global solver,
PIP, which in general computes only suboptimal solutions, can obtain
high-quality suboptimal solutions in significantly less time required by the
global solver. It will also illustrate the power of the PIP scheme for 
large-scale constrained A-HSCOP problems when the global solver fails to 
find feasible solutions within the time limit. 

There are two prominent piecewise structures in the Heaviside 
composite function 
$\onebld_{[ \, 0,\infty )} \circ \phi_{ij}$;
one is the outer Heaviside function $\onebld_{[ \, 0,\infty )}$
and the other is the inner function $\phi_{ij}$.  These are handled by
the two computational constructions that are the basis for the design 
of IDSA.  We describe these two steps below along with an overview of 
the computational study of IDSA.

\noindent $\bullet $ {\bf Approximation:}
To deal with the non-upper-semicontinuity of A-HSCOP, we construct a
u.s.c.\ problem with a closed constraint set by approximating each
Heaviside term $\psi_{ij} \onebld_{[ \, 0,\infty )}(\phi_{ij}(x))$ 
for $\psi_{ij}<0$ with an upper semicontinuous minorant 
$\psi_{ij} \onebld_{( \, -\varepsilon,\infty )}(\phi_{ij}(x))$
for some $\varepsilon > 0$. We 
show that the approximated problem is equivalent to the A-HSCOP in terms of 
their local maximizers under a {\it local sign-invariance} condition;
see Proposition~\ref{pr:N&S II locmax AHSCOP}.

\noindent $\bullet $ {\bf Decomposition:} 
To address the PA structure of the
inner functions, we employ a strategy similar to the one for a dc program
with difference of pointwise maximum functions
\cite{PangRazaAlvarado17}, leading to a family of decomposed approximated 
subproblems, which are the computational workhorse of IDSA. We show that 
these subproblems preserve the same local optimality with reference to the 
original A-HSCOP; see Proposition~\ref{pr:N&S III locmax AHSCOP}. 

\noindent $\bullet $ {\bf PIP:}  We solve each approximated and decomposed 
subproblem as an integer program obtained by substituting each Heaviside
term by a binary variable; the latter is a mixed-integer linear program
if $c(x)$ is linear.  The method is enhanced by the progressive idea
wherein IPs of reduced sizes are solved to save on computational efforts.
Convergence of IDSA is proved and implementation is performed.


\noindent $\bullet $ {\bf Computational study:}
We conduct case studies of IDSA on precision-constrained classification 
and decision tree problems, whose straightforward IP formulations are
of large-scale when derived from data-driven applications. The obtained
results illustrate the significant numerical advantages of the PIP-based IDSA 
for efficiently computing high-quality suboptimal solutions of the A-HSCOP
and dominating Pareto curves in comparison with other prominent IP approaches 
based on a full-integer formulation. 

The rest of this paper provides details of the above outline and is 
organized as follows.
In Section~\ref{sec:approx formulation}, we describe the approximation 
strategy and its local equivalence with the A-HSCOP in terms of their 
local maximizers under the local sign invariance condition.  
In Section~\ref{sec:PA decomp} we introduce the 
piecewise affine decomposition strategy and characterize the global/local 
maximizers for the decomposed subproblems. In Section~\ref{sec:ISA-PIP}, 
we present the PIP-based iterative decomposed shrinkage algorithm that 
combines the iterative approximation and decomposition, and analyze the local 
optimality of the limit points of a sequence of subproblems relative to 
the given A-HSCOP.  Finally,
in Section~\ref{sec:multiclass classification}, we present the case studies 
along with numerical results for solving multiclass classification problems.   

\section{The Local Semicontinuous Equivalent: $\varepsilon$-Approximation}
\label{sec:approx formulation}


Approximating each Heaviside term 
$\psi_{ij} \onebld_{[ \, 0,\infty )}(\phi_{ij}(x))$ with $\psi_{ij}<0$, 
by an upper semicontinuous term 
$\psi_{ij} \onebld_{( \, -\varepsilon,\infty )}(\phi_{ij}(x))$
for some $\varepsilon > 0$, we obtain the $\varepsilon$-approximating 
problem:
\begin{equation} \label{eq:partitioned I AHSCOP} 
\begin{array}{ll}
\left( \, \mbox{\bf P}_{\rm AHS}^{\, \varepsilon} \, \right) \epc
\displaystyle{
\operatornamewithlimits{\mbox{\bf maximize}}_{x \in P}
}  & \theta^{\, \varepsilon}(x) \, \triangleq \, c(x) \, + 
\displaystyle{
\sum_{j=1}^{J_0}
} \, \psi_{0j}^+ \, \onebld_{[ \, 0,\infty )}( \phi_{0j}(x) ) -
\displaystyle{
\sum_{j=1}^{J_0}
} \, \psi_{0j}^- \, 
\onebld_{( \, -\varepsilon,\infty )}( \phi_{0j}(x) ) \\ 
\hspace{0.7in} \mbox{\bf subject to} & \mbox{for all $i \in [ \, I \, ]
\triangleq \{ \, 1, \cdots, I \, \}$:} \\ [0.1in]
& A_{\, i \bullet} \, x +   \displaystyle{
\sum_{j=1}^{J_i}
} \, \psi_{ij}^+ \, \onebld_{[ \, 0,\infty )}( \phi_{ij}(x) ) -
 \displaystyle{
\sum_{j=1}^{J_i}
} \, \psi_{ij}^- \, 
\onebld_{( \, -\varepsilon,\infty )}( \phi_{ij}(x) )  \, \geq \, \eta_i, 
\end{array}  
\end{equation}
where $\psi_{ij}^{\pm} \triangleq \max( \pm \psi_{ij},0 )$ are the nonnegative
($+$) and nonpositive ($-$) parts of the scalar $\psi_{ij}$, respectively.
We note several properties of the open Heaviside function 
$\onebld_{( \, -\varepsilon,\infty )}( \bullet )$, which are instrumental 
in the following analysis.  The proof of these properties is omitted.

\begin{lemma} \label{lm:epsilon_approximation} \rm
\noindent {\bf (i)} 
$\onebld_{( \, -\varepsilon,\infty )}(t) \, \geq \,
\onebld_{( \, -\varepsilon^{\, \prime},\infty )}(t) \, \geq \,
\onebld_{[ \, 0, \infty )}(t), 
\epc \forall \, \varepsilon \, > \, \varepsilon^{\, \prime} 
\, > \, 0 \ \mbox{ and } \ \forall \, t \, \in \, \mathbb{R}$;

\noindent {\bf (ii)} $\displaystyle{
\lim_{\varepsilon \downarrow 0}
} \, \onebld_{(-\varepsilon,\infty)}(t) 
\, = \, \onebld_{[\,0,\,\infty)}(t), \epc \forall \, 
t \, \in \, \mathbb{R}$;

\noindent {\bf (iii)} for any scalar $t_* \neq 0$, a scalar
$\varepsilon_* > 0$ and a neighborhood ${\cal T}_*$ of $t_*$ exist
such that
\begin{equation} \label{eq:limit indicator}
\onebld_{( \, -\varepsilon,\infty )}(t) \, = \,
\onebld_{[ \, 0, \infty )}(t_*), \epc \forall \, ( t,\varepsilon ) 
\in {\cal T}_* \times ( 0,\varepsilon_* );
\end{equation}
{\bf (iv)} for any sequence $\{ t_k \} \to 0$, a sequence
$\{ \varepsilon_k \} \downarrow 0$ exists such that 
$\onebld_{( \, -\varepsilon_k,\infty )}(t_k) =
\onebld_{[ \, 0, \infty )}(t_k)$ for all $k$. \hfill $\Box$
\end{lemma}

Let $X_{\rm AHS}$ and $X_{\rm AHS}^{\varepsilon}$ denote the constrained 
set of the A-HSCOP problem (\ref{eq:A_HSCOP}) and 
$\left( \, \mbox{\bf P}_{\rm AHS}^{\, \varepsilon} \, \right) $ respectively. 
Clearly, we have the following monotonic properties for all 
$\varepsilon \, \in \, ( 0,\varepsilon^{\prime} )$:
\begin{equation} \label{eq:2 epsilons}
X_{\rm AHS}^{\, {\varepsilon^{\, \prime}}} \, \subseteq \, 
X_{\rm AHS}^{\varepsilon} \, \subseteq \, X_{\rm AHS} \epc \mbox{and} \epc
\theta(x) \, \geq \, \theta^{\, \varepsilon}(x) \, \geq \, 
\theta^{\, \varepsilon^{\prime}}(x), 
\end{equation}
It follows that
\begin{equation} \label{eq:equal suprema}
\sup_{x \in X_{\rm AHS}} \, \theta(x) \, \geq \, \displaystyle{
\lim_{\varepsilon \downarrow 0}
} \ \sup_{x \in X_{\rm AHS}^{\varepsilon}} \, 
\theta^{\, \varepsilon}(x) \, = \,
\displaystyle{
\sup_{\varepsilon > 0}
} \ \sup_{x \in X_{\rm AHS}^{\varepsilon}} \, 
\theta^{\, \varepsilon}(x).
\end{equation}
Without assuming attainment of the suprema, the following result 
shows that inequality in (\ref{eq:equal suprema}) must hold as an equality.  
The noteworthy point of the result 
is its validity under minimal assumptions.   

\begin{proposition} \label{pr:equal suprema} \rm
Suppose that $c$ and $\phi_{ij}$ are continuous for 
all $(i,j)$ and that $\displaystyle{
\sup_{x \in X_{\rm AHS}}
}\, \theta(x)$ is finite (in particular 
$X_{\rm AHS} \neq \emptyset$).
Then $$\sup_{x \in X_{\rm AHS}} \, \theta(x) = \displaystyle{
\lim_{\varepsilon \downarrow 0}
} \ \sup_{x \in X_{\rm AHS}^{\varepsilon}} \, 
\theta^{\, \varepsilon}(x) \, = \,
\displaystyle{
\sup_{\varepsilon > 0}
} \ \sup_{x \in X_{\rm AHS}^{\varepsilon}} \, 
\theta^{\, \varepsilon}(x).$$  Furthermore, if $\displaystyle{
\sup_{x \in X_{\rm AHS}}
} \, \theta(x)$ is attained, say at $x^*$, then 
there exists $\varepsilon_* > 0$ such that $x^*$ is 
a maximizer of $\theta^{\, \varepsilon}(x)$ on 
$X_{\rm AHS}^{\, \varepsilon}$ 
for all $\varepsilon \in ( \, 0, \varepsilon_* \, ]$.
\end{proposition}

\begin{proof}  
Let $\{ x^{\nu} \} \subset X_{\rm AHS}$ be 
a sequence such that $\displaystyle{
\lim_{\nu \to \infty}
} \, \theta(x^{\nu}) = \displaystyle{
\sup_{x \in X_{\rm AHS}}
} \, \theta(x)$.  For each $\nu$, there exists 
$\varepsilon_{\nu} > 0$
such that $\onebld_{[ \, 0,\infty )}(\phi_{ij}(x^{\nu})) = 
\onebld_{( \, -\varepsilon_{\nu},\infty )}(
\phi_{ij}(x^{\nu}))$ for all
$(i,j)$.  Thus 
$x^{\nu} \in X_{\rm AHS}^{\, \varepsilon_{\nu}}$.  
The sequence
$\{ \varepsilon_{\nu} \}$ can be chosen to converge to zero.  
Hence,
\[
\theta(x^{\nu}) \, = \, \theta^{\, \varepsilon_{\nu}}(x^\nu) 
\, \leq \, \displaystyle{
\sup_{x \in X_{\rm AHS}^{\, \varepsilon_{\nu}}} 
} \, \theta^{\, \varepsilon_{\nu}}(x)
\]
Since $\displaystyle{
\sup_{x \in X_{\rm AHS}^{\, \varepsilon}} 
} \, \theta^{\, \varepsilon}(x)$ is monotonically nonincreasing 
in $\varepsilon$, letting $\nu \to \infty$ yields 
\[
\displaystyle{
\sup_{x \in X_{\rm AHS}}
} \, \theta(x) \, \leq \, \displaystyle{
\limsup_{\nu \to \infty}
} \, \theta(x^{\nu}) \, = \, \displaystyle{
\lim_{\varepsilon \downarrow 0}
} \ \sup_{x \in X_{\rm AHS}^{\, \varepsilon}} \, 
\theta^{\, \varepsilon}(x) = \displaystyle{
\sup_{\varepsilon > 0}
} \ \sup_{x \in X_{\rm AHS}^{\varepsilon}} \, 
\theta^{\, \varepsilon}(x).
\]
Combining with \eqref{eq:equal suprema}, we thus obtain the desired equality.

To prove the last statement, we know that 
there exists $\varepsilon_* > 0$ such that 
\[
\onebld_{[ \, 0,\infty )}(\phi_{ij}(x^*)) = 
\onebld_{( \, -\varepsilon,\infty )}(\phi_{ij}(x^*)),
\epc \forall \, (i,j) \ \mbox{ and } \ \forall \,
\varepsilon \, \in \, ( 0,\varepsilon_* ].
\]
Thus, $x^* \in X_{\rm AHS}^{\, \varepsilon}$ for all such $\varepsilon$;
moreover, $\theta^{\, \varepsilon}(x^*) \, = \, \theta(x^*) 
\, \geq \, \theta(x) \geq 
\theta^{\, \varepsilon}(x)$ for all $ x \, \in \, X_{\rm AHS}$. 
Since $X_{\rm AHS}^{\, \varepsilon} \subseteq X_{\rm AHS}$, $x^*$ is thus a maximizer of $
\theta^{\, \varepsilon}(x)$ on $X_{\rm AHS}^{\, \varepsilon}$. 
\end{proof}




We note an immediate consequence of property (iii) in 
Lemma~\ref{lm:epsilon_approximation}.  Namely,
for a given vector $\bar{x} \in X_{\rm AHS}$, defining the index sets
for $i = 0, 1, \cdots, I$:
\begin{equation} \label{eq:index sets}
\begin{array}{l}
{\cal J}_{i,0}(\bar{x})   \triangleq \left\{ \, j   \mid  
\phi_{ij}(\bar{x})   =   0 \, \right\}, \ 
{\cal J}_{i,>}(\bar{x})  \triangleq   \left\{ \, j   \mid   
\phi_{ij}(\bar{x}) \, >   0 \,  \right\},  \mbox{and} \\ [0.1in]
{\cal J}_{i,0}^{\, -}(\bar{x})  \triangleq   
\left\{ \, j   \mid \psi_{ij}  <   0,  
\phi_{ij}(\bar{x})   =  0 \, \right\}.
\end{array}
\end{equation} 
we may deduce that there exists a neighborhood ${\cal N}$ of $\bar{x}$ 
and a scalar $\bar{\varepsilon} > 0$ (dependent on $\bar{x}$)
such that for all $(i,j)$ with $j \not\in {\cal J}_{i,0}^-(\bar{x})$ 
and $i = 0, 1, \cdots, I$,
\begin{equation} \label{eq:sign invariant}
\psi_{ij}^- \, \onebld_{( \, -\varepsilon,\infty )}( \phi_{ij}(x) ) 
\, = \, \psi_{ij}^- \, \onebld_{[ \, 0,\infty )}( \phi_{ij}(x) ) 
\, = \, 
\psi_{ij}^- \, \onebld_{[ \, 0,\infty )}( \phi_{ij}(\bar{x}) )
\end{equation}
for all $\varepsilon \in ( \, 0, \bar{\varepsilon} \, ]$ and all
$x \in {\cal N}$.  We say that $\bar{x}$ satisfies the 
\emph{local sign-invariance} (LSI) condition if
$\phi_{ij}$ is nonnegative near 
$\bar{x}$ for all $j \in {\cal J}_{i,0}^{\, -}(\bar{x})$ and 
all $i = 0, 1, \cdots, I$.
Below we show that the $\varepsilon$-approximating problem 
$\left( \, \mbox{\bf P}_{\rm AHS}^{\, \varepsilon} \, \right)$ 
is equivalent to 
the A-HSCOP problem \eqref{eq:A_HSCOP} in terms of local maximizer 
under this condition.

\begin{proposition} \label{pr:N&S II locmax AHSCOP} \rm
Let $\bar{x} \in X$ given.  Let $\bar{\varepsilon} > 0$ be such that 
(\ref{eq:sign invariant}) holds in the neighborhood of $\bar x$ for all 
$\varepsilon \in ( \, 0, \bar{\varepsilon} \, ]$ and all 
$i = 0, 1, \cdots, I$. The following two statements hold:

\noindent {\bf (A)}  
If $\bar{x}$ is a local maximizer of (\ref{eq:A_HSCOP}), 
then $\bar{x}$ is a local maximizer of 
$\left( \, \mbox{\bf P}_{\rm AHS}^{\, \varepsilon} \, \right)$ 
for every $\varepsilon \in ( \, 0,\bar{\varepsilon} \, ]$. 

\noindent {\bf (B)}  Conversely, if $\bar{x}$ is a local maximizer of 
$\left( \, \mbox{\bf P}_{\rm AHS}^{\, \varepsilon} \, \right)$ satisfying
the local sign-invariance condition, then $\bar{x}$ is a local maximizer 
of (\ref{eq:A_HSCOP}).
\end{proposition}

\begin{proof}
%
(A) Suppose that $\bar{x}$ is a local maximizer of 
(\ref{eq:A_HSCOP}).   Let ${\cal N}$ be a neighborhood 
of $\bar{x}$ within which $\bar{x}$ is optimal and also such 
that (\ref{eq:sign invariant}) holds.
Let $x$ be a feasible solution of
$\left( \, \mbox{\bf P}_{\rm AHS}^{\, \varepsilon} \, \right)$
for an arbitrary $\varepsilon \in ( \, 0,\bar{\varepsilon} \, ]$.  
Then, for all $i = 1, \cdots, I$,
\[ \begin{array}{lll}
A_{\, i \bullet} x + \displaystyle{
\sum_{j=1}^{J_i}
} \, \psi_{ij}^+ \, \onebld_{[ \, 0,\infty )}( \phi_{ij}(x) ) 
& \geq & \displaystyle{
\sum_{j=1}^{J_i}
} \, \psi_{ij}^- \,
\onebld_{( \, -\varepsilon,\infty )}( \phi_{ij}(x) ) + \eta_i 
\\ [0.1in]
& \geq & \displaystyle{
\sum_{j=1}^{J_i}
} \, \psi_{ij}^- \onebld_{[ \, 0,\infty )}( \phi_{ij}(x))
+ \eta_i.
\end{array}
\]
Thus $x$ is feasible to (\ref{eq:A_HSCOP}).  Hence, if 
$x \in {\cal N}$ additionally, then
\[ \begin{array}{l}
c(\bar{x}) + \displaystyle{
\sum_{j=1}^{J_0}
} \, \psi_{0j}^+ \, \onebld_{[ \, 0,\infty )}( \phi_{0j}(\bar{x}) ) 
- \displaystyle{
\sum_{j=1}^{J_0}
} \, \psi_{0j}^- \,
\onebld_{( \, -\varepsilon,\infty )}( \phi_{0j}(\bar{x}) ) \\ [0.1in]
\epc = \, c(\bar{x}) + \displaystyle{
\sum_{j=1}^{J_0}
} \, \psi_{0j} \, \onebld_{[ \, 0,\infty )}( \phi_{0j}(\bar{x}) ) 
\hspace{0.2in} \mbox{by the choice of $\varepsilon$ in (\ref{eq:sign invariant})}
\\ [0.1in]
\epc \geq \, c(x) + \displaystyle{
\sum_{j=1}^{J_0}
} \, \psi_{0j} \, \onebld_{[ \, 0,\infty )}( \phi_{0j}(x) ) 
\hspace{0.2in} \mbox{since 
$\bar{x}$ is a local maximizer of (\ref{eq:A_HSCOP})} \\ [0.1in]
\epc \geq \, c(x) + \displaystyle{
\sum_{j=1}^{J_0}
} \, \psi_{0j}^+ \, \onebld_{[ \, 0,\infty )}( \phi_{0j}(x) ) 
- \displaystyle{
\sum_{j=1}^{J_0}
} \, \psi_{0j}^- \,
\onebld_{( \, -\varepsilon,\infty )}( \phi_{0j}(x) ),
\end{array}
\]
where the last inequality holds by  (\ref{eq:sign invariant}) again,
showing that $\bar{x}$ is the local maximizer for 
$\left( \, \mbox{\bf P}_{\rm AHS}^{\, \varepsilon} \, \right)$.

\gap

\noindent (B)  Without loss of generality, we may assume that 
$\bar{x}$ is optimal for
$\left( \, \mbox{\bf P}_{\rm AHS}^{\, \varepsilon} \, \right)$ 
in ${\cal N}$.  
We claim that $\bar{x}$ is a maximizer of (\ref{eq:A_HSCOP}) within 
this neighborhood by first showing that 
$X_{\rm AHS}^{\, \varepsilon} \cap {\cal N} = X_{\rm AHS} \cap {\cal N}$ 
under the given assumptions.  It suffices to show the inclusion 
$X_{\rm AHS} \cap {\cal N} \subseteq X_{\rm AHS}^{\, \varepsilon} \cap {\cal N}$.
For this purpose, let $x \in X_{\rm AHS} \cap {\cal N}$; we then have,
for $i = 0, 1, \cdots, I$,
\begin{equation} \label{eq:feasibility equivalence} 
\begin{array}{ll}
 \displaystyle{
\sum_{j=1}^{J_i}
} \, \psi_{ij}^- \, \onebld_{[ \, 0,\infty )}( \phi_{ij}(x) )   & = \,  \displaystyle{
\sum_{j \in {\cal J}_{i,0}^-(\bar{x})}
} \, \psi_{ij}^- \onebld_{[ \, 0,\infty )}( \phi_{ij}(x)) +
\displaystyle{
\sum_{j \not\in {\cal J}_{i,0}^-(\bar{x})}
} \, \psi_{ij}^- \onebld_{[ \, 0,\infty )}( \phi_{ij}(x))
\\ [0.3in]
& = \, \underbrace{\displaystyle{
\sum_{j \in {\cal J}_{i,0}^-(\bar{x})}
} \, \psi_{ij}^- \onebld_{( \, -\varepsilon,\infty )}( 
\phi_{ij}(x))}_{\mbox{
by LSI condition}} + \underbrace{\displaystyle{
\sum_{j \not\in {\cal J}_{i,0}^-(\bar{x})}
} \, \psi_{ij}^- \onebld_{( \, -\varepsilon,\infty )}( 
\phi_{ij}(x))}_{\mbox{ by  } \eqref{eq:sign invariant} }\\ [0.1in]
& = \,  \displaystyle{
\sum_{j=1}^{J_i}
} \, \psi_{ij}^- \onebld_{
( \, -\varepsilon,\infty )}( \phi_{ij}(x) ),
\end{array} \end{equation}
proving the equality $X_{\rm AHS}^{\varepsilon} \cap {\cal N} = 
X_{\rm AHS} \cap {\cal N}$. Thus we can show the optimality as follows, 
\[ \begin{array}{l}
c(\bar{x}) + \displaystyle{
\sum_{j=1}^{J_0}
} \, \psi_{0j} \, \onebld_{[ \, 0,\infty )}( \phi_{0j}(\bar{x}) ) 
\\ [0.1in]
\epc \geq \, c(\bar{x}) + \displaystyle{
\sum_{j=1}^{J_0}
} \, \psi_{0j}^+ \, \onebld_{[ \, 0,\infty )}( \phi_{0j}(\bar{x}) ) 
- \displaystyle{
\sum_{j=1}^{J_0}
} \, \psi_{0j}^- \,
\onebld_{( \, -\varepsilon,\infty )}( \phi_{0j}(\bar{x}) ) 
\\ [0.1in]
\epc \geq \, c(x) + \displaystyle{
\sum_{j=1}^{J_0}
} \, \psi_{0j}^+ \, \onebld_{[ \, 0,\infty )}( \phi_{0j}(x) ) - 
\displaystyle{
\sum_{j=1}^{J_0}
} \, \psi_{0j}^- \,
\onebld_{( \, -\varepsilon,\infty )}( \phi_{0j}(x) ) \quad \mbox{ for any } x \in  X_{\rm AHS}^{\varepsilon} \cap {\cal N}
\\ [0.05in]
\epc = c(x) + \displaystyle{
\sum_{j=1}^{J_0}
} \, \psi_{0j} \onebld_{( \, -\varepsilon,\infty )}( \phi_{0j}(x)) 
\epc \mbox{based on  
(\ref{eq:feasibility equivalence})},
\end{array}
\]
proving the desired optimality of $\bar{x}$ for 
(\ref{eq:A_HSCOP}) in $X_{\rm AHS} \cap {\cal N}$.
\end{proof}

\section{The Decomposed Equivalent: Piecewise Affine Decomposition}
\label{sec:PA decomp}

Although the non-upper semicontinuity of the A-HSCOP is satisfactorily
addressed by the approximation of the closed Heaviside function 
$\onebld_{[ \, 0,\infty )}$ by the open
Heaviside function $\onebld_{( \, -\varepsilon,\infty )}$, due to the 
nonconvexity/nonconcavity of the inner piecewise
affine functions $\{\phi_{ij}\}$, it is still hard to solve the 
$\varepsilon$-approximating problem efficiently.  To facilitate the 
computation, we 
exploit the piecewise structure of these functions
and introduce convex/concave surrogates 
%
whereby we lower and upper bound the convex/concave PA functions 
$\phi_{ij}^{\rm cvx}$ and $\phi_{ij}^{\rm cve}$ by one of
its maximizing/minimizing affine pieces at a current feasible vector 
$\bar{x} \in X_{\rm AHS}$. 
Toward this end, we define the following index sets for a 
given $\bar{x} \in X_{\rm AHS}$ and a scalar 
$\delta \geq 0$
\[ \begin{array}{lll}
{\cal K}_{ij}^{\delta}(\bar{x}) & \triangleq &
\left\{ \, k \, \in \, [ K_{ij} ] \, \mid \,
( a_{ij}^k )^{\top} \bar{x} + \alpha_{ij}^k \, \geq \, 
\phi_{ij}^{\rm cvx}(\bar{x}) - \delta \, \right\} \\ [0.15in]
{\cal L}_{ij}^{\delta}(\bar{x}) & \triangleq &
\left\{ \, \ell \, \in \, [ L_{ij} ] \, \mid \,
( b_{ij}^{\, \ell} )^{\top} \bar{x} + \beta_{ij}^{\, \ell} 
\, \leq \, 
\phi_{ij}^{\rm cve}(\bar{x}) + \delta \, \right\};
\end{array} \]
we write ${\cal K}_{ij}(\bar{x})$ and ${\cal L}_{ij}(\bar{x})$ for
${\cal K}_{ij}^0(\bar{x})$ and ${\cal L}_{ij}^0(\bar{x})$,
respectively; these are the maximizing and minimizing index sets of 
$\phi_{ij}^{\rm cvx}(\bar{x})$ and 
$\phi_{ij}^{\rm cve}(\bar{x})$, respectively.  We let

\[ \begin{array}{lll}
{\cal M}^{\boldsymbol{\delta}}(\bar{x}) & \triangleq & \left\{ \, 
( \mathbf{k},\boldsymbol{\ell} ) \, \triangleq \, \left\{ \, 
\left\{ \, ( k_{ij},\ell_{ij} ) \in 
{\cal K}_{ij}^{\, \delta_{ij}}(\bar{x})
\times {\cal L}_{ij}^{\,\delta_{ij}}(\bar{x}) 
\, \right\}_{j=1}^{J_i} 
\, \right\}_{i=0}^I \, \right\} \\ [0.25in]
& = & \displaystyle{
\prod_{i=0}^I
} \ \displaystyle{
\prod_{j=1}^{J_i}
} \, \left[ \ {\cal K}_{ij}^{\, \delta_{ij}}(\bar{x})
\times {\cal L}_{ij}^{\, \delta_{ij}}(\bar{x}) \, \right] 
\, \supseteq \, \displaystyle{
\prod_{i=0}^I
} \ \displaystyle{
\prod_{j=1}^{J_i}
} \, \left[ \ {\cal K}_{ij}(\bar{x})
\times {\cal L}_{ij}(\bar{x}) \, \right] \, \triangleq \, 
{\cal M}(\bar{x}).
\end{array}
\]
The idea of using a positive scalar $\delta$
to enlarge the two 
families of maximizing indices is derived from the enhancement of the
difference-of-convex methodology \cite{PangRazaAlvarado17}.  
An important role of the scalar $\delta$ is the following simple fact
(see also \cite[Lemma~6.1.2]{CuiPang2021}).

\begin{lemma} \label{lm:simple fact} \rm
For any tuple ${\boldsymbol{\delta}} \triangleq \left\{\left\{\delta_{ij}\right\}_{j=1}^{J_i}\right\}_{i=0}^{I}$ 
with each scalar $\delta_{ij}>0$,
there exists a neighborhood ${\cal N}$ of $\bar{x}$ such that
${\cal M}(x) \subseteq {\cal M}(\bar{x}) \subseteq 
{\cal M}^{\boldsymbol{\delta}}(x')$ for all $x, x' \in {\cal N}$. 
Moreover, there exists $\bar{\delta} > 0$ such that 
\[
{\cal K}_{ij}(\bar{x}) \, \times \, {\cal L}_{ij}(\bar{x}) \, = \,
{\cal K}_{ij}^{\, \delta_{ij}}(\bar{x}) \, \times \, 
{\cal L}_{ij}^{\, \delta_{ij}}(\bar{x}),
\epc \forall \, \delta_{ij} \in \left[ \, 0, \bar{\delta} \, \right]
\]
for all $(i,j)$. Thus the two families 
${\cal M}^{\boldsymbol{\delta}}(\bar{x})$ 
and ${\cal M}(\bar{x})$ are equal for all ${\boldsymbol{\delta}}$ with 
each $\delta_{ij} > 0$ sufficiently small.  \hfill $\Box$
\end{lemma}

For an arbitrary pair of indices $( k,\ell )$, define the
piecewise affine functions:
\begin{equation} \label{eq:conv/conv phi} 
\begin{array}{l}
\phi_{ij}^{k;+}(x) \, \triangleq \,  ( a_{ij}^{\, k} )^{\top}x + 
\alpha_{ij}^{\, k} + \displaystyle{
\min_{1 \leq \ell^{\, \prime} \leq L_{ij}}
} \, \left[ \, \left( b_{ij}^{\, \ell^{\, \prime}} \right)^{\top}
x + \beta_{ij}^{\, \ell^{\, \prime}} \, \right] \\ [0.2in]
\phi_{ij}^{\, \ell;-}(x) \, \triangleq \, 
\displaystyle{
\max_{1 \leq k^{\, \prime} \leq K_{ij}}
} \, \left[ \, \left( a_{ij}^{\, k^{\, \prime}} \right)^{\top}x + 
\alpha_{ij}^{\, k^{\, \prime}} \, \right] + 
( b_{ij}^{\, \ell} )^{\top}x + \beta_{ij}^{\, \ell}
\end{array}
\end{equation}
with the former one being concave and the latter one being convex.  
We note the following bounds:
\begin{equation} \label{eq:bounding fncs}
\phi_{ij}^{\ell;-}(x) \, \geq \, \phi_{ij}(x) \, \geq \, 
\phi_{ij}^{k;+}(x) \epc \mbox{for any pair of indices $(k,\ell)$ 
and any $x$},
\end{equation}
with equalities holding 
if $(k,\ell) \in {\cal K}_{ij}(x) \times {\cal L}_{ij}(x)$.
For each tuple $( \mathbf{k},\boldsymbol{\ell} ) \in 
{\cal M}^{\boldsymbol{\delta}}(\bar{x})$, define the decomposed 
$\varepsilon$-approximating problem
$\left( \, \mbox{\bf P}_{{\rm AHS}}^{\, \varepsilon;
\mathbf{k},\boldsymbol{\ell}} \, \right)$:
\begin{equation} \label{eq:approximated II AHSCOP}
\underbrace{ \begin{array}{ll}
\displaystyle{
\operatornamewithlimits{\mbox{\bf maximize}}_{x \in P}
} \ & \theta^{\, \varepsilon;\mathbf{k}_0,\boldsymbol{\ell}_0}(x) 
\, \triangleq \, c(x) \, + \displaystyle{
\sum_{j=1}^{J_0}
} \, \psi_{0j}^+ \, \onebld_{[ \, 0,\infty )}\left( 
\phi_{0j}^{\, k_{0j};+}(x) \right) - \displaystyle{
\sum_{j=1}^{J_0}
} \, \psi_{0j}^- \,
\onebld_{( \, -\varepsilon,\infty )}\left( 
\phi_{0j}^{\, \ell_{0j};-}(x) \right) \\ [0.3in]
\mbox{\bf subject to} \ &
A_{\, i \bullet} \, x + \displaystyle{
\sum_{j=1}^{J_i}
} \, \psi_{ij}^+ \, \onebld_{[ \, 0,\infty )}\left( 
\phi_{ij}^{\, k_{ij};+}(x) \right) -  \displaystyle{
\sum_{j=1}^{J_i}
} \, \psi_{ij}^- \, \onebld_{( \, -\varepsilon,\infty )}\left( 
\phi_{ij}^{\, \ell_{ij};-}(x) \right) \geq \eta_i, \, \forall i \in [ \, I \, ].
\end{array}  }_{\mbox{feasible set denoted
$X_{\rm AHS}^{\, \varepsilon;\mathbf{k},\boldsymbol{\ell}}$}}
\end{equation}
Extending the inequalities (\ref{eq:2 epsilons}), 
we obtain following bounds which connect the objective and feasible set of  
$\left( P_{\rm AHS}^{\, \varepsilon;\mathbf{k},\boldsymbol{\ell}} \right)$ 
with  (\ref{eq:A_HSCOP}). Specifically, for all tuples 
$( \mathbf{k},\boldsymbol{\ell} )=\left\{ \, \left\{ \, ( k_{ij},{\ell}_{ij} ) 
\, \right\}_{j=1}^{J_i} \, \right\}_{i=0}^I$: 
\begin{equation} \label{eq:minorization}
\begin{array}{ll}
\theta(x) \, \geq \, \theta^{\, \varepsilon}(x) 
\, \geq \, 
\theta^{\, \varepsilon;
\mathbf{k}_0,\boldsymbol{\ell}_0}(x) 
\epc \mbox{for any $x$}, \\[0.1in]
X_{\rm AHS}^{\, \varepsilon^{\, \prime}; 
\mathbf{k},\boldsymbol{\ell}} \subseteq
X_{\rm AHS}^{\, \varepsilon;
\mathbf{k},\boldsymbol{\ell}} \subseteq X_{\rm AHS}^{\varepsilon} 
\subseteq X_{\rm AHS}, \quad 
\forall  \varepsilon^{\prime} > \varepsilon > 0. 
\end{array}
\end{equation}
By the simple fact that 
\[ \begin{array}{l}
\displaystyle{
\sum_{j=1}^{J_i}
} \, \psi_{ij}^+ \, \onebld_{[ \, 0,\infty )}( \phi_{ij}(x) ) - 
\displaystyle{
\sum_{j=1}^{J_i}
} \, \psi_{ij}^- \,
\onebld_{( \, -\varepsilon,\infty )}( \phi_{ij}(x) ) \\[0.1in]
\epc = \, \displaystyle{
\sum_{j=1}^{J_i}
} \, \max_{1 \leq k \leq K_{ij}} \left\{ \psi_{ij}^+ \, 
\onebld_{[ \, 0,\infty )}\left( 
\phi_{ij}^{\, k;+}(x) \right) \right\}  - \displaystyle{
\sum_{j=1}^{J_i}
} \, \min_{1 \leq \ell \leq L_{ij} } \left\{ \, \psi_{ij}^- \,
\onebld_{( \, -\varepsilon,\infty )}\left( 
\phi_{ij}^{\, \ell;-}(x)  \right)  \right\}, 
\end{array}
\]
we obtain that $X_{\rm AHS}^{\, \varepsilon} = \displaystyle{
\bigcup_{\mathbf{k},\boldsymbol{\ell}}
} \, X_{{\rm AHS}}^{\, \varepsilon;
\mathbf{k},\boldsymbol{\ell}}$. 

With $P$ being a compact set,  the optimal objective value of 
(\ref{eq:approximated II AHSCOP}), which we denote 
$\theta_{\rm opt}^{\, \varepsilon; 
\mathbf{k},\boldsymbol{\ell}}$, 
is finite if the problem is feasible.
Let
\[ \begin{array}{lll}
\theta_{\rm \max}^{\, \varepsilon;\boldsymbol{\delta}}(\bar{x}) 
& \triangleq &
{\bf maximum}\left\{ \, \theta_{\rm opt}^{\, \varepsilon;
\mathbf{k},\boldsymbol{\ell}} \, \mid \, 
( \mathbf{k},\boldsymbol{\ell} ) \in 
{\cal M}^{\boldsymbol{\delta}}(\bar{x}) \, \right\} \\ [0.1in]
& \geq & {\bf maximum}\left\{ \, 
\theta_{\rm opt}^{\, \varepsilon;
\mathbf{k},\boldsymbol{\ell}} \, \mid \, 
( \mathbf{k},\boldsymbol{\ell} ) \in {\cal M}(\bar{x}) \, \right\}, 
\epc \mbox{since ${\cal M}(\bar{x}) \subseteq 
{\cal M}^{\boldsymbol{\delta}}(\bar{x})$}.
\end{array} \]
If $\bar{x}$ is a global maximizer of 
(\ref{eq:A_HSCOP}), then by (\ref{eq:minorization}), we have for all 
index tuples $( \mathbf{k}_0,\boldsymbol{\ell}_0 )$ and $x \in X_{\rm AHS}$,
\[
\theta(\bar{x}) \, \geq \, \theta(x) \, \geq \, 
\theta^{\, \varepsilon;
\mathbf{k}_0,\boldsymbol{\ell}_0}(x). 
\]  
Taking $x$ over all feasible vectors of 
$\left( \, \mbox{\bf P}_{\rm AHS}^{\, \varepsilon; 
\mathbf{k},\boldsymbol{\ell}} \, \right)$ readily yields $
\theta(\bar{x}) \, \geq \, 
\theta_{\rm \max}^{\, \varepsilon;\boldsymbol{\delta}}(\bar{x})$ 
for all $\varepsilon \, > \, 0$ and $\delta \, \geq \, 0$.

We will employ a 
regularization of the problem (\ref{eq:approximated II AHSCOP}) 
by a scalar $\rho > 0$ without changing
the feasible set 
$X_{\rm AHS}^{\, \varepsilon;\mathbf{k},\boldsymbol{\ell}}$,
resulting in the following  decomposed $\varepsilon$-approximating problem, which we denote
$\left( \, \mbox{\bf P}_{\rho}^{\, \varepsilon; 
\mathbf{k},\boldsymbol{\ell}}(\bar{x}) \, \right)$:
\begin{equation} \label{eq:approximated II AHSCOP regularized}
\begin{array}{l}
\displaystyle{
\operatornamewithlimits{\mbox{\bf maximize}}_{x \in X_{\rm AHS}^{\varepsilon; \mathbf{k}, \boldsymbol{\ell}}}
} \ \theta_{\rho}^{\, \varepsilon;
\mathbf{k}_0,\boldsymbol{\ell}_0}(x;\bar{x}) 
\, \triangleq \,  \theta^{\, \varepsilon;\mathbf{k}_0,\boldsymbol{\ell}_0}(x)
- \displaystyle{
\frac{\rho}{2}
} \, \| x - \bar{x} \|_2^2  
\end{array}
\end{equation}
Since $\onebld_{( \, -\varepsilon,\infty )}(t) = 
1 - \onebld_{[ \, 0,\infty )}(-t - \varepsilon)$ for all 
$t \in \mathbb{R}$,
we obtain the following equivalent closed Heaviside formulation of the 
objective and constraint functions in the above problem 
\begin{equation} \label{eq:usc approximated II AHSCOP regularized}
\begin{array}{ll}
\theta_{\rho}^{\, \varepsilon;
\mathbf{k}_0,\boldsymbol{\ell}_0}(x;\bar{x}) 
\, =  \,& c(x) - \displaystyle{
\frac{\rho}{2}
} \, \| \, x - \bar{x} \, \|_2^2 - {\displaystyle{
\sum_{j=1}^{J_0}
} \, \psi_{0j}^-} \, + \\
& \displaystyle{
\sum_{j=1}^{J_0}
} \, \psi_{0j}^+ \, \onebld_{[\, 0,\infty )}\left( 
\phi_{0j}^{\, k_{0j};+}(x) \right) + \displaystyle{
\sum_{j=1}^{J_0}
} \, \psi_{0j}^- \,
\onebld_{[ \, 0,\infty )}\left( 
-\phi_{0j}^{\, \ell_{0j};-}(x) - \varepsilon \right).
\end{array}
\end{equation}

Proposition~\ref{pr:N&S III locmax AHSCOP} below establishes the
necessary and sufficient 
conditions for a feasible vector of (\ref{eq:A_HSCOP}) 
to be its global/local maximizer.  The first part of the proposition 
below establishes necessary 
conditions for a vector $\bar{x}$ to be a global/local maximizer 
of (\ref{eq:A_HSCOP}), 
respectively, in terms of the problems 
$\left( \, \mbox{\bf P}_{\rm AHS}^{\, \varepsilon; 
\mathbf{k},\boldsymbol{\ell}} \, \right)$, whereas the second provides 
sufficient conditions for the local optimality of $\bar{x}$ in 
terms of the local optimality of the decomposed
problems (\ref{eq:approximated II AHSCOP}) under the LSI condition of 
the vector $\bar{x}$ of interest; this part is the same as part (B) of
Proposition~\ref{pr:N&S II locmax AHSCOP}.  The third part shows the 
equivalence between the regularized HSCOP and original HSCOP in terms 
of local maximizers. 
It is noteworthy that the equivalence holds when the regularization is 
applied to a non-semicontinuous optimization problem, in contrast to the 
classical equivalence for convex programs.  

\begin{proposition} \label{pr:N&S III locmax AHSCOP} \rm
For any arbitrary $\bar{x} \in X_{\rm AHS}$, the following three 
statements hold.

\noindent {\bf (A)} If $\bar{x}$ is a local maximizer of 
(\ref{eq:A_HSCOP}), then there
exists $\bar{\varepsilon} > 0$ (dependent on $\bar{x}$) such that
for all 
$\varepsilon \in \left( \, 0, \bar{\varepsilon} \, \right]$,
$\bar{x}$ is a local maximizer of 
$\left( \, \mbox{\bf P}_{\rm AHS}^{\, \varepsilon; 
\mathbf{k},\boldsymbol{\ell}} \, \right)$ for all tuples 
$( \mathbf{k},\boldsymbol{\ell} ) \in  {\cal M}(\bar{x})$.

\noindent {\bf (B)} Conversely,  if   $\bar{x}$ is a local maximizer of
$\left( \, \mbox{\bf P}_{\rm AHS}^{\, \varepsilon; 
\mathbf{k},\boldsymbol{\ell}} \, \right)$ for all tuples 
$( \mathbf{k},\boldsymbol{\ell} ) \in  {\cal M}(\bar{x})$, suppose that 
$\bar x$ satisfies the LSI condition, then $\bar{x}$ is a local maximizer 
of (\ref{eq:A_HSCOP}).

\noindent {\bf (C)}  $\bar x$ is a local maximizer to problem 
\eqref{eq:A_HSCOP} if and only if $\bar x$ is a local maximizer to the 
regularized problem: $\underset{x\in X_{\rm AHS}}{\mbox{\bf maximize}}
~\displaystyle{\theta_{\rm AHS}(x) - \frac{\rho}{2}\|x- \bar x\|_2^2}$. 
\end{proposition}

\begin{proof}  (A) There exists $\bar{\varepsilon} > 0$ such that
$\onebld_{( \, -\varepsilon,\infty )}(\phi_{ij}(\bar{x})) = 
\onebld_{[ \, 0,\infty )}(\phi_{ij}(\bar{x}))$ for all 
$\varepsilon \in ( 0,\bar{\varepsilon} ]$ and all $(i,j)$. 
[This is (\ref{eq:limit indicator}) if $\phi_{ij}(\bar{x}) \neq 0$; 
the equality clearly holds if $\phi_{ij}(\bar{x}) = 0$.]  
Hence, since $\bar{x}$ is feasible
to (\ref{eq:A_HSCOP}), it follows that $\bar{x}$ must
be feasible to  $\left( \, \mbox{\bf P}_{\rm AHS}^{\, \varepsilon;
\mathbf{k},\boldsymbol{\ell}} \, \right)$ for all pairs 
$( \mathbf{k},\boldsymbol{\ell} ) \in {\cal M}(\bar{x})$. 
For any pair of indices
$( k_{0j},\ell_{\, 0j} ) \in {\cal K}_{0j}(\bar{x}) \times 
{\cal L}_{0j}(\bar{x})$ for $j = 1, \cdots, J_0$, we have
$\phi_{0j}^{\, k_{0j};+}(\bar{x}) = 
\phi_{0j}^{\, \ell_{0j};-}(\bar{x}) = \phi_{0j}(\bar{x})$.   
Since any $x$ feasible to 
$\left( \, \mbox{\bf P}_{\rm AHS}^{\, \varepsilon; 
\mathbf{k},\boldsymbol{\ell}}\, \right)$ must be feasible 
to (\ref{eq:A_HSCOP}), thus if such an
$x$ is sufficiently close to $\bar{x}$, then
\[ \begin{array}{l}
c(\bar{x}) +
\displaystyle{
\sum_{j=1}^{J_0}
} \, \psi_{0j}^+ \, \onebld_{[ \, 0,\infty )}\left( 
\phi_{0j}^{\, k_{0j};+}(\bar{x}) \right) - \displaystyle{
\sum_{j=1}^{J_0}
} \, \psi_{0j}^- \,
\onebld_{( \, -\varepsilon,\infty )}\left( 
\phi_{0j}^{\, \ell_{0j};-}(\bar{x}) \right) \\ [0.15in]
\epc = \, c(\bar{x}) + \displaystyle{
\sum_{j=1}^{J_0}
} \, \psi_{0j} \, \onebld_{[ \, 0,\infty )}( \phi_{0j}(\bar{x}) ) 
\epc \mbox{by the above remarks} \\ [0.15in]
\epc \geq \, c(x) + \displaystyle{
\sum_{j=1}^{J_0}
} \, \psi_{0j} \, \onebld_{[ \, 0,\infty )}( \phi_{0j}(x) ) 
\hspace{0.15in} \mbox{since 
$\bar{x}$ is a local maximizer of (\ref{eq:A_HSCOP})} 
\\ [0.15in]
\epc \geq \, c(x) + \displaystyle{
\sum_{j=1}^{J_0}
} \, \psi_{0j}^+ \, \onebld_{[ \, 0,\infty )}( 
\phi_{0j}^{k_{0j};+}(x) ) - \displaystyle{
\sum_{j=1}^{J_0}
} \, \psi_{0j}^- \,
\onebld_{( \, -\varepsilon,\infty )}( \phi_{0j}^{\ell_{0j};-}(x) )
\epc \mbox{by (\ref{eq:minorization}).}
\end{array}
\]
This establishes (A).

\noindent (B)  Since ${\cal M}(\bar{x})$ is a finite set,
we may assume without loss of generality that $\bar{x}$ is 
a maximizer of
$\left( \, \mbox{\bf P}_{\rm AHS}^{\, \varepsilon; 
\mathbf{k},\boldsymbol{\ell}} \, \right)$ within
the neighborhood ${\cal N}$ for all tuples 
$( \mathbf{k},\boldsymbol{\ell} ) \in  {\cal M}(\bar{x})$.
We may further assume that ${\cal N}$ is such that 
Lemma~\ref{lm:simple fact} is valid.  Following the same proof for part (B) 
of Proposition~\ref{pr:N&S II locmax AHSCOP}, under  LSI condition,  we 
obtain that  
$X_{\rm AHS} \cap {\cal N} = X_{\rm AHS}^{\varepsilon} \cap {\cal N}$. 
Let $x \in X_{\rm AHS} \cap {\cal N}$.
Let the tuple $( \mathbf{k},\boldsymbol{\ell} ) \in {\cal M}(x)$ 
be such that
$\phi_{ij}(x) = \phi_{ij}^{k_{ij};+}(x) = 
\phi_{ij}^{\ell_{ij};-}(x)$ for all
pairs $(i,j)$. Thus $x \in X_{\rm AHS}^{\, \varepsilon; 
\mathbf{k},\boldsymbol{\ell}}  \cap \cal N$. By Lemma~\ref{lm:simple fact}, 
this pair $( \mathbf{k},\boldsymbol{\ell} ) \in {\cal M}(\bar{x})$.    Since
$\bar{x}$ is optimal for 
$\left( \, \mbox{\bf P}_{\rm AHS}^{\, \varepsilon; 
\mathbf{k},\boldsymbol{\ell}} \, \right)$, we have
\[ \begin{array}{l}
c(\bar{x}) + \displaystyle{
\sum_{j=1}^{J_0}
} \, \psi_{0j} \, \onebld_{[ \, 0,\infty )}\left( 
\phi_{0j}(\bar{x}) \right) \\ [0.2in]
\epc \geq \, c(\bar{x}) + \displaystyle{
\sum_{j=1}^{J_0}
} \, \psi_{0j}^+ \, \onebld_{[ \, 0,\infty )}\left( 
\phi_{0j}(\bar{x}) \right) - \displaystyle{
\sum_{j=1}^{J_0}
} \, \psi_{0j}^- \,\onebld_{( \, -\varepsilon,\infty )}\left( 
\phi_{0j}(\bar{x}) \right) \\ [0.2in]
\epc = \, c(\bar{x}) + \displaystyle{
\sum_{j=1}^{J_0}
} \, \psi_{0j}^+ \, \onebld_{[ \, 0,\infty )}\left( 
\phi_{0j}^{\, k_{0j};+}(\bar{x}) \right) - \displaystyle{
\sum_{j=1}^{J_0}
} \, \psi_{0j}^- \,\onebld_{( \, -\varepsilon,\infty )}\left( 
\phi_{0j}^{\, \ell_{0j};-}(\bar{x}) \right) \\ [0.2in]
\epc \geq \, c(x) + \displaystyle{
\sum_{j=1}^{J_0}
} \, \psi_{0j}^+ \, \onebld_{[ \, 0,\infty )}\left( 
\phi_{0j}^{\, k_{0j};+}(x) \right) - \displaystyle{
\sum_{j=1}^{J_0}
} \, \psi_{0j}^- \,\onebld_{( \, -\varepsilon,\infty )}\left( 
\phi_{0j}^{\, \ell_{0j};-}(x) \right),\\ [0.2in]
\epc = \, c(x) + \displaystyle{
\sum_{j=1}^{J_0}
} \, \psi_{0j} \, \onebld_{[ \, 0,\infty )}\left( 
\phi_{0j}(x) \right), 
\end{array}
\]
proving the desired optimality of $\bar{x}$ 
for (\ref{eq:A_HSCOP}) in 
$X_{\rm AHS} \cap {\cal N}$.

\noindent \textbf{(C)} In turn,
to prove the local equivalence, it suffices to show 
the ``if'' implication.
Suppose that ${\cal N}_{\bar{x}}$ be a neighborhood of $\bar{x}$ within
which $\bar{x}$ maximizes $\theta_{\rm AHS}(x) -\displaystyle{
\frac{\rho}{2}
} \, \| \, x - \bar{x} \, \|_2^2$ for all 
$x \in X_{\rm AHS} \cap {\cal N}_{\bar{x}}$.  Then the pair
\[ \begin{array}{lll}
( \, \bar{x},\theta_{\rm AHS}(\bar{x}) \, ) \, \in & 
\displaystyle{
\operatornamewithlimits{\mbox{\bf argmax}}_{
(x,t) \in \mathbb{R}^{n+1}}
} & t - \displaystyle{
\frac{\rho}{2}
} \, \| \, x - \bar{x} \, \|_2^2 \\ [0.2in] 
& \mbox{\bf subject to} &
(x,t) \, \in \, \mbox{hypo}(\theta_{\rm AHS}) \cap 
\left[ \, ( X_{\rm AHS} \cap {\cal N}_{\bar{x}} ) \, \times \, \mathbb{R} 
\, \right],
\end{array} \]
where $\mbox{hypo}(\theta_{\rm AHS})$ denotes the hypograph of the function
$\theta_{\rm AHS}$.
We first show the local star-shape of $X_{\rm AHS}$ at $\bar{x}$; i.e., 
for $x \in X_{\rm AHS}$ that is sufficiently close to $\bar x$, there exists 
$\bar \tau$ such that $\bar x + \tau (x-\bar x) \in X_{\rm AHS}$ for any 
$\tau \in [0, \bar \tau]$. By \cite{HanCuiPang24}, for all 
$x \in X_{\rm AHS}$ sufficiently close to $\bar{x}$, 
the vector $x - \bar{x}$ belongs to the tangent cone 
${\cal T}(X_{\rm AHS};\bar{x})$ of $X_{\rm AHS}$ at $\bar{x}$; moreover,
for all $i = 1, \cdots, I$,
\[
A_{\, i \bullet} \, \bar{x} + \displaystyle{
\sum_{j \in {\cal J}_{i,0}(\bar{x})}
} \, \psi_{ij} \, 
\onebld_{[ \, 0,\infty )}( 
\phi_{ij}^{\, \prime}(\bar{x};x - \bar{x}) ) + \displaystyle{
\sum_{j \in {\cal J}_{i,>}(\bar{x})}
} \, \psi_{ij} \, \geq \, \eta_i 
\]
and if equality holds, then
$A_{\, i \bullet} \, ( x - \bar{x}) \geq 0$.
It then follows that with $x^{\tau} \triangleq
\bar{x} + \tau ( x - \bar{x} )$,
we have for all $x$ sufficiently close to $\bar{x}$ and $\tau > 0$ 
sufficiently small,
\[ 
\begin{array}{l}
A_{\, i \bullet} \, x^{\tau} + \displaystyle{
\sum_{i=1}^{J_i}
} \, \psi_{ij} \, 
\onebld_{[ \, 0,\infty )}( \phi_{ij}(x^{\tau})) - \eta_i \\ [0.2in]
= \, \left(  A_{\, i \bullet} \, \bar{x} + \displaystyle{
\sum_{j \in {\cal J}_{i,0}(\bar{x})}
} \, \psi_{ij} \, \onebld_{[ \, 0,\infty )}( 
\phi_{ij}^{\, \prime}(\bar{x};x - \bar{x}) ) +
\displaystyle{
\sum_{j \in {\cal J}_{i,>}(\bar{x})}
} \, \psi_{ij}  - \eta_i \, \right) + 
\tau \, A_{\, i \bullet} \, ( x - \bar{x} ) \geq 0
\end{array} \]
Thus $X_{\rm AHS}$ has the star-shape property, and similarly, we also 
have the local star-shape property of $\mbox{hypo}(\theta_{\rm AHS})$ at the
pair $\bar{z} \triangleq (\bar{x},\bar{t})$ where 
$\bar{t} \triangleq \theta_{\rm AHS}(\bar{x})$. 

We may assume without loss
of generality that the neighborhood ${\cal N}_{\bar{x}}$
is such that there exists 
a neighborhood ${\cal N}_{\bar{t}}$ of $\bar{t}$
satisfying:  $ \theta_{\rm AHS}(x) \in {\cal N}_{\bar{t}}$ for 
all $x \in {\cal N}_{\bar{x}}$; moreover, for any 
$(x,t) \in {\cal N}_{\bar{x}} \times {\cal N}_{\bar{t}}$, 
a scalar $\bar{\tau} > 0$ exists such that 
$\bar{x} + \tau ( x - \bar{x} ) \in X_{\rm AHS}$ and 
$(\bar{x},\bar{t}) + \tau (x - \bar{x},t -\bar{t}) \in 
\mbox{hypo}( \theta_{\rm AHS})$
for all $\tau \in [ \, 0,\bar{\tau} \, ]$.    Hence, for all 
$x \in {\cal N}_{\bar{x}}$. the pair
$(\bar{x},\bar{t}) + \tau (x - \bar{x}, \theta_{\rm AHS}(x) - \bar{t})$ 
belongs to $\mbox{hypo}(\theta_{\rm AHS})$.  Thus if
$x \in X_{\rm AHS} \cap {\cal N}_{\bar{x}}$, we have
\[
\bar{t} + \tau (  \theta_{\rm AHS}(x) - \bar{t} ) - \displaystyle{
\frac{\rho}{2}
} \, \tau^2 \, \| \, x - \bar{x} \, \|_2^2 \, \leq \,  \theta_{\rm AHS}(\bar{x}),
\epc \forall \, \tau \, \in \, [ \, 0,\bar{\tau} \, ],
\]
which yields $ \theta_{\rm AHS}(x) -  \theta_{\rm AHS}(\bar{x}) - \displaystyle{
\frac{\rho}{2}
} \, \tau \, \| \, x - \bar{x} \|_2^2 \leq 0$ for all 
$\tau \in [ \, 0,\bar{\tau} \, ]$.  Passing to the limit
$\tau \downarrow 0$, we deduce $ \theta_{\rm AHS}(x) \leq 
\theta_{\rm AHS}(\bar{x})$, establishing
that $\bar{x}$ is a local maximizer of $ \theta_{\rm AHS}$ on $X$. 
\end{proof}

\section{Iterative Shrinkage Methods for the A-HSCOP} \label{sec:ISA-PIP}

Based on the approximation and decomposition strategies in 
Sections~\ref{sec:approx formulation} and \ref{sec:PA decomp}, we design 
two IP-based algorithms for solving
the A-HSCOP (\ref{eq:A_HSCOP}) with mixed signed
$\psi_{ij}$ and PA internal functions $\phi_{ij}$.  
Both algorithms employ a sequence of shrinking
$\{ \varepsilon_{\nu} \} \downarrow 0$ and 
solve subproblems  
$\left( \, \mbox{\bf P}_{\rho}^{\, \varepsilon_{\nu}; 
\mathbf{k},\boldsymbol{\ell}}(\bar{x}^{\nu}) \, \right)$, to either global 
or local optimality.  The first algorithm is the iterative shrinkage 
algorithm 
with full IP subproblems solved to global optimality. To facilitate 
computational efficiency, we combine the iterative shrinkage scheme with a 
``reduction" idea by progressively fixing a portion of Heaviside function 
values, which leads to the solution of a sequence of partial IP subproblems 
with less number of integer variables.   The development and convergence 
analysis of the algorithms are presented in the following two subsections.

\subsection{Iterative decomposed shrinkage algorithm: 
The full IP version}
\label{subsec:Iterative shrinkage full}

Below we present the iterative shrinkage algorithm, 
in which we solve at each iteration the decomposed 
approximation subproblem 
$\left( \, \mbox{\bf P}_{\rho}^{\, \varepsilon_{\nu}; 
\mathbf{k},\boldsymbol{\ell}}(x^{\nu}) \, \right)$
to global optimality, by employing its
full IP formulation.  Since this formulation is a
special case of the partial formulation to be
introduced later (see (\ref{eq:reduced IP formulation of subproblems})), 
we omit the full formulation here
and directly present the full IP-based IDSA for solving
(\ref{eq:A_HSCOP}).

\begin{algorithm}[h]
\caption{Iterative decomposed shrinkage 
algorithm (IDSA)--Full MIP.}
\label{alg:IDSA-IP}
\begin{algorithmic}[1]
\State \textbf{Initialization.} Let 
$\{\varepsilon_{\nu} \}_{\nu=0}^{\infty}$
be a decreasing sequences of positive scalars.  
Let $x^0$ be a feasible iterate to the problem 
$\left( \, \mbox{\bf P}_{\rm AHS}^{\, \varepsilon_0} \, 
\right)$. Set $\rho\in \mathbb R_+$, each 
$\delta_{ij} \in \mathbb R_+$ and $\nu = 0$.
\Statex
\For{$\nu=0, 1, 2, \ldots,$}
\State \textbf{Main Computation.} Solve the problem
$\left( \, \mbox{\bf P}_{\rho}^{\, \varepsilon_{\nu}; 
\mathbf{k},\boldsymbol{\ell}}(x^{\nu}) \, \right)$ 
using an IP solver for 
all pairs $( \mathbf{k},\boldsymbol{\ell} )$ in 
${\cal M}^{\boldsymbol{\delta}}(x^{\nu})$.  
Let $x^{\nu+1}$ be an optimal solution of $\left( \, 
\mbox{\bf P}_{\rho}^{\, \varepsilon_{\nu}; 
\mathbf{k}^{\nu},\boldsymbol{\ell}^{\, \nu}}(x^{\nu}) 
\, \right)$ where 
$( \mathbf{k}^{\nu},\boldsymbol{\ell}^{\, \nu} )$ is 
a pair in ${\cal M}^{\boldsymbol{\delta}}(x^{\nu})$ 
with the highest objective value, i.e.,  
\[ \begin{array}{l}
\displaystyle{
\operatornamewithlimits{\mbox{\bf maximum}}_{
x \in X_{\rm AHS}^{\, \varepsilon_{\nu}; 
\mathbf{k}^{\nu},\boldsymbol{\ell}^{\nu}}}
} \, \theta_{\rho}^{\, \varepsilon_{\nu};
\mathbf{k}^{\nu}_0,\boldsymbol{\ell}^{\, \nu}_0}(x;x^{\nu})   = {\bf maximum}\left\{ \, \displaystyle{
\operatornamewithlimits{\mbox{\bf maximum}}_{
x \in X_{\rm AHS}^{\, \varepsilon_{\nu}; 
\mathbf{k},\boldsymbol{\ell}}}
} \, \theta_{\rho}^{\, \varepsilon_{\nu}; 
\mathbf{k}_0,\boldsymbol{\ell}_0}(x;x^{\nu}) \, \mid \, 
( \mathbf{k},\boldsymbol{\ell} ) \in 
{\cal M}^{\boldsymbol{\delta}}(x^{\nu}) \, \right\}  
\end{array} \] 
\State \textbf{Termination check.}  If a prescribed criterion
is satisfied (e.g.\ $\| x^{\nu+1} - x^{\nu} \| \leq \mbox{ given tolerance}$),
terminate; else continue.
\EndFor
\end{algorithmic}
\end{algorithm}

We make several remarks about the sequence 
$\{ x^{\nu} \}$.  
First and foremost is the existence of the initial $x^0$ satisfying
$x^0 \in X_{\rm AHS}^{\varepsilon_0}$.  If a vector 
$x^0 \in X_{\rm AHS}$ is available, then an
$\varepsilon_0 > 0$ can always be easily picked so that 
$x^0 \in X_{\rm AHS}^{\varepsilon_0}$.  In general, without 
knowing the 
existence of such a $x^0$, we may consider the modification 
of problem $\left( \, 
\mbox{\bf P}_{\rm AHS}^{\, \varepsilon_0} \, \right)$
by incorporating the minimization of a constraint residual 
and apply the
algorithm to the resulting problem: for a given scalar 
$\lambda > 0$,
\begin{equation} \label{eq:subproblem with residual} 
\begin{array}{l}
\displaystyle{
\operatornamewithlimits{\mbox{\bf maximize}}_{
x \, \in \, P; \, \gamma \, \geq \, 0}
} \,   c(x) + \displaystyle{
\sum_{j=1}^{J_0}
} \, \psi_{0j}^+ \, \onebld_{[ \, 0,\infty )}\left( 
\phi_{0j}(x) \right) - \displaystyle{
\sum_{j=1}^{J_0}
} \, \psi_{0j}^- \,
\onebld_{( \, -\varepsilon_0,\infty )}\left( 
\phi_{0j}(x) \right) \, - \\ [0.1in]
\hspace{1.2in} \displaystyle{
\frac{\rho}{2}
} \, \| x - \bar{x} \|_2^2 - \lambda \, \gamma \\ [0.1in]
\mbox{\bf subject to } \mbox{for all $i = 1, \cdots, I$:} \\ [3pt]
A_{\, i \bullet} \, x + \displaystyle{
\sum_{j=1}^{J_i}
} \, \psi_{ij}^+ \, \onebld_{[ \, 0,\infty )}\left( 
\phi_{ij}(x) \right) + \gamma \, \geq \, \displaystyle{
\sum_{j=1}^{J_i}
} \, \psi_{ij}^- \, \onebld_{( \, -\varepsilon_0,\infty )} \left( 
\phi_{ij}(x) \right) + \eta_i, 
\end{array} 
\end{equation}
which is always feasible. In the subsequent numerical implementation,
We set $\lambda$ to be a very large constant 
and apply the IDSA algorithm with an easily found feasible solution to the 
above residual-added problem \eqref{eq:subproblem with residual}. 
During the iterative process, when a feasible solution  to $\left( \, 
\mbox{\bf P}_{\rm AHS}^{\, \varepsilon_0} \, \right)$  is obtained
with $\gamma = 0$, then al subsequent subproblems will have a feasilble
solution on hand without the $\gamma$-variable.
As presented, we have not specified properties
of the stand-alone function $c(x)$ except for the implicit understanding
that it is such that each subproblem
$\left( \, \mbox{\bf P}_{\rho}^{\, \varepsilon_{\nu}; 
\mathbf{k},\boldsymbol{\ell}}(x^{\nu}) \, \right)$ can be solved 
to global optimality by an IP solver.  

In the following convergence
analysis, we assume that the initial pair $(x^0,\varepsilon_0)$ 
as stated in the algorithm exists. By the set inclusion 
property \eqref{eq:minorization}, $x^{\nu+1}$ is feasible to
$\left( \, \mbox{\bf P}_{\rm AHS}^{\, \varepsilon_{\nu+1}} 
\, \right)$; thus the 
iteration can be continued with $x^{\nu+1}$ replacing $x^{\nu}$.  
Thus the entire sequence $\{ x^{\nu+1} \}$ is well defined provided 
that $P$ is compact.  The following result establishes the convergence of 
Algorithm~\ref{alg:IDSA-IP}.

\begin{theorem} \label{th:convergence Algorithm IDSA-IP} \rm
Let $P$ be a compact set.  Let $c$ be continuous and 
each $\phi_{ij}$ be piecewise affine.
Suppose that a pair $(x^0,\varepsilon_0)$ exists satisfying
$\varepsilon_0 > 0$ and $x^0 \in X_{\rm AHS}^{\, \varepsilon_0}$. 
For an arbitrary pair $(\rho,\boldsymbol{\delta}) > 0$, 
Algorithm~\ref{alg:IDSA-IP} 
generates a well-defined, bounded sequence 
$\{ x^{\nu} \} \subseteq X_{\rm AHS}$.  For
any accumulation point $x^{\infty}$ of such a sequence, 
provided that the
functions $\phi_{ij}$ are nonnegative near $x^{\infty}$ for all 
$j \in {\cal J}_{i,0}^-(x^{\infty})$ and all 
$i = 0,1, \cdots, I$,
it holds that $x^{\infty}$ is a local maximizer of
(\ref{eq:A_HSCOP}).  
\end{theorem}

\begin{proof}
By the set inclusion property \eqref{eq:minorization}, we have 
$\{ x^{\nu} \} \subseteq X_{\rm AHS}$; thus this 
sequence is bounded and has at least one accumulation point.
Let $x^{\infty}$ be the limit of a convergent subsequence
$\{ x^{\nu} \}_{\nu \in \kappa}$ of the sequence $\{ x^{\nu} \}$
produced by the algorithm.  We first claim that 
$\{x^{\nu+1}\}_{\nu \in \kappa}$ also converges to $x^\infty$. 
We have, for all $\nu = 0, 1, \cdots$,
\begin{equation*} \label{eq:max over many} 
\begin{array}{l}
\theta^{\, \varepsilon_{\nu+1}}(x^{\nu+1}) \, \geq \,
\theta^{\, \varepsilon_{\nu}}(x^{\nu+1})  \geq \, 
\theta^{\, \varepsilon_{\, \nu}; 
\mathbf{k}^{\nu}_0,\boldsymbol{\ell}^{\, \nu}_0}(x^{\nu+1}) 
\, = \, \theta_{\rho}^{\, \varepsilon_{\nu}; 
\mathbf{k}^{\nu}_0,\boldsymbol{\ell}^{\, \nu}_0}(x^{\nu+1};x^{\nu})
+ \displaystyle{
\frac{\rho}{2}
} \, \| x^{\nu+1} - x^{\nu} \|_2^2.  
\end{array}
\end{equation*}
Hence,  by optimality of $x^{\nu+1}$ we obtain
\begin{equation}
\label{eq:descent sequence}
\theta^{\, \varepsilon_{\nu+1}}(x^{\nu+1}) \, \geq \, \theta^{\, \varepsilon_{\nu};
\mathbf{k}_0,\boldsymbol{\ell}_0}(x^{\nu}) +
\displaystyle{
\frac{\rho}{2}
} \, \| x^{\nu+1} - x^{\nu} \|_2^2 =   
\theta^{\, \varepsilon_{\nu}}(x^{\nu}) +
\displaystyle{
\frac{\rho}{2}
} \, \| x^{\nu+1} - x^{\nu} \|_2^2,
\end{equation}
where $( \mathbf{k},\boldsymbol{\ell} ) \in {\cal M}(x^{\nu})$ in the 
second inequality.  Hence the sequence 
$\{ \theta^{\, \varepsilon_{\nu}}(x^{\nu}) \}$ is nondecreasing.  Since 
$\theta^{\, \varepsilon}(x )  \, \leq \theta(x ) \, \leq \, c(x) + 
\displaystyle{
\sum_{j=1}^{J_0}
} \, | \, \psi_{0j} \, |$ for all $x \, \in \, P$ and $\varepsilon >0$,  
it follows that  the sequence 
$\{ \theta^{\, \varepsilon_{\nu}}(x^{\nu}) \}$ 
converges; thus the sequence 
$\{ \, \| x^{\nu+1} - x^{\nu} \| \, \} \to 0$.
Hence, $\displaystyle{
\lim_{\nu ( \in \kappa ) \to \infty}
} \, x^{\nu+1} = x^{\infty}$.

We next show that $x^{\infty}$ is feasible to (\ref{eq:A_HSCOP}) if
$x^{\infty}$ is such that
$\phi_{ij}$ is nonnegative in a neighborhood of ${x}^{\infty}$ 
for all $j \in {\cal J}_{i,0}^-(x^{\infty})$ and all 
$i = 0,1, \cdots, I$; i.e., if $x^{\infty}$ satisfies the LSI condition. 
We have, for every $\nu \in \kappa$ and all $i = 1, \cdots, I$,
\[ \begin{array}{l}
A_{\, i \bullet} \, x^{\nu+1} + \displaystyle{
\sum_{j=1}^{J_i}
} \, \psi_{ij}^+ \, \onebld_{[ \, 0,\infty )}(\phi_{ij}( x^{\nu+1})) 
\, \geq \, A_{\, i \bullet} \, x^{\nu+1} + \displaystyle{
\sum_{j=1}^{J_i}
} \, \psi_{ij}^+ \, \onebld_{[ \, 0,\infty )}\left( 
\phi_{ij}^{\, k_{ij}^{\nu};+}( x^{\nu+1} ) \right) \\ [0.2in]
\geq \, \displaystyle{
\sum_{j=1}^{J_i}
} \, \psi_{ij}^- \, 
\onebld_{( \, -\varepsilon_{\nu},\infty )}\left( 
\phi_{ij}^{\, \ell_{ij}^{\, \nu};-}( x^{\nu+1} ) \right) + \eta_i  
\geq \, \displaystyle{
\sum_{j=1}^{J_i}
} \, \psi_{ij}^- \, \onebld_{( \, -\varepsilon_{\nu},\infty )}( 
\phi_{ij}(x^{\nu+1})) + \eta_i \\ [0.2in]
= \, \displaystyle{
\sum_{j \in {\cal J}_{i,0}^-(x^{\infty})}
} \, \psi_{ij}^- \, \onebld_{( \, -\varepsilon_{\nu},\infty )}( 
\phi_{ij}(x^{\nu+1})) + \displaystyle{
\sum_{j \not\in {\cal J}_{i,0}^-(x^{\infty})}
} \, \psi_{ij}^- \, \onebld_{( \, -\varepsilon_{\nu},\infty )}( 
\phi_{ij}(x^{\nu+1})) + \eta_i \\ [0.25in] 
= \, \displaystyle{
\sum_{j=1}^{J_i}
} \, \psi_{ij}^- \, \onebld_{[ \, 0,\infty )}( 
\phi_{ij}(x^{\infty})) + \eta_i, \epc 
\forall \, \nu \, \in \, \kappa
\mbox{ sufficiently large, since $\varepsilon_{\nu} \downarrow 0$},
\end{array} \]
where the last equality holds by (\ref{eq:sign invariant}) and the 
LSI condition at $x^\infty$.  Hence taking limsup on both sides, 
we obtain, by the upper semicontinuity of
$\onebld_{[ \, 0,\infty )}( \bullet )$,
\[ \begin{array}{l}
A_{\, i \bullet} \, x^{\infty} + \displaystyle{
\sum_{j=1}^{J_i}
} \, \psi_{ij}^+ \, 
\onebld_{[ \, 0,\infty )}(\phi_{ij}( x^{\infty}))  \geq \, \displaystyle{
\sum_{j=1}^{J_i}
} \, \psi_{ij}^- \, \onebld_{[ \, 0,\infty )}( 
\phi_{ij}(x^{\infty})) + \eta_i, 
\end{array}
\]
showing that $x^{\infty}$ is feasible to (\ref{eq:A_HSCOP}).
We next show that $x^{\infty}$ is a maximizer of 
$\theta$ on $X_{\rm AHS} \cap {\cal N}$,
where ${\cal N}$ is a neighborhood of $x^{\infty}$ satisfying

\gap

\noindent $\bullet $ ${\cal N}$ contains $x^{\nu}$ and $x^{\nu+1}$
for all $\nu \in \kappa$ sufficiently large;

\gap

\noindent $\bullet $ for all $x \in {\cal N}$, all $\varepsilon > 0$
sufficiently small, all $j \not\in {\cal J}_{i,0}^-(x^{\infty})$ 
and all $i = 0, 1, \cdots, I$, 
\begin{equation} \label{eq:same sign for IDSA IP}
\onebld_{( \, -\varepsilon,\infty )}(\phi_{ij}(x)) \, = \,
\onebld_{[ \, 0,\infty )}(\phi_{ij}(x)) \, = \, 
\onebld_{[ \, 0,\infty )}(\phi_{ij}(x^{\infty}));
\end{equation}
$\bullet $ $\phi_{ij}$ is nonnegative in ${\cal N}$ for all
all $j \in {\cal J}_{i,0}^-(x^{\infty})$ and all 
$i = 0, 1, \cdots, I$; and

\noindent $\bullet $ for any $x$ and $x^{\prime}$ both in 
${\cal N}$,
${\cal M}(x) \subseteq {\cal M}^{\boldsymbol{\delta}}(x^{\prime})$ 
(by Lemma \ref{lm:simple fact}); this is
where the positivity of $\delta$ is needed. 

\noindent Let $x \in X_{\rm AHS} \cap {\cal N}$ be arbitrary.  
Let the pair 
$( \mathbf{k},\boldsymbol{\ell} ) \in {\cal M}(x)$, we must have 
$( \mathbf{k},\boldsymbol{\ell} ) \in 
{\cal M}^{\boldsymbol{\delta}}(x^{\nu})$. 
With the same argument in Proposition \ref{pr:N&S III locmax AHSCOP} part (B),
we must have  $x \in X^{\, \varepsilon_{\nu};
\mathbf{k},\boldsymbol{\ell}}_{\rm AHS} \cap \mathcal N$.  Hence, 
for all $\nu \in \kappa$ sufficiently large, we have
\[ \begin{array}{l}
c(x^{\nu+1}) + \displaystyle{
\sum_{j=1}^{J_0}
} \, \psi_{0j}^+ \, 
\onebld_{[ \, 0,\infty )}( \phi_{0j}(x^{\nu+1}) )
- \displaystyle{
\sum_{j=1}^{J_0}
} \, \psi_{0j}^- \,
\onebld_{[ \, 0,\infty )}( \phi_{0j}(x^{\nu+1}) ) \\ [0.2in]
= \, c(x^{\nu+1}) + \displaystyle{
\sum_{j=1}^{J_0}
} \, \psi_{0j}^+ \, \onebld_{[ \, 0,\infty )}( 
\phi_{0j}(x^{\nu+1}) )
- \displaystyle{
\sum_{j=1}^{J_0}
} \, \psi_{0j}^- \,
\onebld_{( \, -\varepsilon_{\nu+1},\infty )}( 
\phi_{0j}(x^{\nu+1}) ) 
\epc \mbox{by LSI condition} \\ [0.25in]
= \, \theta^{\varepsilon_{\nu+1}}(x^{\nu+1}) \epc
\mbox{by definition of $\theta^{\varepsilon_{\nu+1}}$} \\ [0.1in] 
\geq \, \theta^{\varepsilon_{\nu}}(x) - 
\displaystyle{
\frac{\rho}{2}
} \, \| x - x^{\nu} \|_2^2
\epc \mbox{by a derivation similar to (\ref{eq:descent sequence})} 
\\ [0.2in]
= \, c(x) + \displaystyle{
\sum_{j=1}^{J_0}
} \, \psi_{0j} \, \onebld_{[ \, 0,\infty )}( \phi_{0j}(x) ) - 
\displaystyle{
\frac{\rho}{2}
} \, \| x - x^{\nu} \|_2^2 \epc \mbox{by LSI condition} 
\end{array}
\]
Taking limits $\nu (\in \kappa) \to \infty$ on both sides
yields 
\[
c(x^{\infty}) +  \displaystyle{
\sum_{j=1}^{J_0}
} \, \psi_{0j} \, \onebld_{[ \, 0,\infty )}( \phi_{0j}(x^{\infty}) )
\, \geq \, c(x) +  \displaystyle{
\sum_{j=1}^{J_0}
} \, \psi_{0j} \, \onebld_{[ \, 0,\infty )}( \phi_{0j}(x) ) -
\displaystyle{
\frac{\rho}{2}
} \, \| x - x^{\infty} \|_2^2,
\]
which establishes the local optimality of $x^{\infty}$ for the regularized HSCOP
problem, and thus for the original
(\ref{eq:A_HSCOP}) by 
Proposition~\ref{pr:N&S III locmax AHSCOP}.
\end{proof}

\subsection{Iterative decomposed shrinkage algorithm: 
the PIP version}

The ambitious goal of the iterative decomposed 
shrinkage algorithm comes with it the real
concern of its efficiency for 
large-scale problems.   As a practical alternative, 
instead of solving 
the full IP at each iteration, we adopt the 
``reduced" idea employed
in \cite{YueFang23} by embedding the PIP algorithm (with details
presented in Algorithm~\ref{alg:pip_details})
into IDSA as an inner loop solver for the decomposed 
approximating subproblems. The resulting PIP version 
of IDSA (denoted IDSA-PIP and described in 
Algorithm~\ref{alg:IDSA-PIP}) therefore solves a sequence 
of partial IP subproblems. The main idea in PIP 
is to iteratively fix part of Heaviside terms that are 
deemed likely to be determinate and let stand the 
remaining indeterminate ones, and progressively 
increase the latter ones until a 
fixed-point criterion is satisfied, resulting in 
significant savings in computational efforts.  Moreover, 
convergence of the IDSA-PIP algorithm can similarly be 
proved.
 
Specifically, at the beginning of an iteration of 
the iterative shrinkage algorithm, given the tuple
$( \varepsilon_\nu, \mathbf{k},\boldsymbol{\ell}, x^\nu )$
with the pair $( \mathbf{k},\boldsymbol{\ell} ) \in 
{\cal M}^{\delta}(x^\nu)$ that defines the 
decomposed approximating problem 
$\left( \, \mathbf P_\rho^{\varepsilon_{\nu};\mathbf{k},
\boldsymbol{\ell}}(x^\nu) \, \right)$ 
in \eqref{eq:approximated II AHSCOP regularized},  
we fix additionally a scalar $\delta^{\prime} > 0$ 
along with an incumbent solution 
$\wh x \in X_{\rm AHS}^{\varepsilon_\nu; \mathbf{k},\boldsymbol{\ell}}$, 
and construct the following disjoint index 
subsets $\left\{ \, \wh{\cal J}_{i,>}^{\, \pm,\delta^{\, 
\prime}}( \wh{x} ), \wh{\cal J}_{i,<}^{\, \pm,\delta^{\, 
\prime}}( \wh{x} ), \, \wh{\cal J}_{i,0}^{\, \pm,\delta^{\prime}}( \wh{x} ) 
\ \right\}$ 
of $\{ 1, \cdots, J_i \}$ for each $i = 0, 1, \cdots, I$: 
\begin{equation} \label{eq:index_subset_J}
\begin{array}{l}
\wh{\cal J}_{i,>}^{\, +,\delta^{\prime}}( \wh{x} ) \triangleq
\{ \, j  \in [J_i] \mid   \psi_{ij} > 0; \ 
\phi_{ij}^{k_{ij};+}( \wh{x} ) \, \geq \, \delta^{\prime} \, \}; 
\\ [0.1in] 
\wh{\cal J}_{i,<}^{\, +,\delta^{\prime}}( \wh{x} ) \triangleq
\{ \, j  \in [J_i] \mid  \psi_{ij} > 0; \ 
\phi_{ij}^{k_{ij};+}( \wh{x} ) \, \leq \, -\delta^{\prime} \, \}; 
\\ [0.1in]
\wh{\cal J}_{i,>}^{\, -,\delta^{\prime}}( \wh{x} ) \triangleq
\{ \, j  \in [J_i]  \mid  \psi_{ij} < 0; \ 
-\phi_{ij}^{\ell_{ij};-}( \wh{x} ) - \varepsilon_\nu \, \geq \, 
\delta^{\prime} \}; \epc \mbox{(omit $\varepsilon_{\nu}$ in notation
$\wh{\cal J}_{i,>}^{\, -,\delta^{\prime}}( \wh{x} )$)}
\\ [0.1in]  
\wh{\cal J}_{i,<}^{\, -,\delta^{\prime}}( \wh{x} ) \triangleq
\{ \, j  \in [J_i]  \mid   \psi_{ij} < 0; \ 
-\phi_{ij}^{\ell_{ij};-}( \wh{x} ) - \varepsilon_\nu \, \leq \, 
-\delta^{\prime} \}; \epc \mbox{(omit $\varepsilon_{\nu}$ in notation
$\wh{\cal J}_{i,<}^{\, -,\delta^{\prime}}( \wh{x} )$)} \\ [0.1in]
\wh{\cal J}_{i,0}^{+,\delta^{\prime}}( \wh{x} ) \, \triangleq \, 
\{ \, j  \in [J_i]  \mid  \psi_{ij} \, > \, 0 \, \} \, \setminus \,
\left( \, \wh{\cal J}_{i,>}^{\, +,\delta^{\prime}}( \wh{x} ) \, \cup \, 
\wh{\cal J}_{i,<}^{\, +,\delta^{\prime}}( \wh{x} ) 
\, \right);  \\[0.15in] 
\wh{\cal J}_{i,0}^{-,\delta^{\prime}}(\wh{x} ) \, \triangleq \, 
\{ \, j  \in [J_i]  \mid   \psi_{ij} \, < \, 0 \, \} \, \setminus \,
\left(   \wh{\cal J}_{i,>}^{\, -,\delta^{\prime}}( \wh{x} ) \, \cup \, 
\wh{\cal J}_{i,<}^{\, -,\delta^{\prime}}( \wh{x} )   \right).
\end{array}
\end{equation}
The superscripts $\pm$ refer to the signs of coefficients
$\{\psi_{ij}\}$,
subscripts $\{>, <\}$ refer to function values that are 
greater or less 
than $\pm \delta^{\prime}$, and the subscript $0$ 
refers to the function value that is in between 
$\pm \delta^{\prime}$. A remark about the scalar $\delta^{\prime}$ is in
order.  In the implementation of PIP, multiple such scalars are used,
one for each index set $\wh{\cal J}_{i,>}^{\pm,\delta^{\prime}}(\wh{x} )$
and  $\wh{\cal J}_{i,<}^{\pm,\delta^{\prime}}(\wh{x} )$; for simplicity
in the description of the partial A-HSCOP and its equivalent IP, we use 
one single scalar $\delta^{\prime}$.  As the computational 
workhorse, the partial A-HSCOP at the 
incumbent solution $\wh x$ is defined as follows:
\begin{equation} \label{eq:partial HSCOP}
\begin{array}{ll}
\displaystyle{
\operatornamewithlimits{\mbox{\bf maximum}}_{
x \, \in \, P}
} \ & c(x) +
\displaystyle{
\sum_{j \in \wh{\cal J}_{0,0}^{\, +,\delta^{\prime}}(\wh{x})}
}   \psi_{0j}^+   \onebld_{[0,\infty)}(\phi_{0j}^{k_{0j};+}(x) ) + 
\displaystyle{
\sum_{j \in \wh{\cal J}_{0,0}^{\, -,\delta^{\prime}}(\wh{x})}
}   \psi_{0j}^-  \onebld_{[0,\infty)}(-\phi_{0j}^{\ell_{0j};-}(x) 
- \varepsilon_\nu ) \\[0.1in]
& + \displaystyle{
\sum_{j \in \wh{\cal J}_{0,>}^{\, +,\delta^{\prime}}(\wh{x})}
}   \psi_{0j}^+    - \displaystyle{
\sum_{j   \in   \wh{\cal J}_{0,0}^{\, -,\delta^{\prime}}(\wh{x})  \, \cup \,   
\wh{\cal J}_{0,<}^{\, -,\delta^{\prime}}(\wh{x})}
}  \psi_{0j}^-  - \displaystyle{
\frac{\rho}{2}
}   \| x - x^{\nu} \|_2^2 \\ [0.3in]
\mbox{\bf subject to}
& A_{\, i \bullet} \, x + \displaystyle{
\sum_{j \in \wh{\cal J}_{i,0}^{\, +,\delta^{\prime}}(\wh{x})}
}  \psi_{ij}^+   \onebld_{[0,\infty)}(\phi_{ij}^{k_{ij};+}(x) ) 
+ \displaystyle{
\sum_{j \in \wh{\cal J}_{i,0}^{\, -,\delta^{\prime}}(\wh{x})}
}  \psi_{ij}^-    \onebld_{[0,\infty)}(-\phi_{ij}^{\ell_{ij};-}(x) - 
\varepsilon_\nu )\\ [0.3in] 
& \epc \geq \, \displaystyle{
\sum_{j   \in   \wh{\cal J}_{i,0}^{\, -,\delta^{\prime}}(\wh{x})  \, \cup \,   
\wh{\cal J}_{i,<}^{\, -,\delta^{\prime}}(\wh{x})}
}  \psi_{ij}^- - \displaystyle{
\sum_{j \in \wh{\cal J}_{i,>}^{\, +,\delta^{\prime}}(\wh{x})}
}  \psi_{ij}^+ + \eta_i, \epc \forall i \in [ \, I \, ] \\[0.4in]
\mbox{\bf and} & \left\{ \begin{array}{ll}
\phi_{ij}^{k_{ij};+}(x) \, \geq \, 0,  &  \forall j \, \in \, 
\wh{\cal J}_{i,>}^{\, +,\delta^{\prime}}(\wh{x}) \\ [0.1in]
-\phi_{ij}^{\ell_{ij};-}(x) \, \geq \, \varepsilon_\nu, 
&  \forall j \, \in \, \wh{\cal J}_{i,>}^{\, -,\delta^{\prime}}(\wh{x}) 
\\ [0.1in]
\end{array} \right\}, \epc \forall i = 0,\cdots, I. 
\end{array}   
\end{equation} 
This partial A-HSCOP has the equivalent
IP formulation (presented in (\ref{eq:reduced IP formulation of subproblems}))
that involves a constant $M > 0$ such that
\begin{equation} \label{eq:phi bounded}
| \, \phi_{ij}(x) \, | \, \leq \, M, \epc 
\forall \, (i,j) 
\ \mbox{ and } \ \forall \, x \, \in \, P.
\end{equation}
The assumption is easily satisfied where $P$ is a 
compact set and each $\phi_{ij}$ is continuous.  
Further, while it may be
possible to develop parameter-free IP-based methods 
without this assumption,
as done in the case of a linear program with linear 
complementarity constraints via a logical Benders 
approach \cite{HMPBKunapuli08}, such an extension is 
beyond the scope of our present work.  

Below we present the IP formulation for the partial HSCOP problem 
\eqref{eq:partial HSCOP} at the incumbent solution $\wh x$. 
\begin{equation} \label{eq:reduced IP formulation of subproblems}
\begin{array}{ll}
\displaystyle{
\operatornamewithlimits{\mbox{\bf maximum}}_{x \, \in \, P; 
\ z^{\pm}}
} \ & c(x) +
\displaystyle{
\sum_{j \in \wh{\cal J}_{0,0}^{\, +,\delta^{\prime}}(\wh{x})}
} \, \psi_{0j}^+ \, z_{0j}^+ + \displaystyle{
\sum_{j \in \wh{\cal J}_{0,0}^{\, -,\delta^{\prime}}(\wh{x})}
} \, \psi_{0j}^- \, z_{0j}^- + \displaystyle{ \sum_{j\in \wh{\cal J}_{0,>}^{\, +, \delta^{\prime}}(\wh{x})} \psi_{0j}^+} \\ [0.25in]
& - \displaystyle{\sum_{j\in \wh{\cal J}_{0,0}^{\, -, \delta^{\prime}}(\wh{x}) \, \cup\, \wh{\cal J}_{0,<}^{\, -, \delta^{\prime}}(\wh{x})}\psi_{0j}^-}
- \displaystyle{
\frac{\rho}{2}
} \, \| x - x^{\nu} \|_2^2 \\ [0.25in]
\mbox{\bf subject to }  
& A_{\, i \bullet} \, x + \displaystyle{
\sum_{j \in \wh{\cal J}_{i,0}^{\, +,\delta^{\prime}}(\wh{x})}
}  \psi_{ij}^+  z_{ij}^+ + \displaystyle{
\sum_{j \in \wh{\cal J}_{i,>}^{\, +, \delta^{\prime}}(\wh{x})}
}  \psi_{ij}^+ + \displaystyle{
\sum_{j \in \wh{\cal J}_{i,0}^{\, -,\delta^{\prime}}(\wh{x})}
}  \psi_{ij}^-   z_{ij}^- \\ [0.3in]
& \geq   \displaystyle{
\sum_{j   \in   \wh{\cal J}_{i,0}^{\, -,\delta^{\prime}}(\wh{x}) \, \cup \,   
\wh{\cal J}_{i,<}^{\, -,\delta^{\prime}}(\wh{x})}
}  \psi_{ij}^- + \eta_i, \epc \forall i\in [ \, I \, ]\\ [5pt]
\mbox{\bf and} & \mbox{\ for\ all\ } i=0,\cdots,I:\\ [5pt]
& \left\{ \begin{array}{lll}
\phi_{ij}^{k_{ij};+}(x) \, \geq \, -M ( 1 - z_{ij}^+ ); 
& z_{ij}^+ \, \in \, \{ \, 0,1 \, \}; & j \, \in \, 
\wh{\cal J}_{i,0}^{\, +,\delta^{\prime}}(\wh{x}) \\ [5pt]
-\phi_{ij}^{\ell_{ij};-}(x) \, \geq \, -M ( 1 - z_{ij}^- ) 
+ \varepsilon_{\nu};
& z_{ij}^+ \, \in \, \{ \, 0,1 \, \}; & j \, \in \, 
\wh{\cal J}_{i,0}^{\, -,\delta^{\prime}}(\wh{x}) \\ [5pt]
\end{array} \right\} \\ [0.3in]
& \left\{ \begin{array}{lll}
\phi_{ij}^{k_{ij};+}(x) \, \geq \, 0; 
& (\mbox{ equivalent to } z_{ij}^+ = 1); & j \, \in \, 
\wh{\cal J}_{i,>}^{\, +,\delta^{\prime}}(\wh{x}) \\ [5pt]
-\phi_{ij}^{\ell_{ij};-}(x) \, \geq \, \varepsilon_{\nu}; 
& (\mbox{ equivalent to } z_{ij}^- = 1); & j \, \in \, 
\wh{\cal J}_{i,>}^{\, -,\delta^{\prime}}(\wh{x}) \\ [5pt]
\end{array} \right\} \\ [0.3in] 
& \left\{ \begin{array}{lll}
\phi_{ij}^{k_{ij};+}(x) \, \mbox{ free}; 
& (\mbox{ equivalent to } z_{ij}^+ = 0); & j \, \in \, 
\wh{\cal J}_{i,<}^{\, +,\delta^{\prime}}(\wh{x}) \\ [5pt]
-\phi_{ij}^{\ell_{ij};-}(x) \, \mbox{ free}; 
& (\mbox{ equivalent to } z_{ij}^- = 0); & j \, \in \, 
\wh{\cal J}_{i,<}^{\, -,\delta^{\prime}}(\wh{x})
\end{array} \right\} .
\end{array}   
\end{equation}
The solution of the above (partial) IPs is the foremost computational 
step in the PIP version of Algorithm~\ref{alg:IDSA-IP}.
Recalling that $\phi_{ij}^{k_{ij};+}$ is a concave PA function and 
$\phi_{ij}^{\ell_{ij};-}$ is a convex PA function 
(see (\ref{eq:conv/conv phi})), we note
when $c$ is a concave function, the above IP is a mixed integer 
concave maximization program and its global optimum can be obtained in 
principle.  When $c$ is a PA function that is not necessarily concave, we 
could reformulate the full and partial IPs as a MILP by imposing additional 
binary variables to represent the pieces for $c$ 
(see \cite{BertsimasGeorghiou15,VielmaAhmedNemhauser10}).  In the rest of the paper, we do not discuss this refinement of 
(\ref{eq:reduced IP formulation of subproblems}).  

It is worthwhile to remind
the reader that behind this subproblem, there are the vector 
$x^{\nu} \in X_{\rm AHS}$ 
and the associated tuple $( \varepsilon_\nu, \mathbf{k},\boldsymbol{\ell} )$
with the pair of index tuples $( \mathbf{k},\boldsymbol{\ell} ) \in 
{\cal M}^{\delta}(x^\nu)$ from which
the index pair $(k_{ij},\ell_{ij})$ are selected to
define the pieces of the PA functions $\phi_{ij}^{k_{ij};+}(x^{\nu})$ and
$\phi_{ij}^{\ell_{ij};-}(x^{\nu})$ (see (\ref{eq:conv/conv phi})).  The 
vector $x^{\nu}$
is the iterate produced by the outer iteration in the iterative 
shrinkage procedure and fixed in the regularized objective throughout 
the PIP procedure, whose details are presented in
Algorithm~\ref{alg:pip_details} at the end of this section; 
this vector differs from 
$\wh{x} \in X_{\rm AHS}^{\varepsilon_\nu; \mathbf{k},\boldsymbol{\ell}}$
which is the PIP iterate that defines the fixing of
the binary variables. In particular, $\wh x$, $\delta^{\prime}$, and the index 
subsets \eqref{eq:index_subset_J} that define 
(\ref{eq:reduced IP formulation of subproblems}) are all 
updated during the iterations of PIP.

When the index sets $\wh{\cal J}_{i,>}^{\, \pm,\delta^{\prime}}$ and
$\wh{\cal J}_{i,<}^{\, \pm,\delta^{\prime}}$ are empty, 
(\ref{eq:reduced IP formulation of subproblems}) becomes the full IP
formulation of (\ref{eq:approximated II AHSCOP regularized}).
Computationally, with the appropriate control of these index sets not all 
empty, the problem (\ref{eq:reduced IP formulation of subproblems})
has less binary variables, thus can be solved more efficiently than the version
with full set of integer variables.  These savings, together with 
the ability to restart each iteration of PIP with a favorable iterate, are 
two main advantages of the algorithm below, which is formally the PIP version
of Algorithm~\ref{alg:IDSA-IP}.  

\begin{algorithm}[h]
\caption{Iterative decomposed 
 shrinkage algorithm(IDSA)-PIP.}
\label{alg:IDSA-PIP}
\begin{algorithmic}[1]
\State \textbf{Initialization.} Let 
$\{\varepsilon_{\nu} \}_{\nu=0}^{\infty}$
be a decreasing sequences of positive scalars.  Let $x^0$ be a feasible 
iterate to the problem 
$\left( \, \mbox{\bf P}_{\rm AHS}^{\, \varepsilon_0} \, 
\right)$. Set $\rho\in \mathbb R_+$, each $\delta_{ij}\in \mathbb R_+$ 
and $\nu = 0$.
\Statex
\For{$\nu=0, 1, 2, \ldots,$}
\State \textbf{Inner Loop Computation.} For each pair
$( \mathbf{k},\boldsymbol{\ell} )$ in 
${\cal M}^{\boldsymbol{\delta}}(x^{\nu})$, solve the problem
$\left( \, \mbox{\bf P}_{\rho}^{\, \varepsilon_{\nu}; 
\mathbf{k},\boldsymbol{\ell}}(x^{\nu}) \, \right)$ using 
the PIP Algorithm~\ref{alg:pip_details}, which returns the corresponding
fixed point of the partial A-HSCOP problem \eqref{eq:partial HSCOP}.  
\Statex
\State \textbf{Update.} Set $x^{\nu+1}$ to be the fixed point computed above 
along with the maximal objective value $\theta_{\rho}^{\, \varepsilon_{\nu}; 
\mathbf{k}_{\nu},\boldsymbol{\ell}_{\nu}}(x^{\nu+1};x^{\nu})$ corresponding
to the index pair $( \mathbf{k}^{\nu},\boldsymbol{\ell}^{\, \nu} )  \in 
{\cal M}^{\boldsymbol{\delta}}(x^{\nu})$.   At the termination of 
PIP for this pair, let $\mu_{\nu}$ be the PIP iteration count and  
$\left\{ \, \wh{\cal J}_{i,>}^{\, \pm \mu_{\nu}}( x^{\nu+1} ), 
\wh{\cal J}_{i,<}^{\, \pm \mu_{\nu}}( x^{\nu+1} ), \, 
\wh{\cal J}_{i,0}^{\, \pm \mu_{\nu}}( x^{\nu+1} ) \, \right\}$ be the index
tuples defined in Step~18 associated with the 
scalars $\delta^{\pm \mu_{\nu}}_{i,>}$ and $\delta^{\pm \mu_{\nu}}_{i,<}$ 
defined in Step~14 of Algorithm~\ref{alg:pip_details}.  Let 
$\delta_{\nu}^{\prime} \triangleq \displaystyle{
\min_{0 \leq i \leq I}
} \, \left\{ \delta^{\pm \mu_{\nu}}_{i,>}, \delta^{\pm \mu_{\nu}}_{i,<} \right\}$.
\State \textbf{Termination check.}  If a prescribed criterion
is satisfied (e.g.\ $\| x^{\nu+1} - x^{\nu} \| \leq \mbox{ given tolerance}$),
terminate; else continue.
\EndFor
\end{algorithmic}
\end{algorithm}

\noindent Algorithm~\ref{alg:IDSA-PIP} is essentially the same as 
Algorithm~\ref{alg:IDSA-IP}  except that instead of solving each subproblem
$\left( \, \mbox{\bf P}_{
\rho}^{\, \varepsilon_{\nu}; 
\mathbf{k},\boldsymbol{\ell}}(x^{\nu}) \, \right)$ to global
optimality for all pairs $( \mathbf{k},\boldsymbol{\ell} ) \in 
{\cal M}^{\boldsymbol{\delta}}(x^{\nu})$, we solve these subproblems by 
the PIP Algorithm~\ref{alg:pip_details}.  According to \cite{YueFang23}, 
during each iteration of the PIP procedure (Inner Loop Computation) embedded 
within Algorithm~\ref{alg:IDSA-PIP}, the objective 
value of problem $\left( \, \mbox{\bf P}_{\rho}^{\, \varepsilon_{\nu}; 
\mathbf{k},\boldsymbol{\ell}}(x^{\nu}) \, \right)$ must increase. 
Moreover, the PIP method  will terminate in a finite number of 
iterations with the output of a fixed point to the partial A-HSCOP 
which is a local maximizer to 
$\left( \, \mbox{\bf P}_{\rho}^{\, \varepsilon_{\nu}; 
\mathbf{k},\boldsymbol{\ell}}(x^{\nu}) \, \right)$.  When solving the problem
$\left( \, \mbox{\bf P}_{\rho}^{\, \varepsilon_{\nu}; 
\mathbf{k},\boldsymbol{\ell}}(x^{\nu}) 
\, \right)$ by PIP, we always start the PIP iterations at $x^{\nu}$.  

The following result is the convergence of the above algorithm
which is the PIP version of 
Theorem~\ref{th:convergence Algorithm IDSA-IP}.  While expected to
be more efficient, there is one major difference between the two 
results. Namely, here, instead of a nonnegativity 
assumption on the functions $\phi_{ij}$ for 
$j \in {\cal J}_{i,0}^-(x^{\infty})$ near an accumulation 
point $x^{\infty}$, i.e., the local sign invariance assumption, the 
more restrictive assumption that this
index set ${\cal J}_{i,0}^-(x^{\infty}) = \emptyset$ 
is needed.

\begin{theorem} \label{th:convergence Algorithm IDSA-PIP} \rm
Let $P$ be a compact set.  Let $c$ be continuous and 
each $\phi_{ij}$ be piecewise affine given by (\ref{eq:PA phi}).
Suppose that a pair $(x^0,\varepsilon_0)$ exists satisfying
$\varepsilon_0 > 0$ and $x^0 \in X_{\rm AHS}^{\, \varepsilon_0}$. 
For an arbitrary tuple 
$(\rho, \boldsymbol{\delta}) > 0$, suppose that 
$\delta^{\prime}_{\nu} \geq \bar \delta > 0$ for all $\nu$.  If
$x^{\infty}$ is an accumulation point of a sequence $\{ x^{\nu} \}$
produced by Algorithm~\ref{alg:IDSA-PIP}, then 
$x^{\infty}$ must be feasible to (\ref{eq:A_HSCOP});
moreover, provided that 
${\cal J}_{i,0}^-(x^{\infty})$ is empty for all 
$i = 0, 1, \cdots, I$, there exists a pair 
$( \mathbf{k}^{\infty},\boldsymbol{\ell}^{\, \infty} ) 
\in {\cal M}(x^{\infty})$
such that $x^{\infty}$ is a local maximizer of the problem:
\begin{equation} \label{eq:limit problem}
 \begin{array}{ll}
\displaystyle{
\operatornamewithlimits{\mbox{\bf maximize}}_{x \in P}
} & c(x) +   \ \displaystyle{
\sum_{j=1}^{J_0}
} \, \psi_{0j}^+ \, \onebld_{[ \, 0,\infty )}\left( 
\phi_{0j}^{\, k_{0j}^{\infty};+}(x) \right) - \displaystyle{
\sum_{j=1}^{J_0}
} \, \psi_{0j}^- \,
\onebld_{[ \, 0,\infty )}\left( 
\phi_{0j}^{\, \ell_{0j}^{\infty};-}(x) \right) \\ [0.3in]
\mbox{\bf subject to}  &
A_{\, i \bullet} x + \displaystyle{
\sum_{j=1}^{J_i}
} \psi_{ij}^+ \, \onebld_{[ \, 0,\infty )}\left( 
\phi_{ij}^{\, k_{ij}^{\infty};+}(x) \right) - \displaystyle{
\sum_{j=1}^{J_i}
} \psi_{ij}^- \, \onebld_{[ \, 0,\infty )}\left( 
\phi_{ij}^{\, \ell_{ij}^{\infty};-}(x) \right) \geq \eta_i, \
\mbox{for all $i \in [ \, I \, ]$}
\end{array}  
\end{equation}
If additionally ${\cal M}(x^{\infty})$ is a singleton, 
then $x^{\infty}$ is a local maximizer of 
(\ref{eq:A_HSCOP}).
\end{theorem}

\begin{proof}
We follow
the proof of Theorem~\ref{th:convergence Algorithm IDSA-IP}.  
We begin by noting that the string of 
inequalities (\ref{eq:descent sequence})
continues to hold because by the progressive improvement of PIP, 
\[
\theta^{\, \varepsilon_{\nu}; 
\mathbf{k}^{\nu}_0,\boldsymbol{\ell}^{\, \nu}_0}(x^{\nu+1})
- \displaystyle{
\frac{\rho}{2}
} \, \| x^{\nu+1} - x^{\nu} \|_2^2  
\, \geq \, \theta^{\, \varepsilon_{\nu};
\mathbf{k}_0^\nu,\boldsymbol{\ell}_0^\nu}(x^{\nu}),
\]
with  $x^{\nu}$ and $x^{\nu+1}$ being the starting and output points 
respectively
in the application of the PIP algorithm to the subproblem
$\left( \, \mbox{\bf P}_{\rho}^{\, \varepsilon_{\nu}; 
\mathbf{k}^{\nu},\boldsymbol{\ell}^{\nu}}(x^{\nu}) \, \right)$ 
and the sequence of objective values 
$\left\{ \, \theta^{\, \varepsilon_{\nu}; 
\mathbf{k}^{\nu}_0,\boldsymbol{\ell}^{\, \nu}_0}(x^{\nu}) 
\, \right\}$ is strictly increasing during this algorithm.  As 
before, we can show that $\displaystyle{
\lim_{\nu \to \infty}
} \, \| \, x^{\nu+1} - x^{\nu} \|_2 = 0$.  Let $x^{\infty}$ be the
limit of a convergent subsequence $\{ x^{\nu} \}_{\nu \in \kappa}$.
Similar to the proof of Theorem~\ref{th:convergence Algorithm IDSA-IP},
it can be shown that $x^{\infty}$ is feasible to 
(\ref{eq:A_HSCOP}).  The proof that $x^{\infty}$ is a local
maximizer differs from that of the previous theorem
due to the non-global optimality of each iterate $x^{\nu+1}$.
Nevertheless, we may take the same neighborhood ${\cal N}$ of
$x^{\infty}$ as given therein; i.e., the following holds:

\gap

\noindent $\bullet $ ${\cal N}$ contains $x^{\nu}$ and $x^{\nu+1}$
for all $\nu \in \kappa$ sufficiently large;

\noindent $\bullet $ for all $x \in {\cal N}$, all $\varepsilon > 0$
sufficiently small, all $j \not\in {\cal J}_{i,0}^-(x^{\infty})$ 
and all $i = 0, 1, \cdots, I$, 
\begin{equation} \label{eq:same sign for IDSA PIP}
\onebld_{( \, -\varepsilon,\infty )}(\phi_{ij}(x)) \, = \,
\onebld_{[ \, 0,\infty )}(\phi_{ij}(x)) \, = \, 
\onebld_{[ \, 0,\infty )}(\phi_{ij}(x^{\infty}));
\end{equation}
\noindent $\bullet $ for any $x$ and $x^{\prime}$ both in 
${\cal N}$,
${\cal M}(x) \subseteq {\cal M}^{\boldsymbol{\delta}}(x^{\prime})$.

\gap

Since there are only finitely many index tuples 
$\{ \mathbf{k}^{\nu},\boldsymbol{\ell}^{\, \nu} \}_{\nu=1}^{\infty}$,
there must exist a pair 
$( \mathbf{k}^{\infty},\boldsymbol{\ell}^{\, \infty} )$
and an infinite subset $\kappa^{\prime}$ of $\kappa$ such that $
( \, \mathbf{k}^{\nu},\boldsymbol{\ell}^{\, \nu} \, ) \, = \, 
( \, \mathbf{k}^{\infty},\boldsymbol{\ell}^{\, \infty} \, )$ 
for all $\, \nu \, \in \, \kappa^{\prime}$. Since 
${\cal M}(x^\nu) \subseteq {\cal M}(x^\infty)$, we have 
$( \mathbf{k}^{\infty},\boldsymbol{\ell}^{\, \infty} )$ 
must belong to ${\cal M}(x^{\infty})$.  Further, there must 
exist an infinite subset $\kappa^{\prime \prime}$ of 
$\kappa^{\prime}$ such that
\begin{equation} \label{eq:limiting tuple}
{\left( \, \wh{\cal J}_{i,>}^{\, \pm}(x^{\nu+1}), 
\wh{\cal J}_{i,0}^{\, \pm}(x^{\nu+1}),
\wh{\cal J}_{i,<}^{\, \pm}(x^{\nu+1}) \, \right)} \, = \, 
\left( \, \wh{\cal J}_{i,>}^{\, \infty;\pm}, 
\wh{\cal J}_{i,0}^{\, \infty;\pm},
\wh{\cal J}_{i,<}^{\, \infty;\pm} \, \right), \epc 
\forall \, \nu \, \in \, \kappa^{\prime \prime}.
\end{equation}  
with this common tuple  of index sets 
$\left\{ \left( \, \wh{\cal J}_{i,>}^{\, \infty;\pm}, 
\wh{\cal J}_{i,0}^{\, \infty;\pm},
\wh{\cal J}_{i,<}^{\, \infty;\pm} \, \right)\right\}_{i=0}^I$ satisfying
\begin{equation}
\label{eq:positivity of phi}
\begin{array}{l}
\wh{\cal J}_{i,>}^{\,\infty;+}\, \subseteq \, 
\{ \, j \, \mid \, \psi_{ij} \, > \, 0; \ 
\phi_{ij}^{k_{ij}^{\infty};+}( x^{\infty} ) \, \geq \,
\bar \delta  \, \}, \quad  \wh{\cal J}_{i,>}^{\, \infty;-} \, \subseteq \, 
\{ \, j \, \mid \, \psi_{ij} \, < \, 0; \
-\phi_{ij}^{\ell_{ij}^{\infty};-}( x^{\infty} ) \, \geq \, 
\bar \delta \, \}; \\ [0.1in]
\wh{\cal J}_{i,<}^{\, \infty;+} \, \subseteq \, 
\{ \, j \, \mid \, \psi_{ij} \, > \, 0; \ 
\phi_{ij}^{k_{ij}^{\infty};+}( x^{\infty} ) \, \leq \, 
-\bar \delta \, \}, \quad \wh{\cal J}_{i,<}^{\, \infty;-} \, \subseteq \, 
\{ \, j \, \mid \, \psi_{ij} \, < \, 0; \ 
-\phi_{ij}^{\ell_{ij}^{\infty};-}( x^{\infty} ) \, \leq \, 
-\bar \delta \, \}; \\ [0.1in]
\wh{\cal J}_{i,0}^{\, \infty;+} \, \triangleq \, 
\{ \, j \, \mid \, \psi_{ij} \, > \, 0 \, \} \setminus  
\left( \, \wh{\cal J}_{i,>}^{\, \infty;+} \, \cup \,
\wh{\cal J}_{i,<}^{\, \infty;+} \, \right), \ 
\wh{\cal J}_{i,0}^{\, \infty;-} \, \triangleq \, 
\{ \, j \, \mid \, \psi_{ij} \, < \, 0 \, \}  \setminus  
\left( \, \wh{\cal J}_{i,>}^{\, \infty;-} \, \cup \, 
\wh{\cal J}_{i,<}^{\, \infty;-} \, \right).
\end{array}
\end{equation}
Since $\delta^{\prime}_{\nu} \geq \bar \delta>0$ , we deduce that
$\wh{\cal J}_{i,0}^{-}(x^\infty) \subseteq \wh{\cal J}_{i,0}^{\,\infty;-}$. 
By assumption, since
${\cal J}_{i,0}^-(x^{\infty}) = \emptyset$ for all 
$i = 0, 1, \cdots, I$ and 
$\ell_{ij}^{\, \infty} \in {\cal L}_{ij}(x^{\infty})$, we have
$\phi_{ij}(x^{\infty}) = 
\phi_{ij}^{\ell_{ij}^{\infty};-}(x^{\infty})$.  Thus, it follows that if 
$\psi_{ij} \, < \, 0$, then 
$\phi_{ij}^{\ell_{ij}^{\infty};-}( x^{\infty} ) \, \neq \, 0.$
Hence, there exist a scalar $\bar{\varepsilon} > 0$ and a 
neighborhood ${\cal N}$ of $x^{\infty}$ such that
\begin{equation} \label{eq:consequence of empty}
\begin{array}{l}
\psi_{ij} \, < \, 0 \ \Rightarrow \ \mbox{for all } 
x \, \in \, {\cal N} \ \mbox{and all} \
\varepsilon \, \in \, ( \, 0,\bar{\varepsilon} \, ):\\ [0.1in] 
\mbox{sgn}\left( -\phi_{ij}^{\ell_{ij}^{\infty};-}(x) \right) 
\, = \,
\mbox{sgn}\left( -\phi_{ij}^{\ell_{ij}^{\infty};-}(x) - 
\varepsilon \right) \, = \, 
\mbox{sgn}\left( -\phi_{ij}^{\ell_{ij}^{\infty};-}(x^{\infty})
\right) \, \neq \, 0.
\end{array}
\end{equation}
In what follows, we may assume, without loss of 
generality, that 
$\kappa^{\prime \prime} = \kappa^{\prime} = \kappa$.  We next prove 
that $x^{\infty}$ is a local maximizer of 
(\ref{eq:A_HSCOP}).  First noting that under the assumption 
${\cal J}_{i,0}^-(x^\infty) = \emptyset$, we have  $\displaystyle{
\sum_{j \in \wh{\cal J}_{i,0}^{\, \infty;-}}
} \, \psi_{ij}^- \, \onebld_{[ \, 0,\infty )}\left(
-\phi_{ij}^{\ell_{ij}^{\infty};-}(x)-\varepsilon_{\nu} \right) = \displaystyle{
\sum_{j \in \wh{\cal J}_{i,0}^{\, \infty;-}}
} \, \psi_{ij}^- \, \onebld_{[ \, 0,\infty )}\left(
-\phi_{ij}^{\ell_{ij}^{\infty};-}(x)  \right)$ in the neighborhood of 
$\cal N$ by \eqref{eq:consequence of empty}. Thus, by the constancy of the 
index sets $\{ ( \mathbf{k}^{\nu},\boldsymbol{\ell}^{\, \nu} ) \}_{
\nu \in \kappa}$,
and that of the tuples (\ref{eq:limiting tuple}) and also by 
the fixed-point optimality property of $x^{\nu+1}$, 
the iterate $x^{\nu+1}$ is a globally optimal solution to 
the following problem in the neighborhood
${\cal N}$ of $x^{\infty}$:
\begin{equation*} \label{eq:limiting intermediate}
 \begin{array}{ll}
\displaystyle{
\operatornamewithlimits{\mbox{\bf maximize}}_{x \in P}
} \ & c(x) - \displaystyle{
\frac{\rho}{2}
}  \| x - x^{\infty} \|_2^2  + 
\displaystyle{
\sum_{j \in \wh{\cal J}_{0,0}^{\, \infty;+}}
}   \psi_{0j}^+  \onebld_{[ \, 0,\infty )}\left( 
\phi_{0j}^{k_{0j}^{\infty;+}}(x) \right) 
+ \displaystyle{
\sum_{j \in \wh{\cal J}_{0,0}^{\, \infty;-}}
}   \psi_{0j}^-  \onebld_{[ \, 0,\infty )}\left( 
-\phi_{0j}^{\ell_{0j}^{\infty;-}}(x) \right) \\ [0.2in]
& + \displaystyle{\sum_{j \in \wh{\cal J}_{0,>}^{\, \infty; +}}} \psi_{0j}^+
- \displaystyle{\sum_{j\in \wh{\cal J}_{0,0}^{\,\infty; -} \, \cup\, \wh{\cal J}_{0,<}^{\,\infty; -}} \psi_{0j}^-}\\ [0.2in]
\mbox{\bf subject to }\ & A_{i \bullet} \, x + \displaystyle{
\sum_{j \in \wh{\cal J}_{i,0}^{\, \infty;+}}
} \, \psi_{ij}^+ \, \onebld_{[ \, 0,\infty )}\left(
\phi_{ij}^{k_{ij}^{\infty};+} (x) \right) + \displaystyle{
\sum_{j \in \wh{\cal J}_{i,>}^{\, \infty;+}}
} \, \psi_{ij}^+ + \, \displaystyle{
\sum_{j \in \wh{\cal J}_{i,0}^{\, \infty;-}}
} \, \psi_{ij}^- \, \onebld_{[ \, 0,\infty )}\left(
-\phi_{ij}^{\ell_{ij}^{\infty;-}}(x) \right) \\ [0.2in]
& \geq \, \displaystyle{
\sum_{j \, \in \, \wh{\cal J}_{i,0}^{\, \infty;-} \, \cup \, 
\wh{\cal J}_{i,<}^{\, \infty;-}}
} \, \psi_{ij}^- + \eta_i, \quad \forall i\in [ \, I \, ] \\ [0.2in]
\mbox{\bf and} & \left\{ \begin{array}{ll}
\phi_{ij}^{k_{ij}^{\infty;+}}(x) \, \geq \, 0; \epc & 
\forall  j \, \in \, 
\wh{\cal J}_{i,>}^{\, \infty;+}  \\ [0.1in]
-\phi_{ij}^{\ell_{ij}^{\infty;-}}(x) \, \geq \, 0; 
\epc &\forall j \, \in \, 
\wh{\cal J}_{i,>}^{\, \infty;-}
\end{array}
\right\} \quad \forall i=0,\cdots,I.
\end{array}  
\end{equation*}
with $\varepsilon_{\nu}$ removed from the Heaviside terms.
Since
\[ \begin{array}{l}
\displaystyle{
\sum_{j=1}^{J_i}
} \, \psi_{ij}^+ \, \onebld_{[ \, 0,\infty )}\left( 
\phi_{ij}^{\, k^{\infty}_{ij};+}(x) \right) 
- \displaystyle{
\sum_{j=1}^{J_i}
} \, \psi_{ij}^- \,
\onebld_{[ \, 0,\infty )}\left( 
\phi_{ij}^{\, \ell^{\infty}_{ij};-}(x) \right) \\ [0.3in]
= \displaystyle{
\sum_{j \in \wh{\cal J}_{i,0}^{\, \infty;+}}
} \, \psi_{ij}^+ \,  \onebld_{[ \, 0,\infty )} \left(\phi_{ij}^{\,k^{\infty}_{ij};+}(x)\right) + 
\displaystyle{
\sum_{j \in \wh{\cal J}_{i,0}^{\, \infty;-}}
} \, \psi_{ij}^- \, \onebld_{[ \, 0,\infty )}\left(-\phi_{ij}^{\, \ell^{\infty}_{ij};-}(x)\right) 
+ {\displaystyle{
\sum_{j \in \wh{\cal J}_{i,>}^{\, \infty;+}}
} \, \psi_{ij}^+ 
- \displaystyle{
\sum_{j \in \wh{\cal J}_{i,0}^{\, \infty;-} \,\cup\,\wh{\cal J}_{i,<}^{\, \infty;-}}
} \, \psi_{ij}^-} 
\end{array}
\]
holds for all $x$ sufficiently near $x^{\infty}$
(in a sub-neighborhood of ${\cal N}$ if needed) and all 
$i = 0, 1, \cdots, I$, it follows by a simple limiting argument and the
inclusions (\ref{eq:positivity of phi}) that 
$x^{\infty}$ is feasible and locally optimal 
for \eqref{eq:limit problem} as desired.  The last statement 
of the theorem requires no further argument.  
\end{proof}

The number of subproblems 
$\left( \, \mbox{\bf P}_{\rho}^{\varepsilon_{\nu};
\mathbf{k},\boldsymbol{\ell}}(x^{\nu}) 
\, \right)$ that need to be solved at each iteration is equal 
to the cardinality of the index set
${\cal M}^{\boldsymbol{\delta}}(x^{\nu})$, which, admittedly, may 
not be small if the sum $\displaystyle{
\sum_{i=0}^I
} \, J_i$ is not small.  In this case, a simplified version of 
the algorithm 
may be designed wherein only {\bf one} subproblem is solved with 
the pair $( \mathbf{k},\boldsymbol{\ell} )$ chosen arbitrarily (but
deterministically) from ${\cal M}^{\boldsymbol{\delta}}(x^{\nu})$.
This version is like the original difference-of-convex algorithm 
for minimizing
a difference-of-convex program in the literature
\cite{LeThiPhamDinh18,PhamDinhLeThi97}.  For this simplified 
algorithm, a limit point $\{ x^{\infty} \}$ of the produced 
sequence will satisfy the same convergence property.
Detailed analysis of this variant of Algorithm~\ref{alg:IDSA-PIP} is 
omitted.  One can also design a probabilistic version of the algorithm that 
selects
the pair $( \mathbf{k},\boldsymbol{\ell} )$ randomly based on a similar
selection rule for the probabilistic version of the difference-of-convex
algorithm as described in \cite{FengYuan26}; the convergence is then
in the almost sure sense.

\label{app:PIP}
\begin{algorithm}
\caption{Details of the PIP method}\label{alg:pip_details}
\begin{algorithmic}[1]
\Require A feasible solution $x^0$ 
to problem
$\left( \mbox{\bf P}_{\rho}^{\,\varepsilon;\mathbf{k},\boldsymbol{\ell}}
(\bar{x}) \right)$; 
an initial fraction $r_0$ of integer variables; maximal fraction $r_{\max}$
of integer variables; incremental fraction $r_{\Delta}$ of integer variables;
maximal overall iteration count $\mu_{\max}$; and maximal number of iterations
$\tilde{\mu}_{\max}$ with no improvement in objective.
\Ensure $\mathrm{obj}_{\,\mu}$ and 
\State $\mu \gets 0$, $\tilde{\mu} \gets 0$ (these are, respectively, 
the overall iteration 
count with $\mu_{\max}$ being the maximum and the iteration
count of no improvement in objective with $\tilde{\mu}_{\max}$ being the
maximum). Compute the objective value, denoted $\mathrm{obj}_0$, of
$\left( \mbox{\bf P}_{\rho}^{\,\varepsilon;\mathbf{k},\boldsymbol{\ell}}(\bar{x}) \right)$
at $(x^0,\gamma_0)$.
\While{$\mu < \mu_{\max}$ \textbf{and} $\tilde{\mu} < \tilde{\mu}_{\max}$}
    \For{$i = 0$ to $I$}
        \State $\Phi^+_{i,>} \gets \emptyset$, $\Phi^+_{i,<} \gets \emptyset$
        \For{$j = 1$ to $J_i$}
            \If{$\phi^{k_{ij};+}_{ij}(x^\mu) > 0$}
                \State $\Phi^+_{i,>} \gets \Phi^+_{i,>} \cup \left\{ \phi^{k_{ij};+}_{ij}(x^\mu) \right\}$,
            \ElsIf{$\phi^{k_{ij};+}_{ij}(x^\mu) < 0$}
                \State $\Phi^+_{i,<} \gets \Phi^+_{i,<} \cup \left\{ \phi^{k_{ij};+}_{ij}(x^\mu) \right\}$,
            \Else
                \State Randomly choose $\Phi_i^+ \in \left\{\Phi^+_{i,>}, \Phi^+_{i,<}\right\}$, $\Phi_i^+ \gets \Phi_i^+ \cup \left\{ \phi^{k_{ij};+}_{ij}(x^\mu) \right\}$.
            \EndIf
        \EndFor
        \State $\delta^{+\,\mu}_{i,>} \gets \Call{lower\_quantile}{\Phi^+_{i,>}, r^\mu}$, $\delta^{+\,\mu}_{i,<} \gets -\,\Call{upper\_quantile}{\Phi^+_{i,<}, r^\mu}$;
        \State ${\cal J}^{+\,\mu}_{i,<} \gets
        \left\{ j \in [J_i] \mid \phi^{k_{ij};+}_{ij}(x^\mu) < -\delta^{+\,\mu}_{i,<} \right\}$, ${\cal J}^{+\,\mu}_{i,>} \gets
        \left\{ j \in [J_i] \mid \phi^{k_{ij};+}_{ij}(x^\mu) > \delta^{+\,\mu}_{i,>} \right\}$, and ${\cal J}^{+\,\mu}_{i,0} \gets
        \left\{ j \in [J_i] \mid -\delta^{+\,\mu}_{i,<} \le
        \phi^{k_{ij};+}_{ij}(x^\mu) \le \delta^{+\,\mu}_{i,>} \right\}$.
        \State Similarly derive ${\cal J}^{-\,\mu}_{i,<}$, ${\cal J}^{-\,\mu}_{i,>}$, and ${\cal J}^{-\,\mu}_{i,0}$
        by replacing $\phi^{k_{ij};+}_{ij}(x^\mu)$ with
        $-\phi^{\ell_{ij};-}_{ij}(x^\mu)-\varepsilon$.
    \EndFor
    \State Construct the regularized partial Heaviside problem at $x^{\mu}$
    with the index sets
    $\{{\cal J}^{\pm\,\mu}_{i,<}, {\cal J}^{\pm\,\mu}_{i,>}, {\cal J}^{\pm\,\mu}_{i,0}\}_{i \in [I]}$ as in \eqref{eq:partial HSCOP}
    and solve its corresponding IP problem to obtain the globally optimal solution
    $(x^{\mu+1},\gamma_{\mu+1})$ with optimal value denoted by 
    $\mathrm{obj}_{\,\mu+1}$.
    \If{$\mathrm{obj}_{\mu+1} = \mathrm{obj}_{\mu}$}
        \State $r_{\mu+1} \gets \min\{r_{\mu} + r_{\Delta}, r_{\max}\}$
        \State $\tilde{\mu} \gets \tilde{\mu} + 1$
   \Else
        \State $r_{\mu+1} \gets r_\mu$
        \State $\tilde{\mu} \gets 0$
    \EndIf
    \State $\mu \gets \mu + 1$
\EndWhile
\State \Return $\mathrm{obj}_{\mu}$ and 
$x^{\mu}$
\end{algorithmic}
\end{algorithm}

\section{Extension: the M-HSCOP} \label{sec:MPA-HSCOP}

The multiplicative Heaviside composite function refers to the
multiplication of an open Heaviside composite function and a 
closed Heaviside composite function. Motivated by various 
applications in  multiclass classification to be presented in 
Section~\ref{sec:multiclass classification}, we consider a class of 
multiplicative Heaviside 
composite optimization problems (M-HSCOP), which we denote by
$( \, \mbox{\bf P}_{\rm MHS} \, )$,
with the general mathematical formulation as follows:
\begin{equation} \label{eq:M_HSCOP}
\begin{array}{ll}
\underset{x\in P}{\textbf{maximize}} &
\theta_{\text{MHS}}(x) \triangleq c(x) + \displaystyle{
\sum_{j=1}^{J_0}
} \, \psi_{0j} \, \mathbf{1}_{[ \, 0,\infty )}(\phi_{0j}(x))
\, \mathbf1_{( \, 0,\infty )}(\varphi_{0j}(x)) \\ [0.1in]
\textbf{subject to} 
& A_{ i \bullet } x + \displaystyle{
\sum_{j=1}^{J_i}
} \, \psi_{ij} \, \onebld_{\left[ \, 0 , \infty \right)} 
(\phi_{ij}(x)) \, \mathbf{1}_{(0,\infty)}(\varphi_{ij}(x)) 
\, \geq \, \eta_i, \epc  \forall  i \in [ \, I \, ].
\end{array}
\end{equation}
Each product can be reformulated as the difference of two Heaviside composite 
functions, given below:
    \begin{equation}\label{eq:MHSCOP-approx-by-diff}
        \begin{array}{l}
            \onebld_{[ \, 0,\infty )}(\phi_{ij}(x))
        \, \mathbf{1}_{( \, 0,\infty )}(\varphi_{ij}(x))  =\, 
        \onebld_{[ \, 0,\infty )}(\phi_{ij}(x)) 
        - \onebld_{[ \, 0,\infty )} ( \phi_{ij}(x) )  \cdot \onebld_{[ \, 0,\infty )} ( -\varphi_{ij}(x) )  
        \\ [0.1in] 
        \geq \,
        \onebld_{[ \, 0,\infty )}(\phi_{ij}(x)) 
        - \onebld_{( \, -\varepsilon,\infty )} ( \phi_{ij}(x) )  \cdot \onebld_{( \, -\varepsilon,\infty )} ( -\varphi_{ij}(x) ) 
        \end{array}
    \end{equation}
    with the $\varepsilon$-approximating formulation in the last inequality. Thus the approximation theory and algorithms in 
Sections~\ref{sec:approx formulation} and \ref{sec:ISA-PIP} are readily 
applicable.  As it turns out, there is another approximation strategy, which 
yields a tighter approximation error. Specifically, we consider the approximation for the term $\onebld_{[ \, 0,\infty )}(\phi_{ij}(x))
        \, \mathbf{1}_{( \, 0,\infty )}(\varphi_{ij}(x)) $ with positive coefficient  first, and the argument for a negative coefficient is similar. The product $\onebld_{[ \, 0,\infty )}(\phi_{ij}(x))
        \, \mathbf{1}_{( \, 0,\infty )}(\varphi_{ij}(x)) $ is lower bounded by $\onebld_{[ \, 0,\infty )}(\phi_{ij}(x))
    \, \mathbf{1}_{[ \, \varepsilon,\infty )}(\varphi_{ij}(x))$ with the latter satisfying 
    \[ \begin{array}{ll}
     \onebld_{[ \, 0,\infty )}(\phi_{ij}(x))
    \, \mathbf{1}_{[ \, \varepsilon,\infty )}(\varphi_{ij}(x)) & = 
    \onebld_{[ \, 0,\infty )}(\phi_{ij}(x)) - \onebld_{[ \, 0,\infty )}(\phi_{ij}(x))\cdot \onebld_{( \, -\varepsilon,\infty )}(-\varphi_{ij}(x))\\
    & \geq \onebld_{[ \, 0,\infty )}(\phi_{ij}(x)) 
        - \onebld_{( \, -\varepsilon,\infty )} ( \phi_{ij}(x) )  \cdot \onebld_{( \, -\varepsilon,\infty )} ( -\varphi_{ij}(x) ),
    \end{array}\]
   which thus yields a tighter $\varepsilon$-approximation. Denoted 
$( \, \mbox{\bf P}_{\rm MHS}^{\, \varepsilon} \, )$, the resulting
$\varepsilon$-approximating problem is as follows:

\begin{equation} \label{eq:M_HSCOP_approx}
\begin{array}{ll}
\underset{x\in P}{\textbf{maximize}} &
\theta_{\text{MHS}}^{\, \varepsilon}(x) \, \triangleq \, 
c(x) + \displaystyle{
\sum_{j=1}^{J_0}
} \, \psi_{0j}^+ \, \onebld_{[ \, 0,\infty )}(\phi_{0j}(x))
\, \mathbf{1}_{[ \, \varepsilon,\infty )}(\varphi_{0j}(x)) 
\\ [0.1in]
& \hspace{0.5in} - \,
\displaystyle{
\sum_{j=1}^{J_0}
} \, \psi_{0j}^- \, 
\mathbf{1}_{( \, -\varepsilon,\infty )}(\phi_{0j}(x)) \,
\onebld_{( \, 0,\infty )}(\varphi_{0j}(x)) \\ [0.1in]
\textbf{subject to} 
& A_{ i \bullet } x + \displaystyle{
\sum_{j=1}^{J_i}
} \, \psi_{ij}^+ \, \onebld_{[ \, 0,\infty )}(\phi_{ij}(x))
\, \mathbf{1}_{[ \, \varepsilon,\infty )}(\varphi_{ij}(x)) 
\\ [0.1in]
& \hspace{0.5in} - \,
\displaystyle{
\sum_{j=1}^{J_i}
} \, \psi_{ij}^- \, 
\mathbf{1}_{( \, -\varepsilon,\infty )}(\phi_{ij}(x)) \,
\onebld_{( \, 0,\infty )}(\varphi_{ij}(x)) \, \geq \, \eta_i, 
\epc \forall \, i \, \in \, [ \, I \, ].
\end{array}
\end{equation}
Besides the upper semicontinuity of the objective
function and the closedness of the feasible set, the above 
approximation also ensures the validity of 
the monotonic properties of 
$\theta_{\rm MHS}^{\, \varepsilon}(x)$:
\begin{equation}
\theta_{\rm MHS}(x) \, \geq \, 
\theta_{\rm MHS}^{\, \varepsilon}(x) \, \geq \, 
\theta^{\varepsilon^{\, \prime}}_{\rm MHS}(x), \epc \forall \, 
\varepsilon \, \in \, ( \, 0,\varepsilon^{\, \prime} \, ).
\label{eq:monotonic_theta_epsilon}
\end{equation}

Define the index sets associated with $\bar x$
\[ 
\begin{array}{l}
{\cal J}_{i,0}^{+}(\bar{x}) \, \triangleq \,
\left\{ \, j \in [J_i] \, \mid \, \psi_{ij} \, > \, 0;\,
\varphi_{ij}(\bar{x}) = 0 \right\}, \,\,  
{\cal J}_{i,>}^{+}(\bar{x}) \, \triangleq \,
\left\{ \, j \in [J_i] \, \mid \, \psi_{ij} \, > \, 0;\,
\varphi_{ij}(\bar{x}) > 0 \right\}, \\ [0.1in]
{\cal J}_{i,<}^{+}(\bar{x}) \, \triangleq \, 
\left\{ \, j \in [J_i] \, \mid \,  \psi_{ij} \, > \, 0;\,
\varphi_{ij}(\bar{x}) < 0 \right\} \\ [0.1in]
{\cal J}_{i,0}^{+}(\bar{x}) \, \triangleq \,
\left\{ \, j \in [J_i] \, \mid \, \psi_{ij} \, < \, 0;\,
\phi_{ij}(\bar{x}) = 0 \right\}, \,\,  
{\cal J}_{i,>}^{+}(\bar{x}) \, \triangleq \,
\left\{ \, j \in [J_i] \, \mid \, \psi_{ij} \, < \, 0;\,
\phi_{ij}(\bar{x}) > 0 \right\}, \\ [0.1in]
{\cal J}_{i,<}^{+}(\bar{x}) \, \triangleq \, 
\left\{ \, j \in [J_i] \, \mid \,  \psi_{ij} \, < \, 0;\,
\phi_{ij}(\bar{x}) < 0 \right\}
\end{array}
\] 
Similar to Proposition~\ref{pr:N&S II locmax AHSCOP}, we
have the following result.

\begin{proposition}\label{pr:equal super MHSCOP} \rm
Let $\bar{x} \in P$ be given.  Let each $\phi_{ij}$ and $\varphi_{ij}$ be 
continuous.  The following two statements hold.

\gap

\noindent {\bf (A)}
If $\bar x$ is a local maximizer of (\ref{eq:M_HSCOP}), then 
$\bar x$ is a local maximizer of (\ref{eq:M_HSCOP_approx}) 
for all $\varepsilon\in (0,\bar{\varepsilon}]$.

\noindent {\bf (B)} Suppose $\bar x$ satisfies the LSI condition:
$\varphi_{ij}$ is nonpositive near $\bar{x}$ for all 
$j \in {\cal J}_{i,0}^{+}(\bar{x})$ and all $i = 0, 1, \cdots, I$; 
$\phi_{ij}$ is nonnegative near $\bar{x}$ for all 
$j \in {\cal J}_{i,0}^{-}(\bar{x}) $ and all $i = 0,1,\cdots, I$.
If $\bar x$ is a local maximizer of (\ref{eq:M_HSCOP_approx}),
then $\bar x$ is a local maximizer of (\ref{eq:M_HSCOP}).
\end{proposition}

\begin{proof}
(A) Similar to property (iii) in Lemma \ref{lm:epsilon_approximation}, 
the closed Heaviside function $\onebld_{[ \, \varepsilon,\infty)}(\bullet)$ 
has the property that for any scalar $t_*\ne 0$, a scalar $\varepsilon_*>0$ 
and a neighborhood ${\cal T}_*$ of $t_*$ exist such that
    \begin{equation} \label{eq:limit indicator closed}
    \onebld_{[ \, \varepsilon,\infty )}(t) \, = \,
    \onebld_{( \, 0, \infty )}(t_*), \epc \forall \, ( t,\varepsilon ) 
    \in {\cal T}_* \times ( 0,\varepsilon_* );
    \end{equation}
These two properties yield that there must exist a neighborhood ${\cal N}$ 
of ${\bar x}$ and a scalar $\bar{\varepsilon} > 0$ (depend on $\bar x$) 
such that for all $x\in {\cal N}$, all 
$\varepsilon\in (0,\bar{\varepsilon}]$, and all $i=1,\ldots,I$:
\begin{equation} \label{eq:local_sign_invariance_MHSCOP}
\begin{array}{l}
\onebld_{( \, - \varepsilon,\infty )}(\phi_{ij}(x))
\, = \,  \onebld_{[ \, 0,\infty )}(\phi_{ij}(x)) 
\, = \,  \onebld_{[ \, 0 ,\infty )}(
\phi_{ij}( \overline { x } ) ), \epc \forall \, 
j \, \in \, \mathcal{J}_{i,>}^{-}(\bar{x}) \cup 
\mathcal{J}_{i,<}^{-}(\bar{x}) \\
\onebld_{[ \, \varepsilon,\infty )}(\varphi_{ij}(x))
\, = \,  \onebld_{( \, 0,\infty )}(\varphi_{ij}(x)) 
\, = \,  \onebld_{( \, 0 ,\infty )}(
\varphi_{ij}( \overline { x } ) ), \epc \forall \, 
j \, \in \, \mathcal{J}_{i,>}^{+}(\bar{x}) \cup 
\mathcal{J}_{i,<}^{+}(\bar{x}) 
\end{array}
\end{equation}
By the same argument as in the proof of 
Proposition~\ref{pr:N&S II locmax AHSCOP}, we obtain (A).

\noindent (B) Suppose $\bar x$ is optimal for $( \, \mbox{\bf P}_{\rm MHS}^{\, \varepsilon} \, )$ in ${\cal N}$. Similar to Proposition \ref{pr:N&S II locmax AHSCOP}, we need to show the inclusion $X_{\rm MHS} \cap {\cal N} \subseteq X_{\rm MHS}^{\varepsilon} \cap {\cal N}$. Let $x\in X_{\rm MHS}\cap  {\cal N}$; we have, for $i=0,1,\ldots,I$, 
\begin{equation}
    \begin{array}{l}
        \displaystyle{
        \sum_{j=1}^{J_i}
        } \, \psi_{ij}^+ \, \onebld_{[ \, 0,\infty )}(\phi_{ij}(x))
        \, \mathbf{1}_{( \, 0,\infty )}(\varphi_{ij}(x)) 
        \\ [0.1in]
        =\, \underbrace{\displaystyle{
        \sum_{j\in {\cal J}_{i,0}^+(\bar x)} 
        } \, \psi_{ij}^+ \, \onebld_{[ \, 0,\infty )}(\phi_{ij}(x))
        \, \mathbf{1}_{[ \, \varepsilon,\infty )}(\varphi_{ij}(x))
        }_{\rm by\ LSI\ condition}
        +  \underbrace{\displaystyle{
        \sum_{j\notin {\cal J}_{i,0}^+ (\bar x)} 
        } \, \psi_{ij}^+ \, \onebld_{[ \, 0,\infty )}(\phi_{ij}(x))
        \, \mathbf{1}_{[ \, \varepsilon,\infty )}(\varphi_{ij}(x))
        }_{\rm by \ (\ref{eq:local_sign_invariance_MHSCOP})}
        \\ [0.3in]
        \displaystyle{
        \sum_{j=1}^{J_i}
        } \, \psi_{ij}^- \, \onebld_{[\,0,\infty )}(\phi_{ij}(x))
        \, \mathbf{1}_{( \, 0,\infty )}(\varphi_{ij}(x)) 
        \\ [0.1in]
        =\, \underbrace{\displaystyle{
        \sum_{j\in {\cal J}_{i,0}^- (\bar x)} 
        } \, \psi_{ij}^- \, \onebld_{( \, -\varepsilon,\infty )}(\phi_{ij}(x))
        \, \mathbf{1}_{(\, 0 ,\infty )}(\varphi_{ij}(x))
        }_{\rm by\ LSI\ condition}
        +  \underbrace{\displaystyle{
        \sum_{j\notin {\cal J}_{i,0}^-(\bar x)} 
        } \, \psi_{ij}^- \, \onebld_{( \, -\varepsilon,\infty )}(\phi_{ij}(x))
        \, \mathbf{1}_{(\,0,\infty )}(\varphi_{ij}(x))
        }_{\rm by \ (\ref{eq:local_sign_invariance_MHSCOP})}
    \end{array}
\end{equation}
By the same argument as in the proof of 
Proposition~\ref{pr:N&S II locmax AHSCOP}, we obtain (B). 
\end{proof}

Noticing that
\[ \begin{array}{l}
\onebld_{[ \, 0,\infty )}(\phi_{ij}(x))
\, \mathbf{1}_{[ \, \varepsilon,\infty )}(\varphi_{ij}(x)) = \onebld_{[ \, 0,\infty )}(\min\{ \phi_{ij}(x), \,
\varphi_{ij}(x) - \varepsilon \}) \\[0.1in]
\mathbf{1}_{( \, -\varepsilon,\infty )}(\phi_{ij}(x)) \,
\onebld_{( \, 0,\infty )}(\varphi_{ij}(x)) 
\, = \,
\mathbf{1}_{( \, -\varepsilon,\infty )}(\min\{ \phi_{ij}(x), \,
\varphi_{ij}(x) + \varepsilon\} ) ,
\end{array}\]
we see that the $\varepsilon$-approximating problem 
(\ref{eq:M_HSCOP_approx}) is equivalent to
\[
\begin{array}{ll}
\underset{x\in P}{\textbf{maximize}} &
\theta_{\text{MHS}}^{\, \varepsilon}(x) \, \triangleq \, 
c(x) + \displaystyle{
\sum_{j=1}^{J_0}
} \, \psi_{0j}^+ \, \onebld_{[ \, 0,\infty )}(\min\{
\phi_{0j}(x), \, \varphi_{0j}(x) - \varepsilon\} ) 
\\ [0.1in]
& \hspace{0.5in} - \,
\displaystyle{
\sum_{j=1}^{J_0}
} \, \psi_{0j}^- \, 
\mathbf{1}_{( \, -\varepsilon,\infty )}(\min\{ \phi_{0j}(x), \,
\varphi_{0j}(x) + \varepsilon\} ) \\ [0.1in]
\textbf{subject to} &  A_{ i \bullet } x + \displaystyle{
\sum_{j=1}^{J_i}
} \, \psi_{ij}^+ \, \onebld_{[ \, 0,\infty )}(\min\{ 
\phi_{ij}(x), \, \varphi_{ij}(x) - \varepsilon\} ) \\[0.15in]
&- \,
\displaystyle{
\sum_{j=1}^{J_i}
} \, \psi_{ij}^- \, 
\mathbf{1}_{( \, -\varepsilon,\infty )}(\min\{ \phi_{ij}(x),\,
\varphi_{ij}(x) + \varepsilon\} ) \, \geq \, \eta_i, \epc \mbox{for all $i = 1, \cdots, I$} 
\end{array}
\]
which is in the form of problem 
(\ref{eq:partitioned I AHSCOP}).  Thus the
results in 
Sections~\ref{sec:PA decomp} and
\ref{sec:ISA-PIP} are applicable;
in particular, the IDSA-PIP algorithm can be applied.  We omit the details.

\section{Case Studies for Multiclass Classification}
\label{sec:multiclass classification}
We conduct computational experiments for two types 
of multiclass classification models: scored-based model in 
Section~\ref{sec:score based classification} and tree-based model 
in Section~\ref{sec:tree based classification}. The goal of the 
classification 
is to maximize the accuracy subject to precision constraints. 
Our experimental results show that (1) for both models, IDSA-PIP 
significantly 
outperforms Full MIP (directly solving (\ref{eq:partitioned I AHSCOP}) by 
its full MIP formulation) in both feasibility and optimality using much less 
computational time; and (2) tree-based classification model solved by IDSA-PIP 
method Pareto-dominates other state-of-the-art decision tree models in 
terms of accuracy and precision.
All numerical experiments are performed on a Dell PowerEdge R940xa server equipped with Intel(R) Xeon(R) Platinum 8260 CPUs (96 physical cores, 192 logical processors, 2.40 GHz) and 1 TB RAM. Unless otherwise specified\footnote{The implementation of \texttt{BinOCT} (see Section \ref{sec:tree based results}) is adapted from \cite{verwer2019learning} coded in Python 3.10 with the Cplex 22.1.2.0 package.}, all computations are coded in Python 3.11 with Numpy, Pandas, Scipy, Gurobipy 11.0.3 packages. 

\subsection{Score-based multiclass classification}
\label{sec:score based classification}

Let $\{(X^{s},y_{s})\}_{s=1}^{N}$ be the dataset of size $N$ with 
$y_{s} \in [\,J\,]\triangleq\{1,2,\ldots,J\}$ and $X^{s}\in\mathbb{R}^p$. 
For each class $j\in[\,J\,]$, we assign a parameterized score function 
$X^s \to w_{j}^{\top}X^{s}+b_{j}$, where 
$(w_{j}, b_{j}) \in \mathbb{R}^{p+1}$ are parameters to be determined.
Sample $s$ is classified to class $j$ if and only if its score satisfies the
following conditions:
\[
\begin{array}{l}
w_{j}^{\top} X^{s}+b_{j} > w_{ m }^{\top}X^{s} +b_{ m }, \epc 
\forall\, 1\le m < j \\
w_{j}^{\top}X^{s}+b_{j} \ge w_{ n }^{\top}X^{s}+b_{ n }, \epc 
\forall\, j < n \le J,
\end{array}
\]
or equivalently, 
\[
j = \min\left\{\underset{m\in [\, J\, ]}{\arg\max} \left\{w_{ m }^{\top}X^{s} +b_{ m }\right\} \right\}
\]
meaning that the sample is classified into the class with the highest score, and in cases where multiple classes achieve the same highest score, the class with the smallest index is selected.
For the sake of concise expression, we define 
\begin{equation}
\label{eq:definition_h}
    \begin{array}{l}
    h_{m,n}(W,\mathbf{b};x)\triangleq (w_{m}-w_{n})^{\top}x+(b_m-b_n),\epc \forall\, m,\ n\in [\, J\,] \\ [0.1in] 
    W=\begin{bmatrix}
        w_1 & w_2 & \cdots & w_{J}
    \end{bmatrix}\in \mathbb{R}^{p\times J},\ \mathbf{b}\in\mathbb{R}^{J}.
    \end{array}
\end{equation} 
There are multiple commonly used metrics in classification problems, which all lead to the multiplication of Heaviside composite functions:
\begin{itemize}
\item {\bf Accuracy of classification:}
\begin{equation}
\begin{array}{l}
\displaystyle{
\frac{\mbox{\#\ of samples that are correctly classified}}{N}
} \\ [0.1in]
= \, \displaystyle{
\frac{1}{N}
} \, \displaystyle{
\sum_{s=1}^{N}
} \, \left[ \, \mathbf{1}_{[0,\infty)}\left( \,
\min\limits_{y_{s}<j\le J}h_{y_{s},j}(W,\mathbf{b};X^{s}) \, \right) 
\, \right] \, \left[ \, 
\mathbf{1}_{(0,\infty)}\left( \, 
\underset{1\le j<y_{s}}{\min}h_{y_{s},j}(W,\mathbf{b};X^{s}) \, \right)
\, \right]
\end{array}  
\end{equation}

\item {\bf Precision of class $j$:}
\begin{equation}
\begin{array}{l}
\displaystyle{
\frac{\mbox{\# of samples that are correctly classified as label $j$}}{
\mbox{\# of samples that are classified as label $j$}}
} \\ [0.1in]
= \, \displaystyle{
\frac{\displaystyle{
\sum_{s=1}^{N}
} \, \left[ \, \mathbf{1}\{y_{s} = j\} \, \mathbf{1}_{[0,\infty)} \,
\left( \, \min_{j<m\le J}h_{j,m}(W,\mathbf{b};X^{s}) \, \right) 
\, \right] \, \left[ \, \mathbf{1}_{(0,\infty)}\left( \,
\underset{1\le n<j}{\min}h_{j,n}(W,\mathbf{b};X^{s}) \, \right) \, \right]}{
\displaystyle{
\sum_{s=1}^N
} \, \left[ \, \mathbf{1}_{[ \, 0,\infty )}\left( \, \min\limits_{j<m\le J} 
h_{j,m}(W,\mathbf{b};X^{s}) \, \right) \, \right] \, 
\left[ \, \mathbf{1}_{(0,\infty)}\left( \, \underset{1\le n<j}{\min} 
h_{j,n}(W,\mathbf{b};X^{s}) \, \right) \, \right]}
}
\end{array}
\end{equation}

\item {\bf Recall of class $j$:}
\begin{equation}
\begin{array}{l}
\displaystyle{
\frac{\mbox{\# of samples that are correctly classified as label $j$}}{
\mbox{\# of samples with label $j$}}
} \\ [0.1in]
= \, \displaystyle{
\frac{\displaystyle{
\sum_{s=1}^N
} \, \left[ \, \mathbf{1}\{y_{s} = j\} \,
\mathbf{1}_{[0,\infty)} \, \left( \, 
\min_{j<m\le J}h_{j,m}(W,\mathbf{b};X^{s}) \, \right) \, \right] \,
\left[ \, \mathbf{1}_{(0,\infty)}\left( \,
\underset{1\le n<j}{\min}h_{j,n}(W,\mathbf{b};X^{s}) \, \right) \, \right]}{
\displaystyle{
\sum_{s=1}^{N}
} \, \mathbf{1}\{y_{s} = j\}}
} .
\end{array}
\end{equation}
\end{itemize}
 
In learning the classification model, one may use a combination of the 
above metrics either in the objective or constraints, depending on the 
contexts.  The reference \cite{grandini2008metrics} defines the macro 
average precision as $\displaystyle{
\frac{1}{J}
} \, \displaystyle{
\sum_{j\in[\,J\,]}
} \, \mbox{precision}_j$, and the micro average precision as 
$\displaystyle{
\frac{1}{N}
} \, \displaystyle{
\sum_{j\in[\,J\,]}
} \, \mbox{precision}_j$.  If (average) precision is set as the objective, 
one may resort to the new methodology on fractional Heaviside programming
that is presently being developed in \cite{LiuPang2026}.
In the following, we consider the maximization of the accuracy with 
precision constraints for a subset of labels $\mathcal{J}\subseteq[\, J\,]$, 
each controlled by a given threshold $\beta_{j} \in (0,1)$.  
Noticing that without margin consideration, the accuracy under Heaviside 
formulation is positive homogeneous and thus may lead to numerical issues 
in maximization.  Thus, we employ a modified accuracy with a soft margin
as the objective and reformulate the precision constraint, yielding the 
following formulation of the classification problem:
\begin{equation}\label{eq:P_score_classitication}
\begin{array}{l}
\underset{(W,\mathbf{b})\in P}{\textbf{maximize}} \
\displaystyle\frac{1}{N} 
\displaystyle{
\sum_{s=1}^{N}
} \ \mathbf{1}_{[0,\infty)}\left( \displaystyle{
\min_{\substack{j\in[\,J\,],\\j\ne y_{s}}}
} \, h_{y_{s},j}(W,\mathbf{b};X^{s})-1 \right) \\ [0.1in]
\textbf{subject to } \mbox{ for all } j \in \, \mathcal J: \\ [3pt]
\displaystyle{
\sum_{s=1}^{N}
} \, \left[ \, \mathbf{1}\{y_{s} = j\} \, \mathbf{1}_{[0,\infty)}\Big( \,
\displaystyle{\min_{j<m\le J}
} \, h_{j,m}(W,\mathbf{b};X^{s}) \, \Big) \right] \, \left[ \, 
\onebld_{(0,\infty)}\Big( \, \displaystyle{
\min_{1\le n< j}
} \, h_{j,n}(W,\mathbf{b};X^{s}) \, \Big) \, \right] \\ 
\hspace{0.2in} -\, \beta_j \, \displaystyle{
\sum_{s=1}^{N}
} \, \left[ \, \mathbf{1}_{[0,\infty)}\Big( \, \displaystyle{
\min_{j<m\le J}
} \, h_{j,m}(W,\mathbf{b};X^{s}) \, \Big) \, \right] \, \left[ \, 
\onebld_{(0,\infty)}\Big( \, \displaystyle{
\min_{1\le n< j}
} \, h_{j,n}(W,\mathbf{b};X^{s}) \, \Big) \, \right] \, \ge \, 0  \\ [0.2in]
\displaystyle{
\sum_{s=1}^{N}
} \, \left[ \, \onebld_{[\, 0,\infty )}\Big( \, \displaystyle{
\min_{j<m\le J}
} \, h_{j,m}(W,\mathbf{b};\, X^{s}) \, \Big) \, \right] \, \left[ \,
\onebld_{(0,\infty)}\Big( \, \displaystyle{
\min_{1\le n< j}
} \, h_{j,n}(W,\mathbf{b};X^{s}) \, \Big) \, \right] \, \geq \, 1,
\end{array}
\end{equation}
where $P\triangleq\{(W,\mathbf{b})\in\mathbb{R}^{(p+1)\times J} \, \mid \, 
\|w_{j}\|_1 \le \tau,\ |b_{j}|\le \tau,\ \forall\, j\in [\, J\,]\}$, and 
the second 
inequality equivalently ensures that the denominator in the precision 
is nonzero by the  discretely valued summation term.  The above formulation 
contains multiple products of an open Heaviside function and a closed 
Heaviside function, each can be approximated by a single closed
Heaviside function.  Indeed, for $(j,s) \in \mathcal J \times [\, N\, ]$, 
the product
\[
\left[ \, \mathbf{1}_{[ \, 0,\infty)}\left( \, \displaystyle{
\min_{j<m\le J}
} \, h_{j,m}(W,\mathbf{b};X^{s}) \, \right) \, \right] \, \left[ \,
\onebld_{( \, 0,\infty)}\left( \, \displaystyle{
\min_{1\le n< j}} h_{j,n}(W,\mathbf{b};X^{s}) \, \right) \, \right]
\]
is lower approximated by:
\begin{equation*}
\begin{array}{l}
\left[ \, \mathbf{1}_{[0,\infty)} \left( \, \displaystyle{
\min_{j<m\le J}
} \, h_{j,m}(W,\mathbf{b};X^{s}) \, \right) \, \right] \, \left( \,
\onebld_{[ \, \varepsilon,\infty)}\left( \, \displaystyle{
\min_{1\le n< j}
} \, h_{j,n}(W,\mathbf{b};X^{s}) \, \right) \, \right] \\ [0.2in]
\epc \, = \, \onebld_{[ \, 0,\infty)}\left( \, \displaystyle{
\min_{1\le n <j<m \le J}
} \, \left\{ \, h_{j,n}(W,\mathbf{b};X^{s}) -
\varepsilon,h_{j,m}(W,\mathbf{b};X^{s}) \, \right\} \, \right);
\end{array}
\end{equation*}
and similarly upper approximated  by:
\[ 
\onebld_{( \, -\varepsilon,\infty)} \left( \displaystyle{
\min_{1\le n <j<m \le J} 
} \, \left\{ \, h_{j,n}(W,\mathbf{b};X^{s}) - 
\varepsilon,h_{j,m}(W,\mathbf{b};X^{s}) \, \right\} \, \right).
\] 
To construct an $\varepsilon$-approximating problem 
for \eqref{eq:P_score_classitication}, the above two lower/upper 
approximations are used depending on the sign of coefficients.  Details
are not repeated.
  
\subsubsection{Numerical results}
To evaluate the performance of the Iterative Decomposed Shrinkage Algorithm(IDSA) and particularly the effectiveness of decomposition and shrinkage techniques, we implement the following methods for comparisons:
\begin{enumerate}
    \item \textbf{Full MIP}: with fixed $\varepsilon$, solve 
    $\left(\mathbf{P}_{\mathrm{AHS}}^{\varepsilon}\right)$ by the 
    IP solver {\sc Gurobi} 
    to its full MIP formulation with 3600s time limit.
    
    \item \textbf{PIP}: with fixed $\varepsilon$, solve $\left(\mathbf{P}_{\mathrm{AHS}}^{\varepsilon}\right)$ by the PIP Algorithm~\ref{app:PIP} without the decomposition of inner PA functions.
    
    \item \textbf{ISA-PIP}: A variant of Algorithm~\ref{alg:IDSA-PIP}, 
    in which the inner loop computation is by PIP algorithm applied to 
    $\left(\mathbf{P}_{\mathrm{AHS}}^{\varepsilon_\nu}\right)$ without the decomposition of the inner PA functions.  
    
   \item \textbf{IDSA-PIP}: A variant of Algorithm~\ref{alg:IDSA-PIP} 
   in which the inner loop computation is by the 
   PIP~Algorithm~\ref{alg:pip_details} 
   applied to a decomposed approximating subproblem 
   $\left(\mathbf{P}_{\rho}^{\varepsilon_\nu;\mathbf{k}_\nu,\boldsymbol{\ell}_\nu }\right)$ with an index pair $(\mathbf{k}_\nu,\boldsymbol{\ell}_\nu)$ arbitrarily selected from ${\cal M}^{\boldsymbol{\delta}}(x^\nu)$.
\end{enumerate}
 
We implement the above methods in 16 instances with six datasets respectively.
Detailed set-up and results are provided in 
Appendix~\ref{app:multi-setting}. In comparison with Full MIP, the PIP 
variants achieve higher objectives on 61.46\% (PIP), 70.83\% (ISA-PIP), 
and 94.79\% (IDSA-PIP) of the instances, with average improvements of 
1.97\% (PIP), 2.20\% (ISA-PIP), and 1.95\% (IDSA-PIP) in training accuracy, 
and 1.93\% (PIP), 2.18\% (ISA-PIP), and 1.98\% (IDSA-PIP) in test accuracy, 
respectively, illustrating the effectiveness of shrinkage and decomposition 
techniques. Figure~\ref{fig:multi-performance-comparison} summarizes the 
performance for six datasets in terms of feasibility and objective value. 
Results show the substantial advantage of IDSA-PIP method among 16 instances 
for all 6 datasets. 
Specifically, Figure \ref{fig:multi-performance-comparison}(a) computes the percentage of problem instances for which the method finds a solution meeting the precision constraint  among the four methods.  
It shows that the PA decomposition and shrinkage techniques enhance precision performance of the PIP algorithm with more than 10\%-50\% instances satisfying precision constraints. In particular, IDSA-PIP achieves feasibility under  precision constraints in all instances. For the \textit{robo} dataset which has  small $J$ and $p$, all methods fully satisfy  precision constraints. For datasets involving a large $J$ (e.g., \textit{segm}, $ J=7$; \textit{fish}, $ J=9$), PIP-based methods significantly outperform Full MIP in terms of precision metric. For small $J$ and large $p$, IDSA-PIP significantly outperforms other methods for large sample size (e.g., \textit{wave}, $p=22,\ J=3$), while  for small data size (e.g., \textit{vehi}, \textit{wine}), PIP-based methods offer no significant precision advantage over Full MIP.  Figure \ref{fig:multi-performance-comparison}(b) computes percentages of problem instances where the method achieves the highest objective value (out-of-margin accuracy).

\begin{figure}[t]
\centering
\begin{minipage}[t]{0.45\textwidth}
\centering
\includegraphics[width=\textwidth]{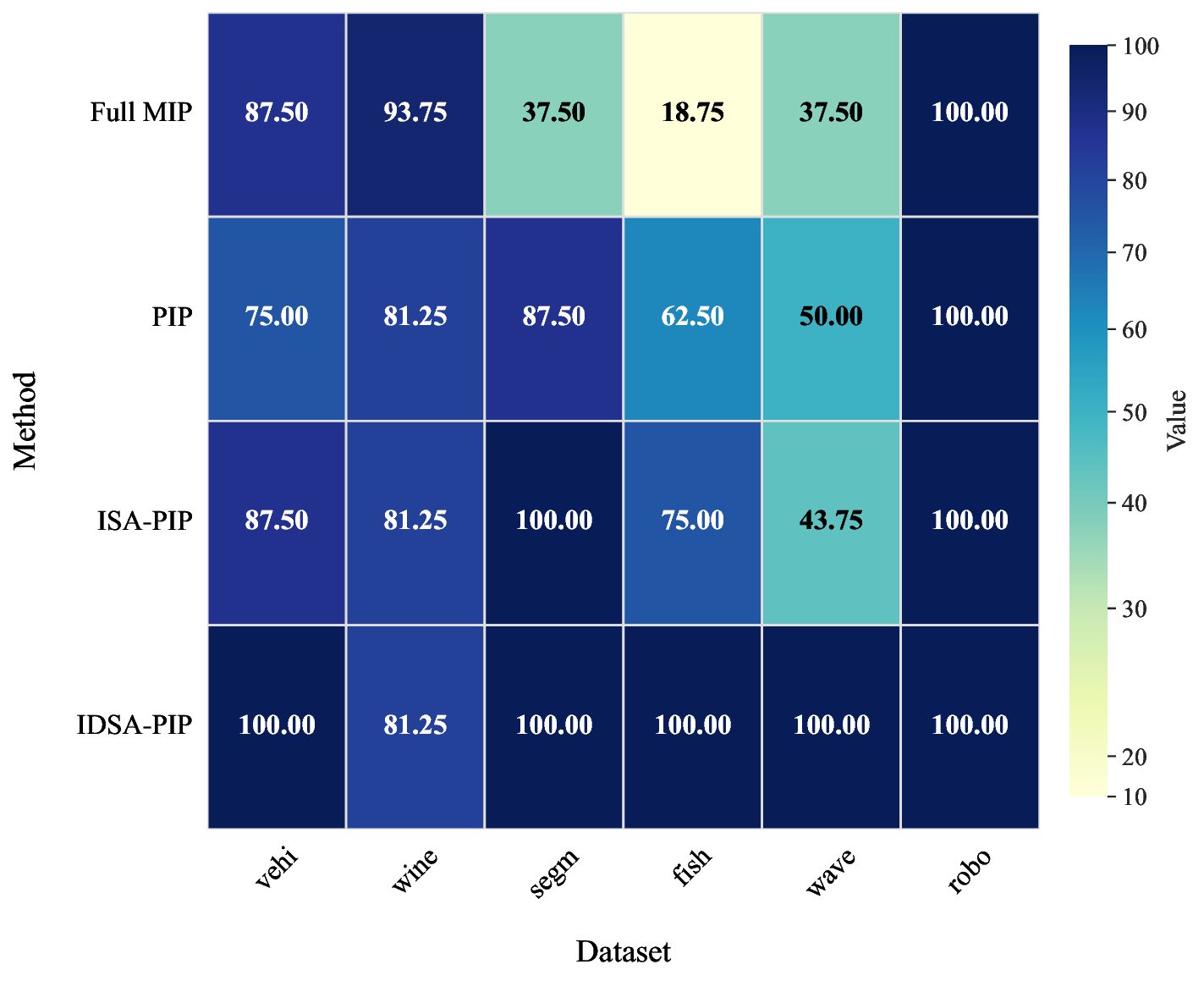}
\\[2pt]
{\small (a) Instance proportion with feasibility guarantee}
\end{minipage}
\hfill
\begin{minipage}[t]{0.45\textwidth}
\centering
\includegraphics[width=\textwidth]{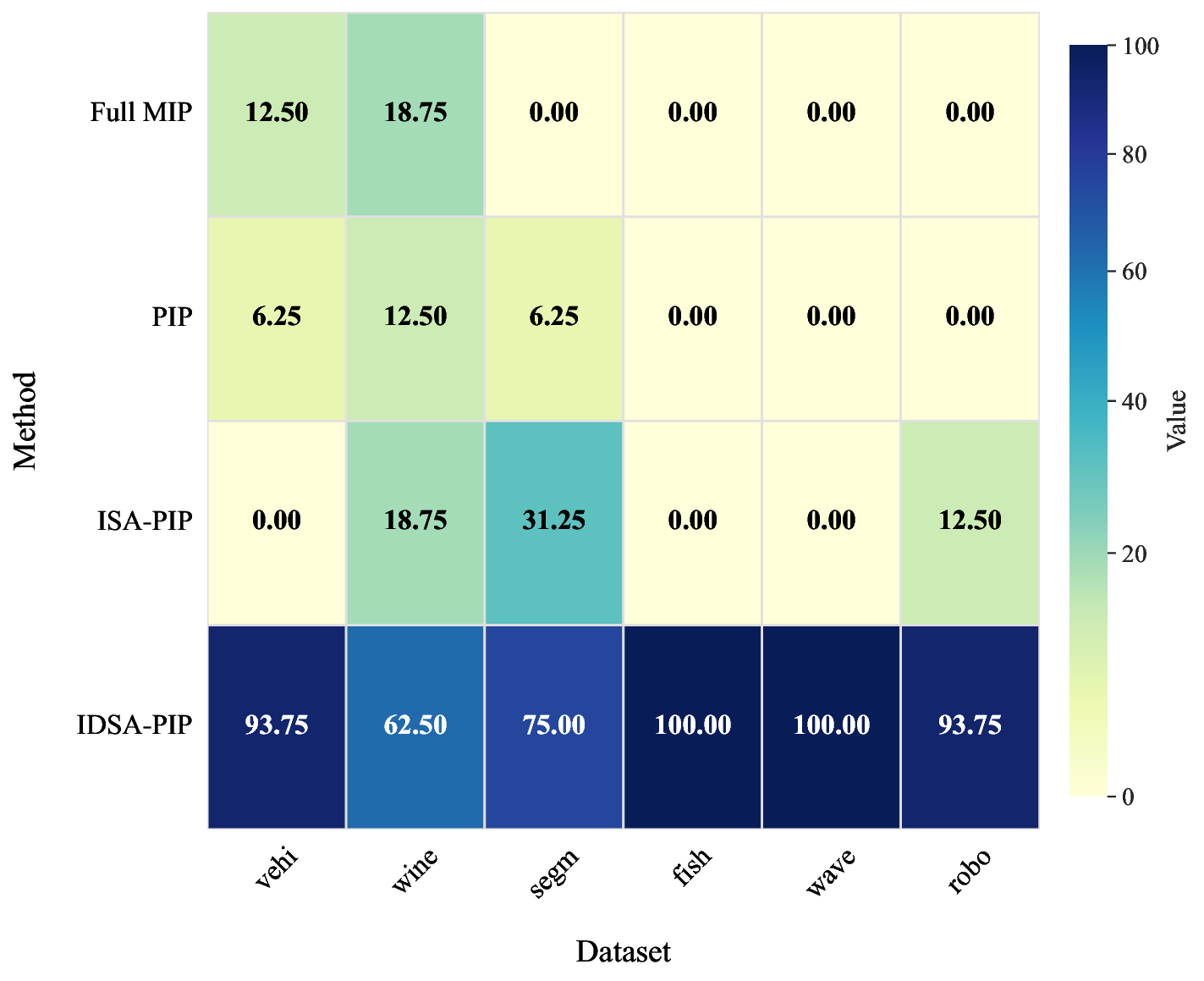}
\\[2pt]
{\small (b) Instance proportion with best objective value}
\end{minipage}
\caption{Performance in terms of feasibility and objective.}
\label{fig:multi-performance-comparison}
\end{figure}

We illustrate the time benefits of PIP-based methods for instances with better objective values than Full MIP method in Figure \ref{fig:multi-time-ratio}. The running time benefit of IDSA-PIP is substantial, with nearly 70\% of instances with 50\% less running time, and nearly 80\% of instances with 40\% less running time among those with better objective values.     Regarding the test accuracy, IDSA-PIP outperforms Full MIP more significantly on larger-scale instances like \textit{segm, fish}, and \textit{wave}.  Table~\ref{tab:multiclass-test-acc} summarizes the instance proportion for PIP-based methods obtaining  better test accuracy. The full MIP approach shows competitive performance on the \textit{robo} dataset, mainly due to the problem’s simplicity, with fewer dimensions and classes. In summary, the decomposition technique (IDSA-PIP) achieves significant time improvement while maintaining the edge on objective value, precision, and test accuracy.  

\begin{figure}[h]
\label{figure:score_classification_time}
\centering
\begin{minipage}[t]{0.45\textwidth}
\centering
\includegraphics[width=\textwidth]{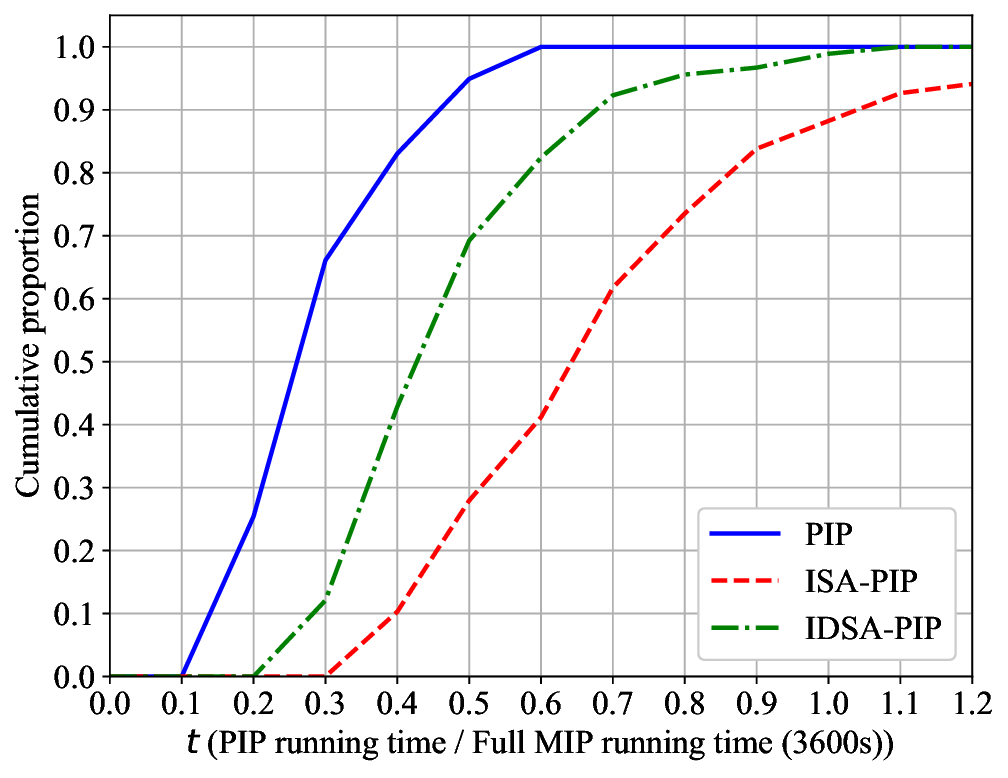}
\\[2pt]
{\small (a) PIP methods versus Full MIP (baseline)}
\end{minipage}
\hfill
\begin{minipage}[t]{0.45\textwidth}
\centering
\includegraphics[width=\textwidth]{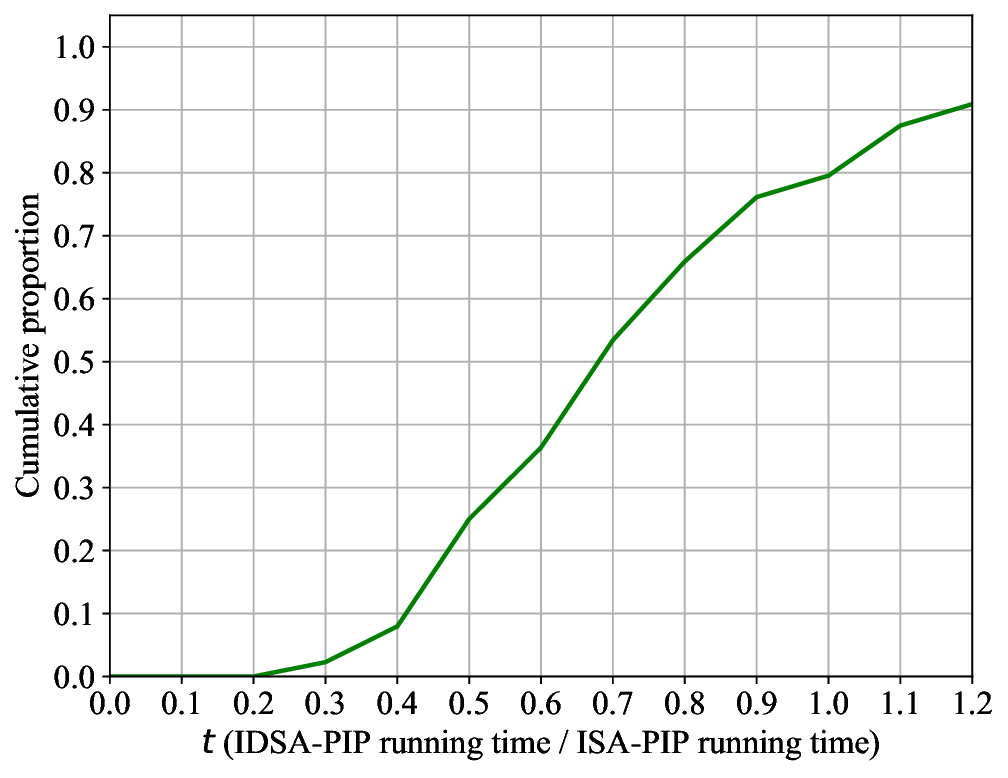}
\\[2pt]
{\small (b) IDSA versus ISA (baseline)}
\end{minipage}
\caption{Cumulative proportion among instances with better objective values versus running time ratio.\\
Note: The \(x\)-axis represents the time ratio (method running time/baseline running time), and the \(y\)-axis represents the cumulative proportion of instances for which the method achieves a better objective value than the baseline within the corresponding time ratio. The baseline method is Full MIP in (a) and ISA-PIP in (b).}
\label{fig:multi-time-ratio}
\end{figure}

\begin{table}[h]
\caption{Instance proportion for obtaining better test accuracy than Full MIP} 
\label{tab:multiclass-test-acc}
\begin{tabular}{@{}lllllll@{}}
\toprule
    & \textit{vehi}    & \textit{wine}    & \textit{segm}    & \textit{fish}    & \textit{wave}    & \textit{robo}    \\ \midrule
PIP      & 31.25\% & 50.00\% & 75.00\% & 62.50\% & 43.75\% & 31.25\% \\
ISA-PIP  & 56.25\% & 62.50\% & 81.25\% & 75.00\% & 31.25\% & 25.00\% \\
IDSA-PIP & 56.25\% & 56.25\% & 87.50\% & 81.25\% & 87.50\% & 12.50\% \\ \bottomrule
\end{tabular}
\footnotetext{Note: A method is considered better than Full MIP if (1) it finds a feasible solution while the baseline does not, or (2) both find feasible solutions and it achieves higher test accuracy. Cases where both methods are infeasible are counted as ties.
}
\end{table}

\subsection{Tree-based classification}
\label{sec:tree based classification}
\subsubsection{Model}
Let $\{(X^s,y_s)\}_{s=1}^N$ be the dataset of size $N$ with $y_s\in [\, J\,]\triangleq \{1,2,\dots,J\}$, $X^s \in \mathbb{R}^p$. With the coefficient matrix $\boldsymbol{a}=\begin{bmatrix}a^1 & a^2 & \cdots & a^{2^D-1}\end{bmatrix}\in \mathbb R^{p\times (2^D-1)}$ and  coefficient vector $\boldsymbol{b}=\begin{bmatrix}b_1 & b_2 & \cdots & b_{2^D-1}\end{bmatrix}^T\in \mathbb{R}^{2^D-1}$, a classification tree with depth $D$ could be constructed as follows. Let ${\cal T}_{\cal B}$ be the set of branching nodes and ${\cal T}_\ell$ be the set of leaf nodes. At each branching node $k\in \mathcal T_{\mathcal B}$, if $(a^k)^\top X^s - b_k \ge 0$, then sample $s$ flows into the right branch, if $(a^k)^\top X^s - b_k < 0$, then sample $s$ flows into the left branch. Each leaf node $t\in \mathcal T_\ell$ is assigned to a class $j_t \in [\, J\,]$. Sample $s$ is classified as class $j$ if it eventually flows into a leaf node $t$ with the assigned class $j_t=j$, and
\[
\begin{array}{l}
(a^k)^\top X^s - b_k \ge 0,\ \forall\, k\in A_R(t), \epc (a^k)^\top X^s -b_k <0,\ \forall\, k\in A_L(t).
\end{array}
\]
 which thus could be modeled via 
 \[
 \mathbf{1}_{[0,\infty)}\left(\min\limits_{k\in A_R(t)}\left\{(a^k)^\top X^s-b_k\right\}\right)\mathbf{1}_{(0,\infty)}\left(\min\limits_{k\in A_L(t)}\left\{-(a^k)^\top X^s +b_k\right\}\right)=1.
 \]
As before, the above product is lower approximated by 
\begin{equation*}
\begin{array}{l}
       \onebld_{[0,\infty)} \left(\underbrace{\min\left\{\min\limits_{k\in A_R(t)}\left\{(a^k)^\top X^s-b_k\right\},\min\limits_{k\in A_L(t)}\left\{-(a^k)^\top X^s +b_k-\varepsilon\right\}\right\}}_{\mbox{represented by $\phi_{st}(\mathbf a, \mathbf b)$}  }\right);
\end{array}
\end{equation*}
and upper approximated by 
\begin{equation*}
\begin{array}{l}
      \onebld_{(-\varepsilon,\infty)} \left(\min\left\{\min\limits_{k\in A_R(t)}\left\{(a^k)^\top X^s-b_k\right\},\min\limits_{k\in A_L(t)}\left\{-(a^k)^\top X^s +b_k-\varepsilon\right\}\right\}\right).
\end{array}
\end{equation*}

In the following, we consider to maximize the out-of-margin accuracy with the precision constraints for a subset of labels ${\cal J}_P\subseteq [\, J\,]$, controlled by a predefined threshold $\beta_j \in (0,1)$. With the above $\varepsilon$-approximation construction, we obtain the following parameter estimation problem:
{\fontsize{8.5pt}{10pt}\selectfont
\begin{equation*}\label{eq:P_tree_classitication}
\begin{array}{l}
 \underset{\substack{(\boldsymbol{a},\boldsymbol{b})\in P\\\{j_t\}\subseteq [J]}}{\textbf{maximize}} 
 \displaystyle{\frac{1}{N}} 
   \underset{t\in \mathcal T_\ell}{\sum} \underbrace{\left(\sum\limits_{s=1}^{N} \mathbf{1}\{y_s=j_t\}\mathbf{1}_{[0,\infty)}\left(\min\left\{\min\limits_{k\in A_R(t)}\left\{(a^k)^\top X^s-b_k-1\right\},\min\limits_{k\in A_L(t)}\left\{-(a^k)^\top X^s +b_k-1\right\}\right\}\right) \right)}_{\mbox{represented by $L_t$}}\\
\textbf{subject to } \mbox{ for all } 
j \in {\cal J}_P: \\[0.05in]
\displaystyle{\sum\limits_{t\in \mathcal T_\ell} }\underbrace{\left(\sum\limits_{s=1}^{N}\mathbf{1}\{y_s=j\}\mathbf{1}\{j_t=j\}\onebld_{[0,\infty)} \left(\min\left\{\min\limits_{k\in A_R(t)}\left\{(a^k)^\top X^s-b_k\right\},\min\limits_{k\in A_L(t)}\left\{-(a^k)^\top X^s +b_k-\varepsilon\right\}\right\}\right) \right)}_{\mbox{represented by $\eta_{jt}$} }\\ [0.2in]
\epc \displaystyle{-\beta_{j}\sum\limits_{t\in \mathcal T_\ell}} \underbrace{\left( \sum\limits_{s=1}^{N}\mathbf{1}\{j_t=j\}\onebld_{(-\varepsilon,\infty)} \left(\min\left\{\min\limits_{k\in A_R(t)}\left\{(a^k)^\top X^s-b_k\right\},\min\limits_{k\in A_L(t)}\left\{-(a^k)^\top X^s +b_k-\varepsilon\right\}\right\}\right)\right)}_{\mbox{represented by $\zeta_{jt}$} } \ge 0 \\ [0.1in]
\displaystyle{\sum\limits_{t\in \mathcal T_{\ell}}\sum\limits_{s=1}^N \onebld\{j_t=j\}\mathbf{1}_{[0,\infty)}\Big(\min\limits_{k\in A_R(t)}\left\{(a^k)^\top X^s-b_k\right\}\Big)\mathbf{1}_{(0,\infty)}\Big(\min\limits_{k\in A_L(t)}\Big\{-(a^k)^\top X^s +b_k\Big\}\Big)\ge 1} 
\end{array}
\end{equation*}
}
where $P\triangleq \{(\bs{a},\bs{b})\in \mathbb R^{(p+1)\times (2^D-1)}:\|a^k\|_1 \le \tau_1, |b_k| \le \tau_1,  \|a^k\|_0 \le \tau_0,~\forall k \in {\cal T_B}\}$, and the second inequality constraint ensures that the denominator in precision is nonzero. 

The IP reformulation is more complicated due to the Heaviside term for the leaf node label assignment $\mathbf{1}\{y_s=j_t\}$ in the objective and $\mathbf{1}\{j_t=j\}$ in the constraints. We use one-hot encoding variables $\{c_{jt}\}_{j\in [J], t\in \mathcal T_{\ell}}$ to represent the lead node labels $\{j_t\}$, then the parameter estimation problem with the above approximation can be reformulated as an integer program using big-M technique as below. 

 \begin{subequations}
    \begin{align}
\underset{\bs{L},\,\bs{\eta},\,\bs{\zeta}\geq 0}{\underset{(\boldsymbol{a},\boldsymbol{b})\in P,\,\bs{c},\,\bs{\xi},\,\bs{z}^\pm,}{\textbf{maximize}}} \quad & \frac{1}{N} \left(\sum\limits_{t\in \mathcal T_\ell}L_t \right) \label{eq:Tree_precision_constraint_epsilon_full_mip-a}\\
    \textbf{subject to} \quad & \left\{\begin{array}{lll}
    (a^k)^\top X^s -b_k -1 \ge -M_{\xi}(1-\xi_{st}), && ~\forall k\in A_R(t),~\forall t\in \mathcal T_\ell, ~\forall s\in [N],\\
     -(a^k)^\top X^s +b_k -1 \ge -M_{\xi}(1-\xi_{st}), && ~\forall k\in A_L(t),~\forall t\in \mathcal T_\ell, ~\forall s\in [N],\\
     L_t \, \leq \, \displaystyle{
     \sum_{s=1}^N
     } \, \mathbf{1}\{j_s=j\}\xi_{st}+M_L(1-c_{jt}), && ~\forall j\in [J],~\forall t\in \mathcal T_\ell\end{array}\right., \label{eq:tree_obj_ip}\\
    & \left\{\begin{array}{lll}\phi_{st}(\mathbf a, \mathbf b) \ge -M_z(1-z^+_{st}), && ~\forall t\in \mathcal T_\ell, ~\forall s\in [N],\\
     \eta_{jt}\le M_{\eta}\cdot c_{jt}, && ~\forall j\in \mathcal J_P,~\forall t\in \mathcal T_\ell,\\
     \eta_{jt} \, \leq \, \displaystyle{ 
     \sum_{s=1}^N
     } \, \onebld\{j_s=j\}z_{st}^+ +M_{\eta}(1-c_{jt}), && ~\forall j\in \mathcal J_P,~\forall t\in \mathcal T_\ell\end{array}\right.,\label{eq:tree_cons_eta_ip}\\
    & \left\{\begin{array}{lll}
     -\phi_{st}(\mathbf a, \mathbf b)-\varepsilon\ge -M_z z_{st}^-, && ~\forall t\in \mathcal T_\ell, ~\forall s\in [N],\\
     \zeta_{jt}\ge -M_{\zeta}\cdot c_{jt}, && ~\forall j\in  \mathcal J_P,~\forall t\in \mathcal T_\ell,\\
     \zeta_{jt} \, \ge \, \displaystyle{
     \sum_{s=1}^N  
     } \, z_{st}^- -M_{\zeta}(1-c_{jt}), && ~\forall 
     j\in  \mathcal J_P,~\forall t\in \mathcal T_\ell
     \end{array}\right.,\label{eq:tree_cons_zeta_ip} \\
     & \displaystyle{
     \sum_{t\in \mathcal T_\ell} 
     } \, \eta_{jt} - \beta_{pj} \displaystyle{
     \sum_{t\in \mathcal T_\ell}
     } \, \zeta_{jt} \ge 0, \qquad  ~\forall j\in \mathcal J_P,\\
    & \displaystyle{
    \sum_{t\in \mathcal T_\ell}
    } \, \eta_{jt}  \ge 1, \qquad  ~\forall j\in \mathcal J_P,\label{eq:Tree_precision_constraint_epsilon_full_mip-30m}\\
      & \displaystyle{
      \sum_{j\in [J]}
      } \, c_{jt}=1, \qquad \forall t\in \mathcal T_{\ell},\label{eq:Tree_precision_constraint_epsilon_full_mip-e}\\
    & \xi_{st},~c_{jt},~z_{st}^+,~z_{st}^-\in \{0,1\}, \qquad  ~\forall s\in [N],~\forall t\in \mathcal T_\ell, ~\forall j\in [J],~\forall \ell\in [D].
    \label{eq:Tree_precision_constraint_epsilon_full_mip-u}
    \end{align}
    \end{subequations}
 In the above program, \eqref{eq:tree_obj_ip} models the objective,  \eqref{eq:tree_cons_eta_ip} and \eqref{eq:tree_cons_zeta_ip} model the first and the second term in the precision constraint respectively. We keep the 
 pointwise mininum function $\phi_{st}(\mathbf a, \mathbf b)$ in the above program, which can be either described by $(D-1)N \, 2^D$ binary variables for piecewise selection following \cite{BertsimasDunn17}, or further decomposed via PA decomposition in the IDSA algorithms.   The number of integers in the above MIP is: $(3N+J) \, 2^D$, including $N \, 2^D$ for Heaviside functions in the objective, $2N  \,2^D$ for Heaviside functions in precision constraint(s), and $J \, 2^D$ for class assignment to leaf nodes.
 
\subsubsection{Numerical results}
\label{sec:tree based results}
For the tree classification problem, we implement the Full MIP and IDSA-PIP methods on 8 datasets, each of which consists of 4 random folds. The details of data information, parameter setting and method description are provided in Appendix \ref{app:settings tree}.  Let $R_{\text{obj}}$ represent the ratio of objective value produced by IDSA-PIP and Full MIP, and if Full MIP fails to find a feasible solution, set $R_{\text{obj}}=\infty$. Let $R_{\text{time}}$ represent the ratio of the running time reported by IDSA-PIP and Full MIP. 

The advantage of IDSA-PIP over Full MIP in terms of objective values becomes more pronounced when the tree depth increases. We find that the IDSA-PIP method exhibits better performance with large datasets for depth-2 decision tree model. Specifically, for depth-2 decision trees, on  datasets with small size (\textit{wine} and \textit{nwth}) $R_{\text{obj}}>,=,<1$ on $12.5\%$, $75\%$ and $12.5\%$ instances respectively. On medium-sized datasets (\textit{htds}, \textit{dmtl} and \textit{blsc}), $R_{\text{obj}} >,= ,< 1$ on $75.0\%$, $8.3\%$, and $16.7\%$  instances, respectively. On large datasets (\textit{ctmc}, \textit{ceva} and \textit{fish}) $R_{\text{obj}}>,=1$ on $91.7\%$ and $8.3\%$ instances respectively. On the dataset \textit{fish} which is the largest one among the 8 datasets, Full MIP for decision trees with all depths fails to find a feasible solution on all instances, while IDSA-PIP always produces feasible solutions. The complete results for all datasets are presented in Table~\ref{tab:wine} to \ref{tab:fish} in Appendix \ref{app:settings tree}. The advantage of IDSA-PIP is not limited to the training set, it also carries over to the test set. Specifically, IDSA-PIP achieves the higher test accuracy than that of Full MIP on $74.0\%$ of the instances, and equal or lower test accuracy on $15.6\%$ and $10.4\%$ of the instances, respectively, indicating a clear overall advantage in generalization performance.

\begin{table}[h]
\caption{Summary of Full MIP and IDSA-PIP performance on Tree with depth 2, 3 and 4.}\label{tab:summary_obj_tree}
\begin{tabular}{crrrrrrc}
\toprule
\multirow{2}{*}{Tree Depth} & \multicolumn{6}{c}{$R_{\text{obj}}$}& \multirow{2}{*}{Full MIP infeasible}\\ 
\cmidrule{2-7}
&(0, 1) & =1 & (1, 1.2] & (1.2, 1.5] & (1.5, 2] & (2, $\infty$) & \\
\midrule
2     & 9.4\% & 25.0\% & 25.0\% & 21.9\% & 6.3\% & 0.0\% & 12.5\% \\
3     & 3.1\% & 15.6\% & 28.1\% & 25.0\% & 0.0\% & 9.4\% & 18.8\% \\
4     & 3.1\% & 21.9\% & 28.1\% & 12.5\% & 0.0\% & 12.5\% & 21.9\% \\
\bottomrule
\end{tabular}
\footnotetext{Note: The last column represents that Full MIP fails to find a feasible solution while IDSA-PIP finds feasible solutions}
\end{table}

Figure \ref{fig:cumulative-instance-ratio-tree} illustrates that IDSA-PIP method only takes less than one-half time than Full MIP in over 90\% instances with no worse objective value  for depth-2/3 decision trees, and over 80\% instances for depth-4 decision tree. 

\begin{figure}[h]
\centering
\includegraphics[width=0.4\textwidth]{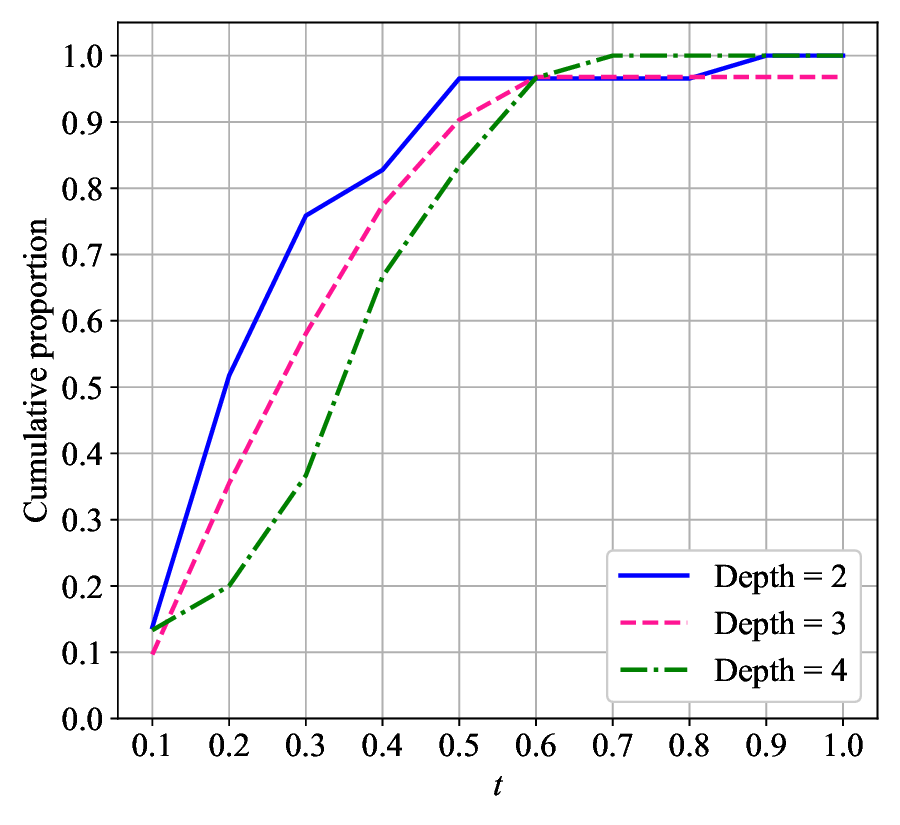}
\caption{Cumulative proportion among instances with $R_{\text{obj}} >=1$.
\\
Note: Each curve in the figure above is for decision trees with a specific depth $\in \{2,3,4\}$. The $y$-axis represents the cumulative proportion among instances with $R_{\text{obj}} >=1$, and $x$-axis represents $R_{\text{time}}$, the ratio of running time by IDSA-PIP and Full MIP.}
\label{fig:cumulative-instance-ratio-tree}
\end{figure}
 
Admitting that the IDSA-PIP method takes longer time compared with classification tree models based on top-down recursive partitioning such as \texttt{CART} in \cite{BreimanFOStone84}, we show that the tree classification model by IDSA-PIP method Pareto-dominates several competitive models including \texttt{CART}, \texttt{BinOCT} in \cite{verwer2019learning} and \texttt{BendersOCT}, \texttt{FlowOCT} in \cite{AghaeiGomezVayanos21} through the accuracy-precision plots and Pareto curves.  Figure \ref{fig:pareto-curve-blsc-depth4} plots the accuracy-precision curves for depth-4 decision trees on dataset \textit{blsc} and we defer the plots for other depths in Figure~\ref{fig:pareto-curve-blsc-depth2}-\ref{fig:pareto-curve-blsc-depth3} and for another dataset \textit{ctmc} in Figure~\ref{fig:pareto-curve-ctmc-depth2}- \ref{fig:pareto-curve-ctmc-depth4} in Appendix \ref{app:settings tree}.  In Figure~\ref{fig:pareto-curve-blsc-depth4}, for models without precision constraint (\texttt{U-PIP}, \texttt{BendersOCT}, \texttt{U-BinOCT} and \texttt{CART}), we plot one single point, and for each model with precision constraint (\texttt{C-PIP}, \texttt{FlowOCT} and \texttt{C-BinOCT}), we set a series of precision thresholds and draw Pareto curves of all non-dominated points for the training dataset and the test dataset respectively. \texttt{C-PIP} achieves a superior Pareto frontier dominating the rest both on the training set and the test set in all 4 Folds, which is possibly because other models such as  \texttt{BendersOCT}, \texttt{FlowOCT} and \texttt{C-BinOCT}) are restricted to univariate splits or discrete feature values in learning the classification trees.

\begin{figure}[h]
\centering
\includegraphics[width=\textwidth]{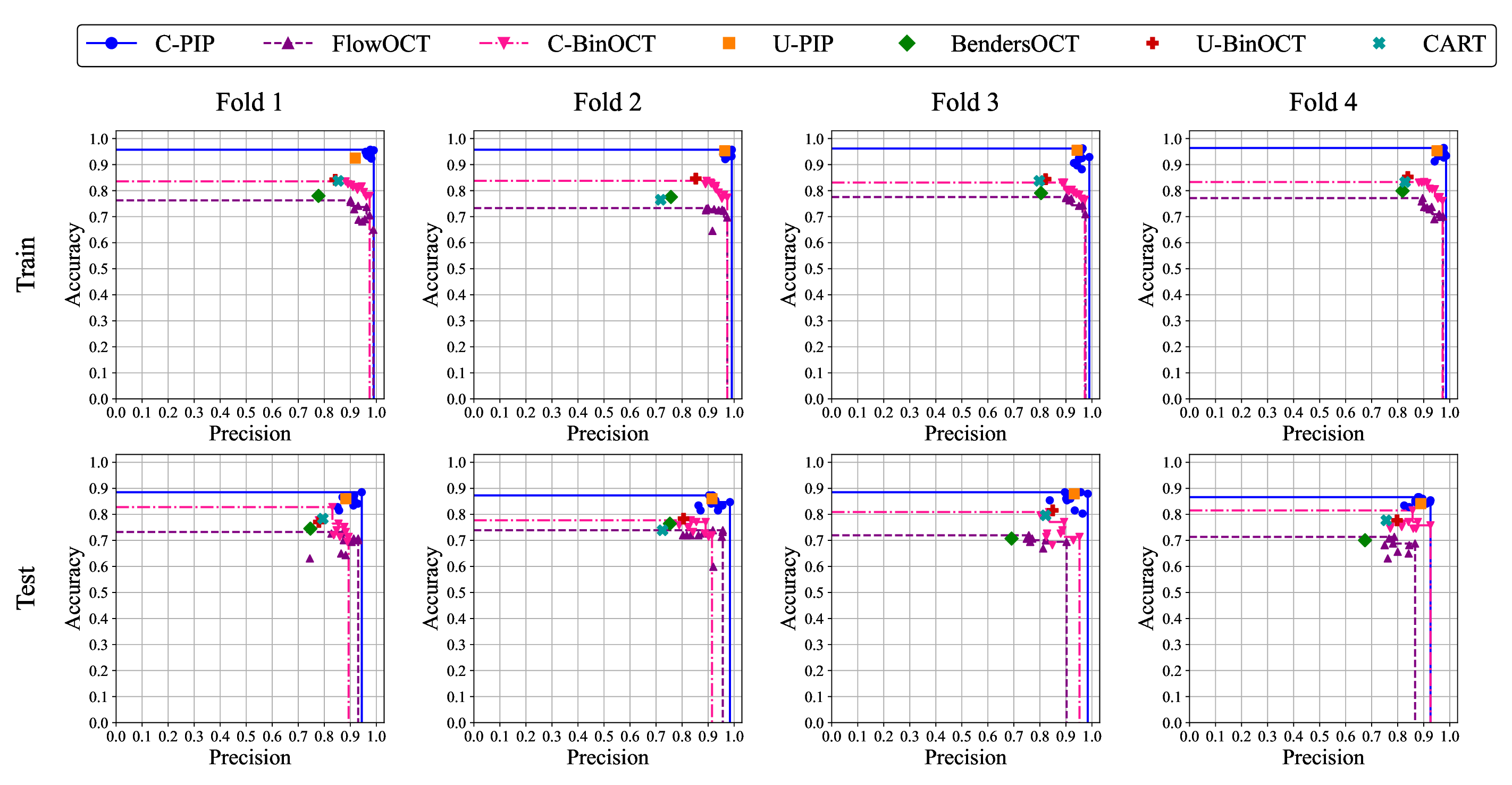}
\caption{Accuracy-precision plots and pareto curve, \textit{blsc} dataset, depth-4. }
\label{fig:pareto-curve-blsc-depth4}
\end{figure}

\begin{table}[h]
\caption{Summary result of decision tree models on \textit{blsc} dataset.} 
\label{tab:summary blsc}
\begin{tabular}{llllllllll}
\toprule
 & \multicolumn{3}{c}{Depth 2}  & \multicolumn{3}{c}{Depth 3}  & \multicolumn{3}{c}{Depth 4} \\
\cmidrule(r){2-4} \cmidrule(lr){5-7} \cmidrule(l){8-10}
& \multicolumn{1}{c}{Train\_acc} &
\multicolumn{1}{c}{Test\_acc} &
\multicolumn{1}{c}{Time(s)} &
\multicolumn{1}{c}{Train\_acc} &
\multicolumn{1}{c}{Test\_acc} &
\multicolumn{1}{c}{Time(s)} &
\multicolumn{1}{c}{Train\_acc} &
\multicolumn{1}{c}{Test\_acc} &
\multicolumn{1}{c}{Time(s)} \\
\midrule
C-PIP      & \textbf{0.916}   & \textbf{0.886}   & 1125    & \textbf{0.907}   & \textbf{0.865}   & 1585    & \textbf{0.937}   & \textbf{0.850}   & 1717    \\
FlowOCT    & 0.531   & 0.512   & 454.6   & 0.657   & 0.645   & 3348    & 0.728   & 0.695   & 3600    \\
C-BinOCT   & 0.564   & 0.548   & 3.111   & 0.715   & 0.672   & 3600    & 0.804   & 0.747   & 3600    \\
\midrule
U-PIP      & \textbf{0.898}   & \textbf{0.870}   & 310.8   & \textbf{0.928}   & \textbf{0.854}   & 402.6   & \textbf{0.947}   & \textbf{0.860}   & 454.6   \\
BendersOCT & 0.686   & 0.667   & 7.073   & 0.742   & 0.688   & 215.1   & 0.786   & 0.729   & 3600    \\
U-BinOCT   & 0.730   & 0.677   & 2.145   & 0.783   & 0.717   & 3600    & 0.846   & 0.787   & 3600    \\
CART       & 0.668   & 0.642   & 0.121   & 0.757   & 0.694   & 0.122   & 0.818   & 0.774   & 0.126  \\
\botrule
\end{tabular}
\footnotetext{Note: For constrained models, the \textbf{Train\_acc, Test\_acc, Time(s)} refer to the \textbf{average} value of all folds and all precision thresholds}
\end{table}

\begin{table}[h]
\caption{Summary result of decision tree models  on \textit{ctmc} dataset.} \label{tab:summary ctmc}
\begin{tabular}{llllllllll}
\toprule
& \multicolumn{3}{c}{Depth 2}  & \multicolumn{3}{c}{Depth 3}  & \multicolumn{3}{c}{Depth 4} \\
\cmidrule(r){2-4} \cmidrule(lr){5-7} \cmidrule(l){8-10}
 & \multicolumn{1}{c}{Train\_acc} &
  \multicolumn{1}{c}{Test\_acc} &
  \multicolumn{1}{c}{Time(s)} &
  \multicolumn{1}{c}{Train\_acc} &
  \multicolumn{1}{c}{Test\_acc} &
  \multicolumn{1}{c}{Time(s)} &
  \multicolumn{1}{c}{Train\_acc} &
  \multicolumn{1}{c}{Test\_acc} &
  \multicolumn{1}{c}{Time(s)} \\
\midrule
C-PIP      & 0.534   & 0.489   & 1123    & \textbf{0.576}   & \textbf{0.524}   & 1274    & \textbf{0.593}   & 0.526   & 1416    \\
FlowOCT    & 0.504   & 0.488   & 572.4   & 0.497   & 0.465   & 3600    & 0.485   & 0.454   & 3601    \\
C-BinOCT   & \textbf{0.547}   & \textbf{0.515}   & 137.4   & 0.566   & 0.522   & 3601    & 0.576   & \textbf{0.527}   & 3601    \\
\midrule
U-PIP      & 0.535   & 0.475   & 370.8   & \textbf{0.598}   & \textbf{0.531}   & 467.8   & \textbf{0.624}   & 0.540   & 542.3   \\
BendersOCT & 0.515   & 0.497   & 34.41   & 0.547   & 0.505   & 3600   & 0.563   & 0.508   & 3600    \\
U-BinOCT   & \textbf{0.554}   & \textbf{0.518}   & 107.0   & 0.577   & 0.527   & 3600    & 0.583   & 0.513   & 3601    \\
CART       & 0.488   & 0.464   & 0.133   & 0.541   & 0.516   & 0.134   & 0.580   & \textbf{0.549}   & 0.148  \\
\botrule
\end{tabular}
\footnotetext{Note: For constrained models, the columns \textbf{Train\_acc, Test\_acc, Time(s)} refer to the \textbf{average} value of all folds and all precision thresholds}
\end{table}

\begin{section}{Summary}
In this paper, we have considered computational methods for solving
the optimization problem defined by 
mixed-signed affine combinations of Heaviside composite functions.  The lack
of matching semicontinuity of the objective function and closedness of the 
constraint set is remedied by an 
$\varepsilon$-approximation of the Heaviside function; each Heaviside term
in the problem
is substituted by a binary variable and the piecewise structure of
the inner functions in
the composition is exploited in the design of an
iterative decomposed shrinkage algorithm for solving the problem.  
Two versions of the algorithm are considered:
one for the full-integer formulation and the other for its partial-integer 
enhancement.  Convergence results of the algorithms are established 
and extensive computational experiments are performed on two
classes of multiclass classification problems. Numerical results support
the effectiveness of the algorithms for solving these problems, further
confirming the promise of the progressive integer programming approach
for solving many optimization problems with continuous variables mixed with
discrete structures.
    
\end{section}

\begin{paragraph}{\rm \textbf{Funding}}

Junyi Liu is supported by National Natural Science Foundation of China under grant 72571155. Jong-Shi Pang is supported by U.S. Air Force Office of Scientific Research under grant FA9550-22-1-0045.

\end{paragraph}

\begin{paragraph}{\rm \textbf{Data and Code Availability}}

The data for this study consists of publicly available and computationally generated
instances. URLs of publicly accessible data sources are cited within the article. The source code for generating the
other instances  and numerical experiments is available in the GitHub repository: \url{https://github.com/zhengke1111/HSCOP_classification}.

\end{paragraph}

\begin{appendices}



\section{Numerical experiment setting} 
\subsection{Score-based multiclass classification}
\label{app:multi-setting}
Dataset information is listed in Table \ref{tab:multi-datasets}. Each dataset is stratified into 4 folds, yielding 4 distinct training sets (each consisting of 3 folds, with the remaining fold as the test set). For each training set, we define 4 precision constraint configurations, generating 4 problem instances per training set.
Training and test sets are standardized: the training set is scaled to mean 0 and variance 1, with the same transformation applied to the test set.

\begin{table}[h]
\caption{Datasets for score-based multiclass classification} \label{tab:multi-datasets}
\begin{tabular}{llccclll}
\toprule
ID & source & $N$ & $p$ & $J$ & \# of each class & \multicolumn{2}{c}{precision setting} \\
\midrule
\multirow{2}{*}{\textit{vehi}} & \multirow{2}{*}{{\href{https://archive.ics.uci.edu/dataset/149/statlog+vehicle+silhouettes}{uci}}} & \multirow{2}{*}{846} & \multirow{2}{*}{18} & \multirow{2}{*}{4} & \multirow{2}{*}{212/217/218/200} 
& $\beta_0 = 0.62, \beta_3 = 0.80$ & $\beta_0 = 0.67, \beta_3 = 0.80$ \\
& & & & & & $\beta_0 = 0.62, \beta_1 = 0.80, \beta_3 = 0.80$ & $\beta_0 = 0.67, \beta_1 = 0.80, \beta_3 = 0.80$ \\
\midrule
\multirow{2}{*}{\textit{wine}} & \multirow{2}{*}{\href{https://archive.ics.uci.edu/dataset/186/wine+quality}{uci}} & \multirow{2}{*}{1599} & \multirow{2}{*}{11} & \multirow{2}{*}{3} & \multirow{2}{*}{63/1319/217} 
& $\beta_1 = 0.85$ & $\beta_1 = 0.90$ \\
& & & & & & $\beta_1 = 0.85, \beta_2 = 0.65$ & $\beta_1 = 0.90, \beta_2 = 0.65$ \\
\midrule
\multirow{2}{*}{\textit{segm}} & \multirow{2}{*}{\href{https://archive.ics.uci.edu/dataset/50/image+segmentation}{uci}} & \multirow{2}{*}{2310} & \multirow{2}{*}{19} & \multirow{2}{*}{7} & \multirow{2}{*}{\makecell[l]{1533/1508/1358/\\707/703/626}} 
& $\beta_2 = 0.80$ & $\beta_2 = 0.85$ \\
& & & & & & $\beta_2 = 0.80, \beta_4 = 0.75$ & $\beta_2 = 0.85, \beta_4 = 0.75$ \\
\midrule
\multirow{2}{*}{\textit{fish}} & \multirow{2}{*}{\href{https://www.kaggle.com/datasets/taweilo/fish-species-sampling-weight-and-height-data?resource=download}{kaggle}} & \multirow{2}{*}{4080} & \multirow{2}{*}{3} & \multirow{2}{*}{9} & \multirow{2}{*}{\makecell[l]{480/476/475/468/458/\\455/435/418/415}} 
& $\beta_0 = 0.55, \beta_3 = 0.4$ & $\beta_0 = 0.60, \beta_3 = 0.40$ \\
& & & & & & $\beta_0 = 0.55, \beta_1 = 0.99, \beta_3 = 0.40$ & $\beta_0 = 0.60, \beta_1 = 0.99, \beta_3 = 0.40$ \\
\midrule
\multirow{2}{*}{\textit{wave}} & \multirow{2}{*}{\href{https://archive.ics.uci.edu/dataset/107/waveform+database+generator+version+1}{uci}} & \multirow{2}{*}{5000} & \multirow{2}{*}{22} & \multirow{2}{*}{3} & \multirow{2}{*}{1696/1657/1647} 
& $\beta_0 = 0.89$ & $\beta_0 = 0.91$ \\
& & & & & & $\beta_0 = 0.89, \beta_1 = 0.80$ & $\beta_0 = 0.91, \beta_1 = 0.80$ \\
\midrule
\multirow{2}{*}{\textit{robo}} & \multirow{2}{*}{\href{https://archive.ics.uci.edu/dataset/194/wall+following+robot+navigation+data}{uci}} & \multirow{2}{*}{5456} & \multirow{2}{*}{2} & \multirow{2}{*}{4} & \multirow{2}{*}{2205/2097/826/328} 
& $\beta_2 = 0.81$ & $\beta_2 = 0.86$ \\
& & & & & & $\beta_2 = 0.81, \beta_1 = 0.90$ & $\beta_2 = 0.86, \beta_1 = 0.90$ \\
\botrule
\end{tabular}
\end{table}

The numerical results on \textit{segm} and \textit{wave} datasets are presented in Table~\ref{tab:segm-1}-\ref{tab:wave-4}. Since {\sc Gurobi} solver fails to find an optimal solution for all instances within the 3600-second time limit, we record the first time that the best solution is obtained within the time limit in Table~\ref{tab:segm-1}- \ref{tab:wave-4}. Additional experimental parameters are provided below:
\begin{itemize}
    \item \textbf{Initial solution}: To generate a high-quality initial solution for the MIP solver, we first train an $\ell_2$-penalized SVM with hinge loss on the training set, obtaining $(W', \mathbf{b}')$ and make the scaling to ensure that $\displaystyle{
    \max_{j\in[J]}
    } \, \{\|w_j'\|_1, |b_j'|\} \leq \tau$. With this adjusted solution as a warm start, we solve the unconstrained Full MIP problem with a 120-second time limit, and the produced solution serves as the initial solution for all methods in subsequent experiments.
    \item \textbf{Model parameters}: The precision thresholds are provided in Table \ref{tab:multi-datasets}. The recall thresholds $\{\alpha_j\}_{j\in\mathcal{J}}$ are uniformly set to $0.1$ on the same classes as precision constraints to prevent empty class assignment.
    We set the penalty parameter $\rho=10^4$ for infeasible solutions, and    $M=20\max_{s\in[N]} \|X^s\|_\infty + 120$ for the big-M formulation, $\tau=10$ uniformly across all experiments, and use $\varepsilon=10^{-5}$ for the problem solved by PIP.
    \item \textbf{Algorithm parameters}: The hyperparameters of the PIP algorithm are set as follows: the initial integer ratio $r_0 = 0.4$, maximum integer ratio $r_{\text{max}} = 0.75$, minimum integer ratio $r_{\text{min}} = 0.3$, and change ratio $r_\Delta = 0.1$. For iterative control, we set the maximum PIP iteration $\mu_{\text{max}} = 10$, maximum unchanged iteration $\tilde{\mu}_{\text{max}} = 4$. For IDSA/ISA-PIP, the maximum iteration $\nu_{\text{max}} = 3$ and $\varepsilon_{\nu}=10^{-2-\nu}$.
    \item \textbf{Time limit}: We set a time limit of 3600 seconds for the {\sc Gurobi} solver when solving the Full MIP problem. For solving PIP subproblems, we use a time limit of 540 seconds for ISA-PIP and PIP, and 360 seconds for IDSA-PIP. Additionally, we adopt an early stopping criterion for the subproblems: if no improvement in the objective function is observed over $10\%$ of the allocated time limit, the solving process is terminated.
\end{itemize}

\begin{table}[h]
\caption{Results of score-based classification  on \textit{segm} dataset with a single precision constraint} 
\label{tab:segm-1}
\begin{tabular}{cclrrrrrr}
\toprule
 $\beta_2$ &    {Fold} & {Method} & {Obj} & {Time(s)} & {Train\_acc} & {Test\_acc} & {Train\_prec} & {Test\_prec} \\
\midrule
\multirow{16}{*}{0.80} & \multirow{4}{*}{1} & Full MIP & 0.888          & 2604 & \textbf{0.948} & \textbf{0.957} & 0.863 &   0.848  \\
& & PIP      & \textbf{0.913} & 1297 & 0.935          & 0.941          & 0.853 &  0.835   \\
& & ISA-PIP  & 0.909          & 3293 & 0.935          & 0.939          & 0.846 & 0.826   \\
& & IDSA-PIP & 0.910          & 796  & 0.936          & 0.941          & 0.856 & 0.843 \\ 
\cmidrule{2-9}
& \multirow{4}{*}{2} & Full MIP & 0.900          & 3600 & 0.930          & 0.939          & 0.833 & 0.874 \\
& & PIP      & 0.911          & 965  & \textbf{0.938} & \textbf{0.941} & 0.813 & 0.852\\
 && ISA-PIP  & 0.916          & 3143 & \textbf{0.938} & 0.931          & 0.807 &  0.815 \\
&& IDSA-PIP & \textbf{0.917} & 1288 & 0.936          & 0.934          & 0.807 &  0.813\\
\cmidrule{2-9}
&\multirow{4}{*}{3} & Full MIP & 0.915          & 3555 & 0.937          & 0.920          & 0.813 & 0.750 \\
&& PIP      & 0.920          & 982  & 0.941          & 0.915          & 0.801 &  0.745 \\
&& ISA-PIP  & \textbf{0.922} & 2553 & 0.942          & 0.915          & 0.836 & 0.756  \\
&& IDSA-PIP & 0.920          & 1053 & \textbf{0.945} & \textbf{0.922} & 0.868 &  0.791 \\
\cmidrule{2-9}
&\multirow{4}{*}{4} & Full MIP & 0.908          & 2717 & 0.931          & 0.906          & 0.805 & 0.760  \\
&& PIP      & 0.923          & 858  & 0.938          & 0.913          & 0.811 & 0.791  \\
&& ISA-PIP  & \textbf{0.931} & 3946 & \textbf{0.941} & \textbf{0.927} & 0.856 &  0.854 \\
&& IDSA-PIP & 0.928          & 1162 & 0.936          & 0.917          & 0.835 & 0.845  \\
\midrule \midrule
\multirow{16}{*}{0.85} & \multirow{4}{*}{1} & Full MIP & infeas.        & 3205 & 0.933          & 0.936          & 0.824 &  0.809   \\
&& PIP      & 0.909          & 887  & \textbf{0.939} & 0.943          & 0.866 & 0.837 \\
&& ISA-PIP  & 0.908          & 2683 & \textbf{0.939} & \textbf{0.945} & 0.863  & 0.856\\
&& IDSA-PIP & \textbf{0.914} & 1228 & 0.935          & 0.939 & 0.857 &  0.843 \\
\cmidrule{2-9}
&\multirow{4}{*}{2} & Full MIP & infeas.        & 3119 & 0.931          & 0.933          & 0.791 &  0.826  \\
& & PIP      & 0.906          & 2076 & \textbf{0.940} & \textbf{0.943} & 0.855   & 0.904  \\
&& ISA-PIP  & 0.913          & 4964 & 0.939          & 0.941          & 0.855 & 0.874\\
 && IDSA-PIP & \textbf{0.920}  & 1715 & 0.934          & 0.934          & 0.851 &  0.874  \\
\cmidrule{2-9}
&\multirow{4}{*}{3} & Full MIP & infeas.        & 2732 & 0.935          & 0.920          & 0.794 &  0.735 \\
&& PIP      & 0.917          & 1346 & \textbf{0.947} & 0.920          & 0.862  & 0.780 \\
&& ISA-PIP  & 0.919          & 3319 & \textbf{0.947} & \textbf{0.924} & 0.856   & 0.789 \\
&& IDSA-PIP & \textbf{0.926} & 1339 & 0.939          & 0.915          & 0.881 &  0.812\\
\cmidrule{2-9}
&\multirow{4}{*}{4} & Full MIP & 0.887          & 3600 & 0.931          & 0.910          & 0.856 &  0.816 \\
&& PIP      & 0.930           & 1766 & \textbf{0.942} & 0.922          & 0.857  & 0.855  \\
&& ISA-PIP  & \textbf{0.932} & 4554 & 0.941          & \textbf{0.925} & 0.853 & 0.854 \\
&& IDSA-PIP & \textbf{0.932} & 1852 & 0.941          & 0.924          & 0.870 &  0.886 \\
\bottomrule
\end{tabular}
\end{table}

\begin{table}[h]
\caption{Results of score-based classification  on \textit{segm} dataset  with two precision constraints, $\beta_2=0.80,\,\beta_4=0.75$} 
\label{tab:segm-3}
\begin{tabular}{clrrrrrrrrrr}
\toprule
     \multirow{2}{*}{Fold} & \multirow{2}{*}{Method} & \multicolumn{1}{l}{\multirow{2}{*}{Obj}} & \multicolumn{1}{l}{\multirow{2}{*}{Time(s)}} & \multicolumn{1}{l}{\multirow{2}{*}{Train\_acc}} & \multicolumn{1}{l}{\multirow{2}{*}{Test\_acc}} & \multicolumn{3}{c}{Train\_prec} & \multicolumn{3}{c}{Test\_prec} \\
\cmidrule{7-12}   &  &  &  &  &  & 
  \multicolumn{1}{l}{label $2$} &
  \multicolumn{2}{l}{label $4$} &
  \multicolumn{1}{l}{label $2$} &
  \multicolumn{2}{l}{label $4$} \\
\midrule
\multirow{4}{*}{1} & Full MIP & 0.903          & 3005 & 0.930          & 0.938          & 0.800 & 0.815 &  & 0.802 & 0.903 &  \\
                       & PIP      & 0.893          & 805  & \textbf{0.943} & \textbf{0.953} & 0.866 & 0.798 &  & 0.854 & 0.872 &  \\
                       & ISA-PIP  & 0.910          & 1850 & 0.937          & 0.941          & 0.856 & 0.805 &  & 0.843 & 0.882 &  \\
                       & IDSA-PIP & \textbf{0.913} & 1136 & 0.938          & 0.939          & 0.860 & 0.809 &  & 0.833 & 0.881 &  \\
\midrule
\multirow{4}{*}{2} & Full MIP & infeas.        & 2773 & 0.930          & 0.933          & 0.785 & 0.810 &  & 0.826 & 0.786 &  \\
                       & PIP      & 0.912          & 788  & \textbf{0.939} & \textbf{0.941} & 0.815 & 0.812 &  & 0.837 & 0.850 &  \\
                       & ISA-PIP  & \textbf{0.917} & 2062 & 0.938          & 0.931          & 0.805 & 0.817 &  & 0.813 & 0.805 &  \\
                       & IDSA-PIP & \textbf{0.917} & 1084 & 0.936          & 0.931          & 0.807 & 0.793 &  & 0.847 & 0.775 &  \\
\midrule
\multirow{4}{*}{3} & Full MIP & infeas.        & 2360 & 0.932          & 0.917          & 0.769 & 0.863 &  & 0.730 & 0.797 &  \\
                       & PIP      & 0.919          & 1118 & 0.941          & \textbf{0.919} & 0.801 & 0.849 &  & 0.750 & 0.771 &  \\
                       & ISA-PIP  & 0.916          & 1958 & 0.942          & 0.915          & 0.827 & 0.843 &  & 0.761 & 0.737 &  \\
                       & IDSA-PIP & \textbf{0.922} & 1248 & \textbf{0.943} & 0.915          & 0.817 & 0.843 &  & 0.747 & 0.737 &  \\
\midrule
\multirow{4}{*}{4} & Full MIP & infeas.        & 3336 & 0.930          & 0.908          & 0.776 & 0.831 &  & 0.723 & 0.809 &  \\
                       & PIP      & 0.922          & 485  & 0.937          & 0.915          & 0.811 & 0.826 &  & 0.802 & 0.741 &  \\
                       & ISA-PIP  & 0.924          & 1428 & 0.941          & 0.920          & 0.861 & 0.826 &  & 0.866 & 0.780 &  \\
                       & IDSA-PIP & \textbf{0.931} & 1582 & \textbf{0.942} & \textbf{0.927} & 0.867 & 0.795 &  & 0.854 & 0.756 &  \\
\bottomrule
\end{tabular}
\end{table}

\begin{table}[h]
\caption{Results of score-based classification on \textit{segm} dataset with  two precision constraints, $\beta_2=0.85,\,\beta_4=0.75$} 
\label{tab:segm-4}
\begin{tabular}{clrrrrrrrrrr}
\toprule
     \multirow{2}{*}{Fold} & \multirow{2}{*}{Method} & \multicolumn{1}{l}{\multirow{2}{*}{Obj}} & \multicolumn{1}{l}{\multirow{2}{*}{Time(s)}} & \multicolumn{1}{l}{\multirow{2}{*}{Train\_acc}} & \multicolumn{1}{l}{\multirow{2}{*}{Test\_acc}} & \multicolumn{3}{c}{Train\_prec} & \multicolumn{3}{c}{Test\_prec} \\
\cmidrule{7-12}   &  &  &  &  &  & 
  \multicolumn{1}{l}{label $2$} &
  \multicolumn{2}{l}{label $4$} &
  \multicolumn{1}{l}{label $2$} &
  \multicolumn{2}{l}{label $4$}  \\
\midrule
\multirow{4}{*}{1} & Full MIP & infeas.        & 3022 & 0.931          & 0.936          & 0.804 & 0.816 &  & 0.802 & 0.889 &  \\
                       & PIP      & 0.902          & 714  & \textbf{0.942} & 0.946          & 0.852 & 0.819 &  & 0.864 & 0.886 &  \\
                       & ISA-PIP  & 0.903          & 2466 & 0.938          & \textbf{0.958} & 0.853 & 0.782 &  & 0.862 & 0.877 &  \\
                       & IDSA-PIP & \textbf{0.913} & 1365 & 0.935          & 0.939          & 0.853 & 0.809 &  & 0.833 & 0.881 &  \\
\midrule
\multirow{4}{*}{2} & Full MIP & infeas.        & 3238 & 0.930          & 0.931          & 0.794 & 0.811 &  & 0.826 & 0.786 &  \\
                       & PIP      & infeas.        & 249  & 0.930          & 0.931          & 0.780 & 0.810 &  & 0.826 & 0.786 &  \\
                       & ISA-PIP  & 0.909          & 2251 & \textbf{0.942} & 0.939          & 0.852 & 0.810 &  & 0.872 & 0.831 &  \\
                       & IDSA-PIP & \textbf{0.922} & 1473 & 0.935          & \textbf{0.943} & 0.857 & 0.796 &  & 0.904 & 0.802 &  \\
\midrule
\multirow{4}{*}{3} & Full MIP & infeas.        & 3066 & 0.931          & 0.919          & 0.762 & 0.862 &  & 0.725 & 0.797 &  \\
                       & PIP      & 0.912          & 1033 & \textbf{0.950} & \textbf{0.922} & 0.868 & 0.849 &  & 0.812 & 0.741 &  \\
                       & ISA-PIP  & 0.911          & 2084 & \textbf{0.950} & \textbf{0.922} & 0.858 & 0.850 &  & 0.787 & 0.747 &  \\
                       & IDSA-PIP & \textbf{0.922} & 1703 & 0.941          & 0.920          & 0.872 & 0.849 &  & 0.802 & 0.744 &  \\
\midrule
\multirow{4}{*}{4} & Full MIP & infeas.        & 3106 & 0.930          & 0.908          & 0.768 & 0.830 &  & 0.730 & 0.809 &  \\
                       & PIP      & infeas.        & 265  & 0.930          & 0.908          & 0.768 & 0.830 &  & 0.723 & 0.821 &  \\
                       & ISA-PIP  & \textbf{0.930} & 2323 & \textbf{0.943} & \textbf{0.924} & 0.858 & 0.809 &  & 0.854 & 0.747 &  \\
                       & IDSA-PIP & 0.925          & 1424 & 0.939          & 0.917          & 0.861 & 0.833 &  & 0.866 & 0.750 &  \\
\bottomrule
\end{tabular}
\end{table}

\begin{table}[h]
\caption{Results of score-based classification on \textit{wave} dataset with a single precision constraint} 
\label{tab:wave-1}
\begin{tabular}{cclrrrrrr}
\toprule
 $\beta_0$ & {Fold} & {Method} & {Obj} & {Time(s)}& {Train\_acc} & {Test\_acc} & {Train\_prec} & {Test\_prec} \\
\midrule
\multirow{16}{*}{0.89} & \multirow{4}{*}{1} & Full MIP & infeas.        & 2556 & 0.875          & 0.869          & 0.876 &    0.872  \\
&  & PIP      & 0.834  & 1250 & \textbf{0.882} & 0.869    & 0.891& 0.867 \\
& & ISA-PIP  & 0.834          & 2567 & \textbf{0.882} & \textbf{0.870} & 0.891  & 0.874 \\
&  & IDSA-PIP & \textbf{0.836} & 1933 & \textbf{0.882} & \textbf{0.870} & 0.891& 0.871 \\
\cmidrule{2-9}
& \multirow{4}{*}{2} & Full MIP & 0.829          & 3483 & 0.878          & 0.864          & 0.892 &  0.860  \\
& & PIP      & 0.831          & 1630 & \textbf{0.880} & \textbf{0.868} & 0.894   & 0.867   \\
& & ISA-PIP  & 0.831          & 3114 & \textbf{0.880} & 0.864          & 0.891& 0.861  \\
& & IDSA-PIP & \textbf{0.834} & 1626 & 0.878          & 0.866          & 0.891 &  0.860   \\
\cmidrule{2-9}
& \multirow{4}{*}{3} & Full MIP & 0.833          & 3265 & 0.876          & 0.853          & 0.898 &   0.879   \\
& & PIP      & 0.838          & 911  & 0.879          & 0.858 & 0.893 &   0.860  \\
& & ISA-PIP  & 0.837          & 2550 & \textbf{0.881} & \textbf{0.860} & 0.896  & 0.867  \\
&  & IDSA-PIP & \textbf{0.840} & 1386 & 0.878          & 0.858          & 0.897 & 0.866   \\
\cmidrule{2-9}
& \multirow{4}{*}{4} & Full MIP & infeas.        & 1408 & 0.875          & 0.870          & 0.878 & 0.881  \\
& & PIP      & 0.831          & 1659 & \textbf{0.881} & 0.874  & 0.891  & 0.884 \\
& & ISA-PIP  & infeas.        & 4344 & 0.881          & 0.877 & 0.890 &  0.892  \\
& & IDSA-PIP & \textbf{0.836} & 2098 & 0.879          & \textbf{0.878} & 0.892 &  0.893   \\
\midrule
\midrule
\multirow{16}{*}{0.91} & \multirow{4}{*}{1} & Full MIP & 0.805          & 3427 & 0.854          & 0.862          & 0.921 &    0.918   \\
&  & PIP      & infeas.        & 1389 & 0.883          & 0.870          & 0.894 &   0.874    \\
&   & ISA-PIP  & infeas.        & 2169 & 0.882          & 0.868          & 0.891 &   0.867   \\
&  & IDSA-PIP & \textbf{0.832} & 2263 & \textbf{0.881} & \textbf{0.868} & 0.910  & 0.885 \\
\cmidrule{2-9}
& \multirow{4}{*}{2} & Full MIP & infeas.        & 3156 & 0.877          & 0.868          & 0.892 &  0.875   \\
&  & PIP      & infeas.        & 860  & 0.879          & 0.866   & 0.893 &  0.874  \\
&  & ISA-PIP  & infeas.        & 2335 & 0.881          & 0.866          & 0.896 &  0.869 \\
&  & IDSA-PIP & \textbf{0.835} & 2465 & \textbf{0.875} & \textbf{0.874} & 0.912   & 0.900   \\
\cmidrule{2-9}
& \multirow{4}{*}{3} & Full MIP & 0.820           & 3573 & 0.869          & 0.849          & 0.911 &  0.892 \\
&  & PIP      & infeas.        & 1513 & 0.886          & 0.867          & 0.908 &  0.878  \\
&  & ISA-PIP  & infeas.        & 3882 & 0.884          & 0.862          & 0.908 &   0.877   \\
&   & IDSA-PIP & \textbf{0.835} & 1870 & \textbf{0.881} & \textbf{0.858} & 0.910   & 0.873  \\
\cmidrule{2-9}
& \multirow{4}{*}{4} & Full MIP & infeas.        & 2355 & 0.877          & 0.872          & 0.879 &   0.886   \\
&  & PIP      & infeas.        & 2060 & 0.878          & 0.868          & 0.900 &  0.901    \\
&  & ISA-PIP  & infeas.        & 3747 & 0.879          & 0.875          & 0.898 &  0.908   \\
&  & IDSA-PIP & \textbf{0.833} & 2253 & \textbf{0.876} & \textbf{0.872} & 0.911   & 0.909   \\
 \bottomrule
\end{tabular}
\end{table}

\begin{table}[h]
\caption{Results of score-based classification models on \textit{wave} dataset with two precision constraints,  $\beta_0=0.89,\,\beta_1=0.80$} 
\label{tab:wave-3}
\begin{tabular}{clrrrrrrrrrr}
\toprule
     \multirow{2}{*}{Fold} & \multirow{2}{*}{Method} & \multicolumn{1}{l}{\multirow{2}{*}{Obj}} & \multicolumn{1}{l}{\multirow{2}{*}{Time(s)}} & \multicolumn{1}{l}{\multirow{2}{*}{Train\_acc}} & \multicolumn{1}{l}{\multirow{2}{*}{Test\_acc}} & \multicolumn{3}{c}{Train\_prec} & \multicolumn{3}{c}{Test\_prec} \\
\cmidrule{7-12}   &  &  &  &  &  & 
  \multicolumn{1}{l}{label $0$} &
  \multicolumn{2}{l}{label $1$} &
  \multicolumn{1}{l}{label $0$} &
  \multicolumn{2}{l}{label $1$}  \\
\midrule
\multirow{4}{*}{1} & Full MIP & infeas.        & 1    & 0.874          & 0.870          & 0.875 & 0.879 &  & 0.872 & 0.868 &  \\
                       & PIP      & 0.834          & 1552 & \textbf{0.882} & 0.867          & 0.892 & 0.882 &  & 0.865 & 0.865 &  \\
                       & ISA-PIP  & 0.837          & 4801 & 0.880          & \textbf{0.869} & 0.893 & 0.876 &  & 0.878 & 0.864 &  \\
                       & IDSA-PIP & \textbf{0.838} & 1872 & 0.879          & 0.868          & 0.890 & 0.880 &  & 0.873 & 0.864 &  \\
\midrule
\multirow{4}{*}{2} & Full MIP & infeas.        & 2929 & 0.878          & 0.863          & 0.889 & 0.864 &  & 0.866 & 0.884 &  \\
                       & PIP      & 0.827          & 1426 & \textbf{0.881} & 0.866          & 0.894 & 0.863 &  & 0.867 & 0.882 &  \\
                       & ISA-PIP  & 0.830          & 4368 & 0.879          & 0.865          & 0.891 & 0.860 &  & 0.866 & 0.875 &  \\
                       & IDSA-PIP & \textbf{0.833} & 1754 & 0.880          & \textbf{0.869} & 0.894 & 0.859 &  & 0.876 & 0.881 &  \\
\midrule
\multirow{4}{*}{3} & Full MIP & 0.835          & 2861 & 0.878          & \textbf{0.863} & 0.890 & 0.869 &  & 0.866 & 0.858 &  \\
                       & PIP      & 0.837          & 1492 & \textbf{0.881} & 0.858          & 0.893 & 0.873 &  & 0.860 & 0.857 &  \\
                       & ISA-PIP  & 0.837          & 3069 & \textbf{0.881} & 0.858          & 0.893 & 0.873 &  & 0.858 & 0.857 &  \\
                       & IDSA-PIP & \textbf{0.842} & 1901 & 0.879          & 0.862          & 0.894 & 0.870 &  & 0.867 & 0.862 &  \\
\midrule
\multirow{4}{*}{4} & Full MIP & infeas.        & 3198 & 0.877          & 0.869          & 0.878 & 0.878 &  & 0.879 & 0.852 &  \\
                       & PIP      & 0.827          & 2069 & 0.879          & 0.862          & 0.891 & 0.869 &  & 0.874 & 0.842 &  \\
                       & ISA-PIP  & 0.833          & 3862 & \textbf{0.882} & \textbf{0.877} & 0.890 & 0.878 &  & 0.895 & 0.858 &  \\
                       & IDSA-PIP & \textbf{0.834} & 2278 & 0.881          & 0.874          & 0.892 & 0.876 &  & 0.892 & 0.855 &  \\
\bottomrule
\end{tabular}
\end{table}

\begin{table}[h]
\caption{Results of score-based classification models on \textit{wave} dataset with two precision constraints, $\beta_0=0.91,\,\beta_1=0.80$} 
\label{tab:wave-4}
\begin{tabular}{clrrrrrrrrrr}
\toprule
     \multirow{2}{*}{Fold} & \multirow{2}{*}{Method} & \multicolumn{1}{l}{\multirow{2}{*}{Obj}} & \multicolumn{1}{l}{\multirow{2}{*}{Time(s)}} & \multicolumn{1}{l}{\multirow{2}{*}{Train\_acc}} & \multicolumn{1}{l}{\multirow{2}{*}{Test\_acc}} & \multicolumn{3}{c}{Train\_prec} & \multicolumn{3}{c}{Test\_prec} \\
\cmidrule{7-12}   &  &  &  &  &  & 
  \multicolumn{1}{l}{label $0$} &
  \multicolumn{2}{l}{label $1$} &
  \multicolumn{1}{l}{label $0$} &
  \multicolumn{2}{l}{label $1$} \\
\midrule
\multirow{4}{*}{1} & Full MIP & infeas.        & 2851 & 0.877          & 0.871          & 0.879 & 0.882 &  & 0.871 & 0.870 &  \\
                       & PIP      & infeas.        & 848  & 0.882          & 0.867          & 0.890 & 0.882 &  & 0.867 & 0.865 &  \\
                       & ISA-PIP  & infeas.        & 2936 & 0.881          & 0.874          & 0.892 & 0.879 &  & 0.881 & 0.868 &  \\
                       & IDSA-PIP & \textbf{0.835} & 1818 & \textbf{0.880} & \textbf{0.870} & 0.911 & 0.870 &  & 0.895 & 0.856 &  \\
\midrule
\multirow{4}{*}{2} & Full MIP & infeas.        & 3600 & 0.877          & 0.866          & 0.892 & 0.858 &  & 0.863 & 0.883 &  \\
                       & PIP      & infeas.        & 2426 & 0.879          & 0.870          & 0.895 & 0.857 &  & 0.878 & 0.880 &  \\
                       & ISA-PIP  & infeas.        & 5186 & 0.884          & 0.868          & 0.901 & 0.865 &  & 0.873 & 0.891 &  \\
                       & IDSA-PIP & \textbf{0.836} & 2444 & \textbf{0.879} & \textbf{0.870} & 0.914 & 0.875 &  & 0.895 & 0.890 &  \\
\midrule
\multirow{4}{*}{3} & Full MIP & 0.826          & 3600 & 0.881          & \textbf{0.860} & 0.915 & 0.854 &  & 0.886 & 0.832 &  \\
                       & PIP      & infeas.        & 3190 & 0.884          & 0.865          & 0.909 & 0.876 &  & 0.881 & 0.866 &  \\
                       & ISA-PIP  & infeas.        & 6569 & 0.879          & 0.851          & 0.909 & 0.874 &  & 0.871 & 0.847 &  \\
                       & IDSA-PIP & \textbf{0.838} & 2045 & \textbf{0.885} & 0.859          & 0.910 & 0.875 &  & 0.877 & 0.856 &  \\
\midrule
\multirow{4}{*}{4} & Full MIP & infeas.        & 3600 & 0.874          & 0.863          & 0.893 & 0.864 &  & 0.888 & 0.838 &  \\
                       & PIP      & infeas.        & 1708 & 0.877          & 0.876          & 0.888 & 0.870 &  & 0.896 & 0.851 &  \\
                       & ISA-PIP  & infeas.        & 5986 & 0.884          & 0.870          & 0.898 & 0.877 &  & 0.891 & 0.851 &  \\
                       & IDSA-PIP & \textbf{0.829} & 3651 & \textbf{0.873} & \textbf{0.860} & 0.917 & 0.876 &  & 0.905 & 0.843 &  \\
\bottomrule
\end{tabular}
\end{table}

\newpage

\subsection{Tree-based classification}
\label{app:settings tree}
Dataset information is listed in Table~\ref{tab:datasets tree}. For each dataset, we create 4 random folds and each of them consists of a training set ($50\%$), a validation set ($25\%$) and a test set ($25\%$). Training, validation and test sets are standardized: the training set is scaled to mean 0 and variance 1, with the same transformation applied to the validation set and the test set.

The numerical results are presented in Table~\ref{tab:wine}-\ref{tab:fish}. Since {\sc Gurobi} solver fails to find an optimal solution for most instances within the 3600-second time limit, we record the first time that the best solution is obtained within the time limit in Table~\ref{tab:wine}-\ref{tab:fish}. If {\sc Gurobi} solver finds an optimal solution within the 3600-second time limit, the corresponding time will be denoted with a superscript $*$. All the time values are rounded up to the nearest integers. Additional experimental parameters are provided below:
\begin{itemize}

\item \textbf{{\sc Gurobi} setting}:

All MIP problems are solved by the {\sc Gurobi} 11.0.3 solver, with parameters in Table \ref{tab:multi-gurobi-setting} set to avoid numerical issues and all others set to default. 

\begin{table}[h]
\caption{{\sc Gurobi} setting for score-based multiclass classification.}
\label{tab:multi-gurobi-setting}
\begin{tabular}{lrrrr}
\toprule
Name & MIPFocus & IntegralityFocus & FeasibilityTol & Threads  \\
\midrule
Value & 1 & 1 &  1e-09 & 4 
\\
\bottomrule
\end{tabular}
\end{table}

\item \textbf{Initial solution}:
The initial solution $(\bs{a}^0,\bs{b}^0)$ is generated by the following procedure: we first train a decision tree by using \texttt{DecisionTreeClassifier} in scikit-learn (default settings), which is not necessarily a complete binary tree. We construct a binary tree by expanding those branches with depth $<$ max\_depth by randomly selecting features and assigning the same class. After multiplying the parameter corresponding to the binary tree by $\tau_1$, we obtain the initial solution $(\bs{a}^0,\bs{b}^0)$ for fair comparison.

\begin{table}[h]
\caption{Datasets for  tree-based  classification.} \label{tab:datasets tree}
\begin{tabular}{lcccclccc}
\toprule
\multicolumn{6}{c}{Descriptive Information } & \multicolumn{3}{c}{\# of integers in Full MIP} \\
ID & Source & $N$ & $p$ & $J$ & \# of each class &  $D=2$ & $D=3$ & $D=4$\\
\midrule
\textit{wine} & \href{https://archive.ics.uci.edu/dataset/109/wine}{uci} & 178 & 13 & 3 & 59/71/48 & 2148 & 5364 & 12864 \\
\textit{nwth} & \href{https://https://archive.ics.uci.edu/dataset/102/thyroid+disease}{uci} & 215 & 5 & 3 & 150/35/30 & 2592 & 6474 & 15528 \\
\textit{htds} & \href{https://archive.ics.uci.edu/dataset/45/heart+disease}{uci} & 301 & 18 & 5 & 163/55/35/35/13 & 3632 & 9070 & 21752 \\
\textit{dmtl} & \href{https://archive.ics.uci.edu/dataset/33/dermatology}{uci} & 358 & 34 & 6 & 111/60/71/48/48/20 & 4320 & 10788 & 25872 \\
\textit{blsc} & \href{https://archive.ics.uci.edu/dataset/12/balance+scale}{uci} & 625 & 4 & 3 & 288/49/288 & 7512 & 18774 & 45048 \\
\textit{ctmc} & \href{https://archive.ics.uci.edu/dataset/30/contraceptive+method+choice}{uci} & 1473 & 11 & 3 & 629/333/511 & 17688 & 44214 & 106104 \\
\textit{ceva} & \href{https://archive.ics.uci.edu/dataset/19/car+evaluation}{uci} & 1728 & 15 & 4 & 1210/384/69/65 & 20752 & 51872 & 124480 \\
\textit{fish} & \href{https://www.kaggle.com/datasets/taweilo/fish-species-sampling-weight-and-height-data?resource=download}{kaggle} & 4080 & 3 & 9 & 476/415/468/435/475/458/418/480/455 & 48996 & 122472 & 293904 \\
\botrule
\end{tabular}
\footnotetext{Note: The name of these ID in their sources are: (1) \textit{wine}: Wine; (2) \textit{nwth}: Thyroid Disease (new-thyroid); (3) \textit{htds}: Heart Disease; (4) \textit{dmtl}: Dermatology; (5) \textit{blsc}: Balance Scale; (6) \textit{ctmc}: Contraceptive Method Choice; (7) \textit{ceva}: Car Evaluation; (8) \textit{fish}: Fish species sampling data - legnth and weight. We divide these 8 datasets into three categories based on the sample size: small datasets: \textit{wine} and \textit{nwth}; medium-sized datasets: \textit{htds}, \textit{dmtl} and \textit{blsc}; large datasets: \textit{ctmc}, \textit{ceva} and \textit{fish}.}
\end{table}

\item \textbf{{\sc Gurobi} setting}:
The MIP problems are solved by the {\sc Gurobi} 11.0.3 solver, the parameters in Table \ref{tab:gurobi setting tree} are set to avoid numerical issues, and all other parameters are default. 

\begin{table}[h]
\caption{{\sc Gurobi} setting for decision tree classification.}
\label{tab:gurobi setting tree}
\begin{tabular}{lrrrrrr}
\toprule
Name & MIPFocus & IntegralityFocus & NumericFocus & FeasibilityTol & LazyConstraints & Threads \\
\midrule
Value & 1 & 1 & 3 & 1e-09 & 1 & 32\\
\botrule
\end{tabular}
\end{table}

\item \textbf{Model parameter}:
Regarding to the norm constraints $P\triangleq \{(\bs{a},\bs{b})\in \mathbb R^{(p+1)\times (2^D-1)}:\|a^k\| _1\le \tau_1, |b_k| \le \tau_1, \|a^k\|_0 \leq \tau_0, ~\forall k \in {\cal T_B}\}$, we set $\tau_1=100$, and  $\tau_0 \in \left[\left \lceil \frac{p}{2} \right \rceil-3,\left \lceil \frac{p}{2} \right \rceil+3\right]$ for each split if the dimension of features $p>5$.  For each dataset, we only impose one precision constraint for the class with the largest sample size among all classes. The choice of precision threshold $\beta_j$ is set according to the following principle: a) on each dataset, we set the same precision threshold for the 4 folds; b)  With the initial solution, if there is no sample classified as $j$ in some of the 4 folds, we set $\beta_j$ as the proportion of class $j$ in the full dataset. Otherwise, we set $\beta_j$ as the highest precision except 1 by the initial solution among 4 folds, taking round up to two decimal places; c) We set $\beta_j=1$ only when the precision for all 4 folds are 1. The purpose of this precision setting is to make it possible for MIP to find feasible solutions. 
\item \textbf{Parameter setting in IDSA-PIP}: 
 For the solution of the subproblems by PIP, we set $r_0=0.4$ if $N<300$ else $r_0=0.2$ for depth-2 trees, $r_0=0.2$ if $N<300$ else $r_0=0.1$ for depth-3 trees, and set $r_0=0.1$ if $N<1000$ else $r_0=0.05$ for depth-4 trees. $r_{\max}=0.6,~r_{\Delta}=0.1,~\mu_{\max}=10,~\tilde{\mu}_{\max}=3$. We set $\{{\varepsilon}_{\nu}\}=\{10^{-\nu}\}_{\nu=1}^4$ as the shrinkage approximation parameter sequence in IDSA-PIP.
\item \textbf{Time limit}:
We set the time\_limit = $3600s$ for {\sc Gurobi} in solving Full MIP. We use callback and time\_limit mechanism to control the termination of subproblems in PIP. Specifically, in solving the subproblem, {\sc Gurobi} procedure will be terminated if either it reaches time\_limit ($300s$), or the objective value remains unchanged for $30s$.  
\item \textbf{Prominent models for comparison}:
We endeavor to retain the original training procedures for these models as presented in the references, except for the necessary alignment of computing resources (for models that require the use of {\sc Gurobi} or {\sc Cplex}) to ensure fair comparison. For each data set, we create 4 random folds of the data each consisting of a training set ($50\%$), a validation set ($25\%$) used to tune the hyperparameters, and a test set ($25\%$). The detailed training process of each model is as follows:
\begin{itemize}
\item \texttt{C-PIP}: solving the tree-based classification problem with precision constraint by IDSA-PIP method.
\item \texttt{U-PIP}: solving the tree-based classification problem without precision constraint by IDSA-PIP method.
\item \texttt{BendersOCT}:  only used to solve the problem without precision constraint. For each split and depth, we train a decision tree for the complexity regularization parameter $\lambda$ on the training set, then tune $\lambda$ on the validation set.  We then train on the union of the training and validation sets with the tuned regularization parameter.
\item \texttt{FlowOCT}: solve problem with precision constraint. For any given split and depth, we use the same $\lambda$ as \texttt{BendersOCT} to train a decision tree on the union of the training and validation sets.
\item \texttt{BinOCT}: there is no hyperparameters to tune, for both unconstrained (\texttt{U-BinOCT}) and precision constrained (\texttt{C-BinOCT}) model, we directly train on the union of the training and validation sets. 
\item \texttt{CART}: \texttt{DecisionTreeClassifier} in scikit-learn. For each split and depth, we tune the hyperparameters according to Table \ref{tab:tree tune parameter} on  the validation set. We then train on the union of the training and validation sets with the tuned hyperparameters.
\end{itemize}
\end{itemize}

\begin{table}[h]
\caption{Hyperparameters that are tuned for model comparison.} 
\label{tab:tree tune parameter}
\begin{tabular}{ll}
\toprule
Model & Hyperparameters \\
\midrule
C-PIP & $\ell_1$-norm constraint $\tau_1 = 100$, $\ell_0$-norm constraint $\tau_0\in \left[\left \lceil p/2 \right \rceil-3,\left \lceil p/2 \right \rceil+3\right]$\\
\midrule
U-PIP & $\tau_1,\,\tau_0$ the same as \texttt{C-PIP}\\
\midrule
BendersOCT & Regularization parameter $\lambda$: $[0, 0.1, 0.2, \dots,0.9]$ \\
\midrule
FlowOCT  & $\lambda$ the same as \texttt{BendersOCT} \\
\midrule
\multirow{3}{*}{CART} &  `criterion': [`gini', `entropy'], `min\_samples\_split': [2, 5, 10],  \\
        & `min\_samples\_leaf': [1, 2, 5],    `max\_features': [None, `sqrt', `log2'], \\
        & `splitter': [`best', `random']\\
\bottomrule
\end{tabular}
\end{table}

\begin{table}[h]
\caption{Results of tree-based classification on \textit{wine} dataset with one single precision constraint} \label{tab:wine}
\begin{tabular}{clllrrrrrr}   
\toprule
\multicolumn{1}{l}{Depth} & \multicolumn{1}{l}{$\beta_j$} & \multicolumn{1}{l}{Fold} & Method & \multicolumn{1}{l}{Obj} & \multicolumn{1}{l}{Time(s)} &
\multicolumn{1}{l}{Train\_acc} & \multicolumn{1}{l}{Test\_acc} & \multicolumn{1}{l}{Train\_prec} & \multicolumn{1}{l}{Test\_prec} \\ 
\midrule
\multirow{8}{*}{2} & \multirow{8}{*}{0.98} & \multirow{2}{*}{1} & Full MIP & \textbf{1.000} & $^*$989 & \textbf{1.000} & \textbf{0.911} & 1.000 & 1.000 \\
      &       &       & IDSA-PIP & \textbf{1.000} & 37    & \textbf{1.000} & 0.867 & 1.000 & 1.000 \\
\cmidrule{3-10}      &       & \multirow{2}{*}{2} & Full MIP & \textbf{1.000} & $^*$82 & \textbf{1.000} & \textbf{0.933} & 1.000 & 0.857 \\
      &       &       & IDSA-PIP & \textbf{1.000} & 72    & \textbf{1.000} & \textbf{0.933} & 1.000 & 0.944 \\
\cmidrule{3-10}      &       & \multirow{2}{*}{3} & Full MIP & \textbf{1.000} & $^*$1059 & \textbf{1.000} & \textbf{0.911} & 1.000 & 0.889 \\
      &       &       & IDSA-PIP & \textbf{1.000} & 125   & \textbf{1.000} & \textbf{0.911} & 1.000 & 0.938 \\
\cmidrule{3-10}      &       & \multirow{2}{*}{4} & Full MIP & 0.977 & 2975  & 0.977 & 0.933 & 1.000 & 1.000 \\
      &       &       & IDSA-PIP & \textbf{0.992} & 497   & \textbf{0.992} & \textbf{0.978} & 1.000 & 1.000 \\
\midrule
\multirow{8}{*}{3} & \multirow{8}{*}{0.99} & \multirow{2}{*}{1} & Full MIP & \textbf{1.000} & $^*$68 & \textbf{1.000} & 0.778 & 1.000 & 0.800 \\
      &       &       & IDSA-PIP & \textbf{1.000} & 154   & \textbf{1.000} & \textbf{0.889} & 1.000 & 0.882 \\
\cmidrule{3-10}      &       & \multirow{2}{*}{2} & Full MIP & \textbf{1.000} & $^*$92 & \textbf{1.000} & 0.911 & 1.000 & 0.818 \\
      &       &       & IDSA-PIP & \textbf{1.000} & 5     & \textbf{1.000} & \textbf{0.933} & 1.000 & 1.000 \\
\cmidrule{3-10}      &       & \multirow{2}{*}{3} & Full MIP & 0.669 & 179   & 0.669 & 0.667 & 1.000 & 1.000 \\
      &       &       & IDSA-PIP & \textbf{1.000} & 328   & \textbf{1.000} & \textbf{0.867} & 1.000 & 0.929 \\
\cmidrule{3-10}      &       & \multirow{2}{*}{4} & Full MIP & \textbf{1.000} & $^*$447 & \textbf{1.000} & \textbf{0.933} & 1.000 & 0.941 \\
      &       &       & IDSA-PIP & \textbf{1.000} & 3     & \textbf{1.000} & 0.911 & 1.000 & 1.000 \\
\midrule
\multirow{8}{*}{4} & \multirow{8}{*}{0.99} & \multirow{2}{*}{1} & Full MIP & \textbf{1.000} & $^*$3601 & \textbf{1.000} & 0.844 & 1.000 & 0.824 \\
      &       &       & IDSA-PIP & \textbf{1.000} & 7     & \textbf{1.000} & \textbf{0.867} & 1.000 & 0.875 \\
\cmidrule{3-10}      &       & \multirow{2}{*}{2} & Full MIP & \textbf{1.000} & $^*$510 & \textbf{1.000} & 0.911 & 1.000 & 0.818 \\
      &       &       & IDSA-PIP & \textbf{1.000} & 8     & \textbf{1.000} & \textbf{0.956} & 1.000 & 0.944 \\
\cmidrule{3-10}      &       & \multirow{2}{*}{3} & Full MIP & \textbf{1.000} & $^*$1766 & \textbf{1.000} & 0.889 & 1.000 & 0.941 \\
      &       &       & IDSA-PIP & \textbf{1.000} & 44    & \textbf{1.000} & \textbf{0.911} & 1.000 & 1.000 \\
\cmidrule{3-10}      &       & \multirow{2}{*}{4} & Full MIP & \textbf{1.000} & $^*$292 & \textbf{1.000} & 0.911 & 1.000 & 0.938 \\
      &       &       & IDSA-PIP & \textbf{1.000} & 6     & \textbf{1.000} & \textbf{0.933} & 1.000 & 0.941 \\
\bottomrule
\end{tabular}
\footnotetext{Note: The symbol * in column ``Time(s)" means that an optimal solution is found and the optimality gap is 0.}
\end{table}

\begin{table}[h]
\caption{Results of tree-based classification on \textit{nwth} dataset with one single precision constraint} 
\begin{tabular}{clllrrrrrr}   
\toprule
\multicolumn{1}{l}{Depth} & \multicolumn{1}{l}{$\beta_j$} & \multicolumn{1}{l}{Fold} & Method & \multicolumn{1}{l}{Obj} & \multicolumn{1}{l}{Time(s)} &
\multicolumn{1}{l}{Train\_acc} & \multicolumn{1}{l}{Test\_acc} & \multicolumn{1}{l}{Train\_prec} & \multicolumn{1}{l}{Test\_prec} \\ 
\midrule
\multirow{8}{*}{2} & \multirow{8}{*}{0.97} & \multirow{2}{*}{1} & Full MIP & \textbf{1.000} & $^*$1484 & \textbf{1.000} & \textbf{0.963} & 1.000 & 0.974 \\
      &       &       & IDSA-PIP & \textbf{1.000} & 687   & \textbf{1.000} & \textbf{0.963} & 1.000 & 0.974 \\
\cmidrule{3-10}      &       & \multirow{2}{*}{2} & Full MIP & \textbf{1.000} & $^*$358 & \textbf{1.000} & \textbf{0.963} & 1.000 & 0.974 \\
      &       &       & IDSA-PIP & 0.994 & 759   & 0.994 & 0.907 & 1.000 & 0.946 \\
\cmidrule{3-10}      &       & \multirow{2}{*}{3} & Full MIP & \textbf{1.000} & $^*$252 & \textbf{1.000} & \textbf{0.907} & 1.000 & 0.923 \\
      &       &       & IDSA-PIP & \textbf{1.000} & 118   & \textbf{1.000} & \textbf{0.907} & 1.000 & 0.902 \\
\cmidrule{3-10}      &       & \multirow{2}{*}{4} & Full MIP & \textbf{1.000} & $^*$155 & \textbf{1.000} & \textbf{0.889} & 1.000 & 0.900 \\
      &       &       & IDSA-PIP & \textbf{1.000} & 37    & \textbf{1.000} & \textbf{0.889} & 1.000 & 0.921 \\
\midrule
\multirow{8}{*}{3} & \multirow{8}{*}{0.99} & \multirow{2}{*}{1} & Full MIP & 0.311 & 1399  & 0.311 & 0.259 & 1.000 & - \\
      &       &       & IDSA-PIP & \textbf{0.783} & 1275  & \textbf{0.783} & \textbf{0.741} & 1.000 & 0.900 \\
\cmidrule{3-10}      &       & \multirow{2}{*}{2} & Full MIP & 0.398 & 3065  & 0.398 & 0.370 & 1.000 & 1.000 \\
      &       &       & IDSA-PIP & \textbf{0.814} & 953   & \textbf{0.814} & \textbf{0.648} & 1.000 & 1.000 \\
\cmidrule{3-10}      &       & \multirow{2}{*}{3} & Full MIP & 0.311 & 429   & 0.311 & 0.315 & 1.000 & 1.000 \\
      &       &       & IDSA-PIP & \textbf{0.994} & 1256  & \textbf{0.994} & \textbf{0.907} & 1.000 & 0.902 \\
\cmidrule{3-10}      &       & \multirow{2}{*}{4} & Full MIP & \textbf{1.000} & $^*$495 & \textbf{1.000} & \textbf{0.907} & 1.000 & 0.923 \\
      &       &       & IDSA-PIP & 0.994 & 995   & 0.994 & \textbf{0.907} & 1.000 & 0.902 \\
\midrule
\multirow{8}{*}{4} & \multirow{8}{*}{1} & \multirow{2}{*}{1} & Full MIP & 0.354 & 3525  & 0.354 & 0.296 & 1.000 & 1.000 \\
      &       &       & IDSA-PIP & \textbf{0.739} & 1329  & \textbf{0.745} & \textbf{0.611} & 1.000 & 0.947 \\
\cmidrule{3-10}      &       & \multirow{2}{*}{2} & Full MIP & 0.304 & 31    & 0.304 & 0.296 & 1.000 & - \\
      &       &       & IDSA-PIP & \textbf{1.000} & 1182  & \textbf{1.000} & \textbf{0.981} & 1.000 & 1.000 \\
\cmidrule{3-10}      &       & \multirow{2}{*}{3} & Full MIP & 0.329 & 182   & 0.329 & 0.333 & 1.000 & 1.000 \\
      &       &       & IDSA-PIP & \textbf{0.981} & 1269  & \textbf{0.981} & \textbf{0.889} & 1.000 & 0.921 \\
\cmidrule{3-10}      &       & \multirow{2}{*}{4} & Full MIP & 0.335 & 29    & 0.335 & 0.259 & 1.000 & 0.667 \\
      &       &       & IDSA-PIP & \textbf{1.000} & 933   & \textbf{1.000} & \textbf{0.907} & 1.000 & 0.946 \\
\bottomrule
\end{tabular}
\footnotetext{Note: The symbol * in column      ``Time(s)" means that an optimal solution is found and the optimality gap is 0. The symbol - in column ``Test\_prec" means that no samples on the test set are assigned to class $j$.}
\end{table}

\begin{table}[h]
\caption{Results of tree-based classification on \textit{htds} dataset with one single precision constraint } 
\begin{tabular}{clllrrrrrr}   
\toprule
\multicolumn{1}{l}{Depth} & \multicolumn{1}{l}{$\beta_j$} & \multicolumn{1}{l}{Fold} & Method & \multicolumn{1}{l}{Obj} & \multicolumn{1}{l}{Time(s)} &
\multicolumn{1}{l}{Train\_acc} & \multicolumn{1}{l}{Test\_acc} & \multicolumn{1}{l}{Train\_prec} & \multicolumn{1}{l}{Test\_prec} \\ 
\midrule
\multirow{8}{*}{2} & \multirow{8}{*}{0.77} & \multirow{2}{*}{1} & Full MIP & 0.636 & 2234  & 0.636 & \textbf{0.579} & 0.804 & 0.696 \\
      &       &       & IDSA-PIP & \textbf{0.738} & 310   & \textbf{0.738} & \textbf{0.579} & 0.840 & 0.704 \\
\cmidrule{3-10}      &       & \multirow{2}{*}{2} & Full MIP & 0.604 & 2065  & 0.604 & 0.500 & 0.805 & 0.744 \\
      &       &       & IDSA-PIP & \textbf{0.764} & 542   & \textbf{0.764} & \textbf{0.539} & 0.880 & 0.795 \\
\cmidrule{3-10}      &       & \multirow{2}{*}{3} & Full MIP & 0.591 & 1021  & 0.591 & 0.461 & 0.803 & 0.732 \\
      &       &       & IDSA-PIP & \textbf{0.747} & 583   & \textbf{0.747} & \textbf{0.645} & 0.858 & 0.796 \\
\cmidrule{3-10}      &       & \multirow{2}{*}{4} & Full MIP & 0.636 & 3373  & 0.636 & 0.474 & 0.932 & 0.743 \\
      &       &       & IDSA-PIP & \textbf{0.707} & 699   & \textbf{0.707} & \textbf{0.579} & 0.852 & 0.692 \\
\midrule
\multirow{8}{*}{3} & \multirow{8}{*}{0.83} & \multirow{2}{*}{1} & Full MIP & 0.600 & 3601  & 0.600 & \textbf{0.500} & 0.853 & 0.762 \\
      &       &       & IDSA-PIP & \textbf{0.791} & 817   & \textbf{0.791} & 0.461 & 0.890 & 0.667 \\
\cmidrule{3-10}      &       & \multirow{2}{*}{2} & Full MIP & 0.676 & 1476  & 0.676 & \textbf{0.487} & 0.868 & 0.744 \\
      &       &       & IDSA-PIP & \textbf{0.773} & 605   & \textbf{0.773} & 0.447 & 0.982 & 0.788 \\
\cmidrule{3-10}      &       & \multirow{2}{*}{3} & Full MIP & 0.618 & 2849  & 0.622 & \textbf{0.526} & 0.841 & 0.800 \\
      &       &       & IDSA-PIP & \textbf{0.778} & 944   & \textbf{0.778} & 0.513 & 0.917 & 0.733 \\
\cmidrule{3-10}      &       & \multirow{2}{*}{4} & Full MIP & 0.613 & 91    & 0.613 & \textbf{0.513} & 0.850 & 0.756 \\
      &       &       & IDSA-PIP & \textbf{0.831} & 634   & \textbf{0.831} & \textbf{0.513} & 0.916 & 0.674 \\
\midrule
\multirow{8}{*}{4} & \multirow{8}{*}{0.83} & \multirow{2}{*}{1} & Full MIP & 0.733 & 3601  & 0.733 & \textbf{0.539} & 0.866 & 0.720 \\
      &       &       & IDSA-PIP & \textbf{0.938} & 1412  & \textbf{0.938} & 0.513 & 0.960 & 0.786 \\
\cmidrule{3-10}      &       & \multirow{2}{*}{2} & Full MIP & 0.702 & 170   & 0.702 & 0.500 & 0.837 & 0.708 \\
      &       &       & IDSA-PIP & \textbf{0.880} & 1290  & \textbf{0.880} & \textbf{0.513} & 0.937 & 0.800 \\
\cmidrule{3-10}      &       & \multirow{2}{*}{3} & Full MIP & 0.676 & 306   & 0.676 & \textbf{0.553} & 0.848 & 0.745 \\
      &       &       & IDSA-PIP & \textbf{0.880} & 1835  & \textbf{0.880} & \textbf{0.553} & 0.938 & 0.733 \\
\cmidrule{3-10}      &       & \multirow{2}{*}{4} & Full MIP & 0.733 & 3601  & 0.733 & \textbf{0.553} & 0.860 & 0.745 \\
      &       &       & IDSA-PIP & \textbf{0.898} & 1465  & \textbf{0.898} & 0.526 & 0.944 & 0.762 \\
\bottomrule
\end{tabular}
\end{table}

\begin{table}[h]
\caption{Results of tree-based classification on \textit{dmtl} dataset with one single precision constraint} 
\begin{tabular}{clllrrrrrr}   
\toprule
\multicolumn{1}{l}{Depth} & \multicolumn{1}{l}{$\beta_j$} & \multicolumn{1}{l}{Fold} & Method & \multicolumn{1}{l}{Obj} & \multicolumn{1}{l}{Time(s)} &
\multicolumn{1}{l}{Train\_acc} & \multicolumn{1}{l}{Test\_acc} & \multicolumn{1}{l}{Train\_prec} & \multicolumn{1}{l}{Test\_prec} \\ 
\midrule
\multirow{8}{*}{2} & \multirow{8}{*}{1.00} & \multirow{2}{*}{1} & Full MIP & \textbf{0.731} & 1400  & \textbf{0.731} & 0.678 & 1.000 & 0.960 \\
      &       &       & IDSA-PIP & \textbf{0.731} & 4     & \textbf{0.731} & \textbf{0.689} & 1.000 & 0.926 \\
\cmidrule{3-10}      &       & \multirow{2}{*}{2} & Full MIP & \textbf{0.731} & 2922  & \textbf{0.731} & 0.656 & 1.000 & 0.900 \\
      &       &       & IDSA-PIP & 0.728 & 76    & 0.728 & \textbf{0.711} & 1.000 & 0.933 \\
\cmidrule{3-10}      &       & \multirow{2}{*}{3} & Full MIP & \textbf{0.731} & 2127  & \textbf{0.731} & \textbf{0.711} & 1.000 & 0.931 \\
      &       &       & IDSA-PIP & 0.724 & 460   & 0.724 & 0.700 & 1.000 & 0.931 \\
\cmidrule{3-10}      &       & \multirow{2}{*}{4} & Full MIP & 0.724 & 3102  & 0.724 & \textbf{0.722} & 1.000 & 0.931 \\
      &       &       & IDSA-PIP & \textbf{0.728} & 51    & \textbf{0.728} & 0.689 & 1.000 & 0.903 \\
\midrule
\multirow{8}{*}{3} & \multirow{8}{*}{1.00} & \multirow{2}{*}{1} & Full MIP & \textbf{0.813} & 1     & \textbf{0.813} & \textbf{0.767} & 1.000 & 0.962 \\
      &       &       & IDSA-PIP & \textbf{0.813} & 491   & \textbf{0.813} & \textbf{0.767} & 1.000 & 0.962 \\
\cmidrule{3-10}      &       & \multirow{2}{*}{2} & Full MIP & \textbf{0.802} & 1     & \textbf{0.802} & \textbf{0.811} & 1.000 & 1.000 \\
      &       &       & IDSA-PIP & \textbf{0.802} & 431   & \textbf{0.802} & \textbf{0.811} & 1.000 & 1.000 \\
\cmidrule{3-10}      &       & \multirow{2}{*}{3} & Full MIP & 0.799 & 1     & 0.799 & 0.822 & 1.000 & 1.000 \\
      &       &       & IDSA-PIP & \textbf{1.000} & 579   & \textbf{1.000} & \textbf{0.989} & 1.000 & 1.000 \\
\cmidrule{3-10}      &       & \multirow{2}{*}{4} & Full MIP & 0.799 & 1     & 0.799 & 0.778 & 1.000 & 0.964 \\
      &       &       & IDSA-PIP & \textbf{0.899} & 777   & \textbf{0.899} & \textbf{0.811} & 1.000 & 0.964 \\
\midrule
\multirow{8}{*}{4} & \multirow{8}{*}{1.00} & \multirow{2}{*}{1} & Full MIP & \textbf{0.929} & 3601  & \textbf{0.929} & 0.878 & 1.000 & 1.000 \\
      &       &       & IDSA-PIP & 0.914 & 583   & 0.914 & \textbf{0.889} & 1.000 & 1.000 \\
\cmidrule{3-10}      &       & \multirow{2}{*}{2} & Full MIP & \textbf{0.907} & 1     & \textbf{0.907} & \textbf{0.933} & 1.000 & 1.000 \\
      &       &       & IDSA-PIP & \textbf{0.907} & 541   & \textbf{0.907} & \textbf{0.933} & 1.000 & 1.000 \\
\cmidrule{3-10}      &       & \multirow{2}{*}{3} & Full MIP & \textbf{0.903} & 1     & \textbf{0.903} & \textbf{0.933} & 1.000 & 0.966 \\
      &       &       & IDSA-PIP & \textbf{0.903} & 726   & \textbf{0.903} & \textbf{0.933} & 1.000 & 0.966 \\
\cmidrule{3-10}      &       & \multirow{2}{*}{4} & Full MIP & \textbf{0.910} & 1     & \textbf{0.910} & \textbf{0.867} & 1.000 & 0.964 \\
      &       &       & IDSA-PIP & \textbf{0.910} & 507   & \textbf{0.910} & \textbf{0.867} & 1.000 & 0.964 \\
\bottomrule
\end{tabular}
\footnotetext{Note: The extremely short times of Full MIP indicate that the corresponding objective values have not changed from the very beginning to the end of the optimization process.}
\end{table}

\begin{table}[h]
\caption{Results of tree-based classification on \textit{blsc} dataset with one single precision constraint} 
\begin{tabular}{clllrrrrrr}   
\toprule
\multicolumn{1}{l}{Depth} & \multicolumn{1}{l}{$\beta_j$} & \multicolumn{1}{l}{Fold} & Method & \multicolumn{1}{l}{Obj} & \multicolumn{1}{l}{Time(s)} &
\multicolumn{1}{l}{Train\_acc} & \multicolumn{1}{l}{Test\_acc} & \multicolumn{1}{l}{Train\_prec} & \multicolumn{1}{l}{Test\_prec} \\ 
\midrule
\multirow{8}{*}{2} & \multirow{8}{*}{0.75} & \multirow{2}{*}{1} & Full MIP & 0.692 & 38    & 0.692 & 0.662 & 0.792 & 0.737 \\
      &       &       & IDSA-PIP & \textbf{0.912} & 668   & \textbf{0.912} & \textbf{0.873} & 0.895 & 0.863 \\
\cmidrule{3-10}      &       & \multirow{2}{*}{2} & Full MIP & 0.895 & 3044  & 0.897 & 0.860 & 0.906 & 0.841 \\
      &       &       & IDSA-PIP & \textbf{0.923} & 894   & \textbf{0.923} & \textbf{0.924} & 0.945 & 0.945 \\
\cmidrule{3-10}      &       & \multirow{2}{*}{3} & Full MIP & 0.731 & 2043  & 0.731 & 0.694 & 0.915 & 0.864 \\
      &       &       & IDSA-PIP & \textbf{0.947} & 1045  & \textbf{0.949} & \textbf{0.860} & 0.986 & 0.954 \\
\cmidrule{3-10}      &       & \multirow{2}{*}{4} & Full MIP & 0.865 & 2098  & 0.868 & 0.828 & 0.915 & 0.870 \\
      &       &       & IDSA-PIP & \textbf{0.927} & 792   & \textbf{0.927} & \textbf{0.866} & 0.946 & 0.877 \\
\midrule
\multirow{8}{*}{3} & \multirow{8}{*}{0.82} & \multirow{2}{*}{1} & Full MIP & 0.746 & 3601  & 0.746 & 0.688 & 0.841 & 0.792 \\
      &       &       & IDSA-PIP & \textbf{0.944} & 1561  & \textbf{0.944} & \textbf{0.879} & 0.972 & 0.932 \\
\cmidrule{3-10}      &       & \multirow{2}{*}{2} & Full MIP & 0.780 & 3603  & 0.780 & 0.688 & 0.835 & 0.789 \\
      &       &       & IDSA-PIP & \textbf{0.929} & 1493  & \textbf{0.929} & \textbf{0.879} & 0.919 & 0.854 \\
\cmidrule{3-10}      &       & \multirow{2}{*}{3} & Full MIP & 0.759 & 89    & 0.759 & 0.732 & 0.834 & 0.842 \\
      &       &       & IDSA-PIP & \textbf{0.932} & 1380  & \textbf{0.932} & \textbf{0.904} & 0.964 & 0.944 \\
\cmidrule{3-10}      &       & \multirow{2}{*}{4} & Full MIP & 0.729 & 847   & 0.729 & 0.694 & 0.839 & 0.789 \\
      &       &       & IDSA-PIP & \textbf{0.900} & 1002  & \textbf{0.900} & \textbf{0.854} & 0.909 & 0.857 \\
\midrule
\multirow{8}{*}{4} & \multirow{8}{*}{0.87} & \multirow{2}{*}{1} & Full MIP & 0.838 & 1366  & 0.838 & 0.777 & 0.875 & 0.812 \\
      &       &       & IDSA-PIP & \textbf{0.923} & 1291  & \textbf{0.923} & \textbf{0.834} & 0.946 & 0.899 \\
\cmidrule{3-10}      &       & \multirow{2}{*}{2} & Full MIP & 0.818 & 386   & 0.818 & 0.739 & 0.876 & 0.864 \\
      &       &       & IDSA-PIP & \textbf{0.968} & 1893  & \textbf{0.968} & \textbf{0.873} & 0.986 & 0.914 \\
\cmidrule{3-10}      &       & \multirow{2}{*}{3} & Full MIP & 0.806 & 987   & 0.806 & 0.745 & 0.871 & 0.860 \\
      &       &       & IDSA-PIP & \textbf{0.964} & 1538  & \textbf{0.964} & \textbf{0.879} & 0.973 & 0.931 \\
\cmidrule{3-10}      &       & \multirow{2}{*}{4} & Full MIP & 0.825 & 1059  & 0.825 & 0.783 & 0.892 & 0.859 \\
      &       &       & IDSA-PIP & \textbf{0.962} & 1600  & \textbf{0.962} & \textbf{0.873} & 0.982 & 0.942 \\
\bottomrule
\end{tabular}
\end{table}

\begin{table}[h]
\caption{Results of tree-based classification on \textit{ctmc} dataset with one single precision constraint} 
\begin{tabular}{clllrrrrrr}   
\toprule
\multicolumn{1}{l}{Depth} & \multicolumn{1}{l}{$\beta_j$} & \multicolumn{1}{l}{Fold} & Method & \multicolumn{1}{l}{Obj} & \multicolumn{1}{l}{Time(s)} &
\multicolumn{1}{l}{Train\_acc} & \multicolumn{1}{l}{Test\_acc} & \multicolumn{1}{l}{Train\_prec} & \multicolumn{1}{l}{Test\_prec} \\ 
\midrule
\multirow{8}{*}{2} & \multirow{8}{*}{0.63} & \multirow{2}{*}{1} & Full MIP & 0.433 & 1499  & 0.433 & 0.425 & 1.000 & 0.935 \\
      &       &       & IDSA-PIP & \textbf{0.540} & 1512  & \textbf{0.540} & \textbf{0.504} & 0.700 & 0.696 \\
\cmidrule{3-10}      &       & \multirow{2}{*}{2} & Full MIP & 0.525 & 2467  & 0.525 & 0.472 & 0.644 & 0.600 \\
      &       &       & IDSA-PIP & \textbf{0.544} & 1276  & \textbf{0.545} & \textbf{0.455} & 0.753 & 0.646 \\
\cmidrule{3-10}      &       & \multirow{2}{*}{3} & Full MIP & 0.428 & 3601  & 0.429 & 0.442 & 0.975 & 0.971 \\
      &       &       & IDSA-PIP & \textbf{0.524} & 1038  & \textbf{0.524} & \textbf{0.472} & 0.745 & 0.667 \\
\cmidrule{3-10}      &       & \multirow{2}{*}{4} & Full MIP & 0.427 & 3601  & 0.427 & 0.398 & 0.987 & 1.000 \\
      &       &       & IDSA-PIP & \textbf{0.528} & 1461  & \textbf{0.529} & \textbf{0.512} & 0.699 & 0.660 \\
\midrule
\multirow{8}{*}{3} & \multirow{8}{*}{0.81} & \multirow{2}{*}{1} & Full MIP & 0.507 & 2024  & 0.507 & 0.491 & 0.822 & 0.770 \\
      &       &       & IDSA-PIP & \textbf{0.583} & 1990  & \textbf{0.583} & \textbf{0.512} & 0.824 & 0.739 \\
\cmidrule{3-10}      &       & \multirow{2}{*}{2} & Full MIP & 0.453 & 469   & 0.453 & 0.393 & 0.988 & 1.000 \\
      &       &       & IDSA-PIP & \textbf{0.595} & 1618  & \textbf{0.595} & \textbf{0.477} & 0.839 & 0.716 \\
\cmidrule{3-10}      &       & \multirow{2}{*}{3} & Full MIP & 0.433 & 1752  & 0.433 & 0.469 & 0.985 & 1.000 \\
      &       &       & IDSA-PIP & \textbf{0.476} & 1045  & \textbf{0.477} & \textbf{0.450} & 0.833 & 0.818 \\
\cmidrule{3-10}      &       & \multirow{2}{*}{4} & Full MIP & 0.444 & 700   & 0.444 & 0.431 & 0.987 & 1.000 \\
      &       &       & IDSA-PIP & \textbf{0.479} & 1133  & \textbf{0.479} & \textbf{0.461} & 1.000 & 0.935 \\
\midrule
\multirow{8}{*}{4} & \multirow{8}{*}{0.69} & \multirow{2}{*}{1} & Full MIP & 0.562 & 3601  & 0.562 & 0.523 & 0.814 & 0.759 \\
      &       &       & IDSA-PIP & \textbf{0.620} & 1572  & \textbf{0.621} & \textbf{0.537} & 0.714 & 0.623 \\
\cmidrule{3-10}      &       & \multirow{2}{*}{2} & Full MIP & 0.560 & 1949  & 0.560 & 0.491 & 0.789 & 0.766 \\
      &       &       & IDSA-PIP & \textbf{0.626} & 1317  & \textbf{0.626} & \textbf{0.512} & 0.709 & 0.645 \\
\cmidrule{3-10}      &       & \multirow{2}{*}{3} & Full MIP & infeas. & 3462  & 0.514 & 0.512 & 0.678 & 0.667 \\
      &       &       & IDSA-PIP & \textbf{0.575} & 1028  & \textbf{0.575} & \textbf{0.564} & 0.776 & 0.683 \\
\cmidrule{3-10}      &       & \multirow{2}{*}{4} & Full MIP & 0.558 & 2161  & 0.558 & 0.520 & 0.783 & 0.764 \\
      &       &       & IDSA-PIP & \textbf{0.631} & 2135  & \textbf{0.631} & \textbf{0.542} & 0.721 & 0.629 \\
\bottomrule
\end{tabular}
\end{table}

\begin{table}[h]
\caption{Results of tree-based classification on \textit{ceva} dataset with one single precision constraint} 
\begin{tabular}{clllrrrrrr}   
\toprule
\multicolumn{1}{l}{Depth} & \multicolumn{1}{l}{$\beta_j$} & \multicolumn{1}{l}{Fold} & Method & \multicolumn{1}{l}{Obj} & \multicolumn{1}{l}{Time(s)} &
\multicolumn{1}{l}{Train\_acc} & \multicolumn{1}{l}{Test\_acc} & \multicolumn{1}{l}{Train\_prec} & \multicolumn{1}{l}{Test\_prec} \\ 
\midrule
\multirow{8}{*}{2} & \multirow{8}{*}{0.82} & \multirow{2}{*}{1} & Full MIP & 0.566 & 129   & 0.566 & 0.525 & 1.000 & 1.000 \\
      &       &       & IDSA-PIP & \textbf{0.919} & 1198  & \textbf{0.919} & \textbf{0.891} & 0.999 & 0.986 \\
\cmidrule{3-10}      &       & \multirow{2}{*}{2} & Full MIP & 0.555 & 373   & 0.555 & 0.558 & 1.000 & 1.000 \\
      &       &       & IDSA-PIP & \textbf{0.886} & 1032  & \textbf{0.886} & \textbf{0.875} & 0.973 & 0.958 \\
\cmidrule{3-10}      &       & \multirow{2}{*}{3} & Full MIP & 0.709 & 3246  & 0.709 & 0.725 & 0.826 & 0.836 \\
      &       &       & IDSA-PIP & \textbf{0.841} & 1079  & \textbf{0.841} & \textbf{0.838} & 0.909 & 0.889 \\
\cmidrule{3-10}      &       & \multirow{2}{*}{4} & Full MIP & \textbf{0.555} & 37    & \textbf{0.555} & \textbf{0.558} & 1.000 & 1.000 \\
      &       &       & IDSA-PIP & \textbf{0.555} & 547   & \textbf{0.555} & \textbf{0.558} & 1.000 & 1.000 \\
\midrule
\multirow{8}{*}{3} & \multirow{8}{*}{0.97} & \multirow{2}{*}{1} & Full MIP & infeas. & 479  & 0.791 & 0.806 & 0.944 & 0.970 \\
      &       &       & IDSA-PIP & \textbf{0.906} & 1657  & \textbf{0.906} & \textbf{0.877} & 0.992 & 0.980 \\
\cmidrule{3-10}      &       & \multirow{2}{*}{2} & Full MIP & infeas. & 2996  & 0.791 & 0.794 & 0.966 & 0.966 \\
      &       &       & IDSA-PIP & \textbf{0.856} & 1256  & \textbf{0.856} & \textbf{0.863} & 0.986 & 0.986 \\
\cmidrule{3-10}      &       & \multirow{2}{*}{3} & Full MIP & 0.771 & 1713  & 0.771 & 0.799 & 1.000 & 1.000 \\
      &       &       & IDSA-PIP & \textbf{0.847} & 1237  & \textbf{0.847} & \textbf{0.863} & 0.988 & 0.989 \\
\cmidrule{3-10}      &       & \multirow{2}{*}{4} & Full MIP & 0.775 & 1311  & 0.775 & 0.785 & 1.000 & 1.000 \\
      &       &       & IDSA-PIP & \textbf{0.904} & 1951  & \textbf{0.904} & \textbf{0.894} & 0.995 & 0.993 \\
\midrule
\multirow{8}{*}{4} & \multirow{8}{*}{0.96} & \multirow{2}{*}{1} & Full MIP & 0.799 & 3428  & 0.799 & 0.792 & 0.982 & 0.996 \\
      &       &       & IDSA-PIP & \textbf{0.911} & 2399  & \textbf{0.911} & \textbf{0.907} & 1.000 & 1.000 \\
\cmidrule{3-10}      &       & \multirow{2}{*}{2} & Full MIP & infeas. & 2  & 0.816 & 0.803 & 0.957 & 0.937 \\
      &       &       & IDSA-PIP & \textbf{0.812} & 929   & \textbf{0.812} & \textbf{0.812} & 0.982 & 0.985 \\
\cmidrule{3-10}      &       & \multirow{2}{*}{3} & Full MIP & 0.792 & 2695  & 0.792 & 0.819 & 0.987 & 0.971 \\
      &       &       & IDSA-PIP & \textbf{0.829} & 2005  & \textbf{0.829} & \textbf{0.850} & 0.967 & 0.953 \\
\cmidrule{3-10}      &       & \multirow{2}{*}{4} & Full MIP & infeas. & 2349  & 0.807 & 0.801 & 0.915 & 0.903 \\
      &       &       & IDSA-PIP & \textbf{0.806} & 750   & \textbf{0.806} & \textbf{0.806} & 0.967 & 0.967 \\
\bottomrule
\end{tabular}
\end{table}

\begin{table}[h]
\caption{Results of tree-based classification on \textit{fish} dataset with one single precision constraint} \label{tab:fish}
\begin{tabular}{clllrrrrrr}   
\toprule
\multicolumn{1}{l}{Depth} & \multicolumn{1}{l}{$\beta_j$} & \multicolumn{1}{l}{Fold} & Method & \multicolumn{1}{l}{Obj} & \multicolumn{1}{l}{Time(s)} &
\multicolumn{1}{l}{Train\_acc} & \multicolumn{1}{l}{Test\_acc} & \multicolumn{1}{l}{Train\_prec} & \multicolumn{1}{l}{Test\_prec} \\ 
\midrule
\multirow{8}{*}{2} & \multirow{8}{*}{0.22} & \multirow{2}{*}{1} & Full MIP & infeas. & 1  & 0.462 & 0.463 & 0.217 & 0.216 \\
      &       &       & IDSA-PIP & \textbf{0.462} & 453   & \textbf{0.462} & \textbf{0.463} & 0.235 & 0.234 \\
\cmidrule{3-10}      &       & \multirow{2}{*}{2} & Full MIP & infeas. & 1  & 0.462 & 0.463 & 0.217 & 0.216 \\
      &       &       & IDSA-PIP & \textbf{0.462} & 422   & \textbf{0.462} & \textbf{0.463} & 0.233 & 0.234 \\
\cmidrule{3-10}      &       & \multirow{2}{*}{3} & Full MIP & infeas. & 1  & 0.462 & 0.463 & 0.217 & 0.216 \\
      &       &       & IDSA-PIP & \textbf{0.462} & 450   & \textbf{0.462} & \textbf{0.463} & 0.234 & 0.234 \\
\cmidrule{3-10}      &       & \multirow{2}{*}{4} & Full MIP & infeas. & 1  & 0.458 & 0.458 & 0.213 & 0.213 \\
      &       &       & IDSA-PIP & \textbf{0.462} & 609   & \textbf{0.462} & \textbf{0.461} & 0.269 & 0.269 \\
\midrule
\multirow{8}{*}{3} & \multirow{8}{*}{0.28} & \multirow{2}{*}{1} & Full MIP & infeas. & 1  & 0.679 & 0.675 & 0.272 & 0.271 \\
      &       &       & IDSA-PIP & \textbf{0.681} & 603   & \textbf{0.681} & \textbf{0.681} & 0.280 & 0.280 \\
\cmidrule{3-10}      &       & \multirow{2}{*}{2} & Full MIP & infeas. & 1  & 0.678 & 0.678 & 0.272 & 0.271 \\
      &       &       & IDSA-PIP & \textbf{0.681} & 509   & \textbf{0.681} & \textbf{0.679} & 0.281 & 0.281 \\
\cmidrule{3-10}      &       & \multirow{2}{*}{3} & Full MIP & infeas. & 1  & 0.679 & 0.677 & 0.272 & 0.271 \\
      &       &       & IDSA-PIP & \textbf{0.681} & 553   & \textbf{0.681} & \textbf{0.681} & 0.281 & 0.280 \\
\cmidrule{3-10}      &       & \multirow{2}{*}{4} & Full MIP & infeas. & 1  & 0.680 & 0.679 & 0.269 & 0.270 \\
      &       &       & IDSA-PIP & \textbf{0.681} & 783   & \textbf{0.681} & \textbf{0.679} & 0.283 & 0.284 \\
\midrule
\multirow{8}{*}{4} & \multirow{8}{*}{0.36} & \multirow{2}{*}{1} & Full MIP & infeas. & 1  & 0.785 & 0.786 & 0.354 & 0.355 \\
      &       &       & IDSA-PIP & \textbf{0.888} & 1564  & \textbf{0.888} & \textbf{0.887} & 0.517 & 0.520 \\
\cmidrule{3-10}      &       & \multirow{2}{*}{2} & Full MIP & infeas. & 2  & 0.786 & 0.785 & 0.355 & 0.354 \\
      &       &       & IDSA-PIP & \textbf{0.843} & 1094  & \textbf{0.846} & \textbf{0.850} & 0.438 & 0.444 \\
\cmidrule{3-10}      &       & \multirow{2}{*}{3} & Full MIP & infeas. & 2  & 0.786 & 0.785 & 0.355 & 0.354 \\
      &       &       & IDSA-PIP & \textbf{0.888} & 1415  & \textbf{0.889} & \textbf{0.881} & 0.519 & 0.511 \\
\cmidrule{3-10}      &       & \multirow{2}{*}{4} & Full MIP & infeas. & 2  & 0.783 & 0.783 & 0.352 & 0.352 \\
      &       &       & IDSA-PIP & \textbf{0.780} & 1053  & \textbf{0.781} & \textbf{0.780} & 1.000 & 1.000 \\
\bottomrule
\end{tabular}
\footnotetext{Note: The extremely short times of Full MIP indicate that the corresponding objective values have not changed from the very beginning to the end of the optimization process.}
\end{table}

\begin{figure}[h]
\centering
\includegraphics[width=\textwidth]{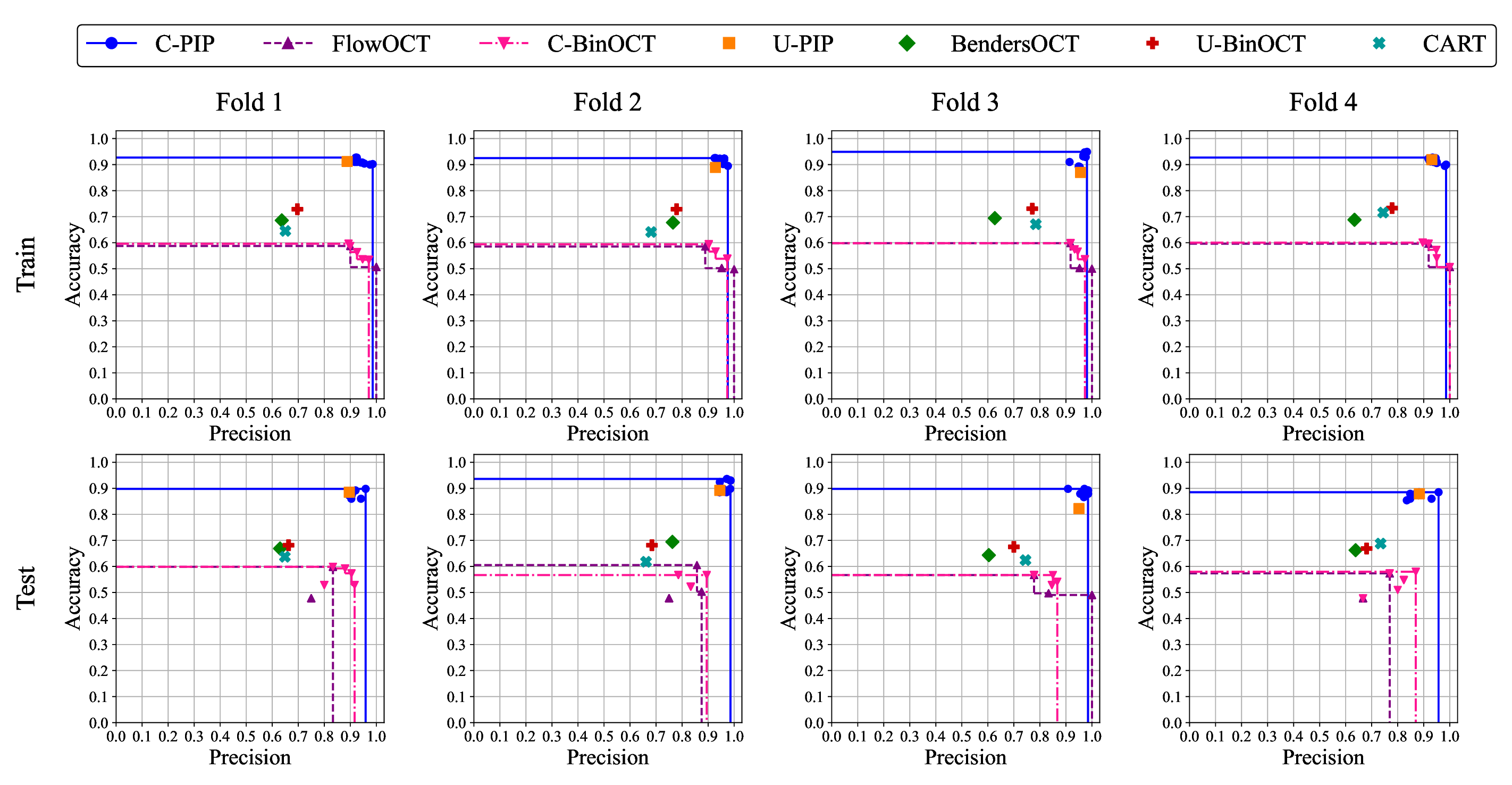}
\caption{Accuracy-precision plots and pareto curve, \textit{blsc} dataset, depth-2.}
\label{fig:pareto-curve-blsc-depth2}
\end{figure}

\begin{figure}[h]
\centering
\includegraphics[width=\textwidth]{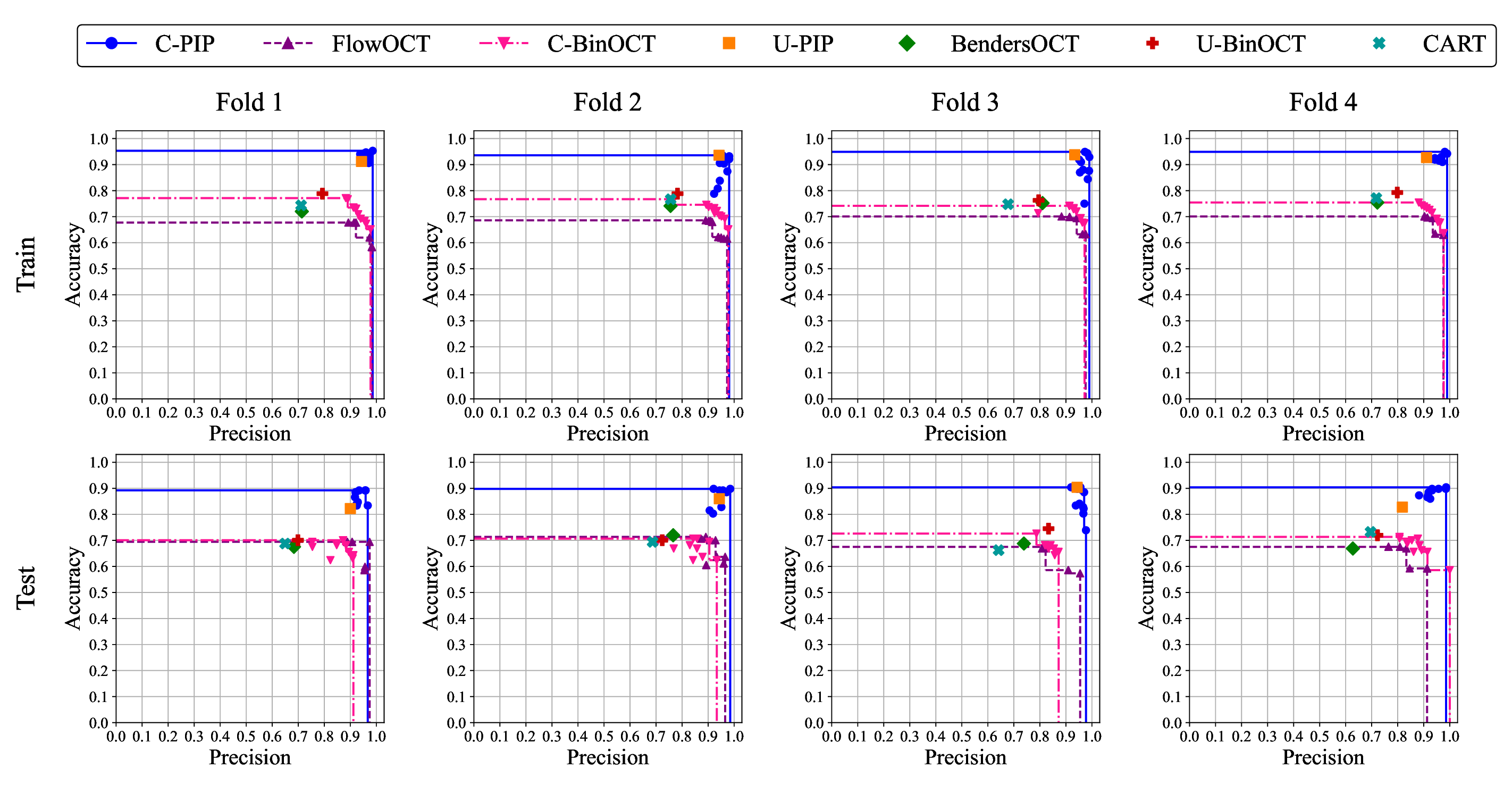}
\caption{Accuracy-precision plots and pareto curve, \textit{blsc} dataset, depth-3.}
\label{fig:pareto-curve-blsc-depth3}
\end{figure}

\begin{figure}[h]
\centering
\includegraphics[width=\textwidth]{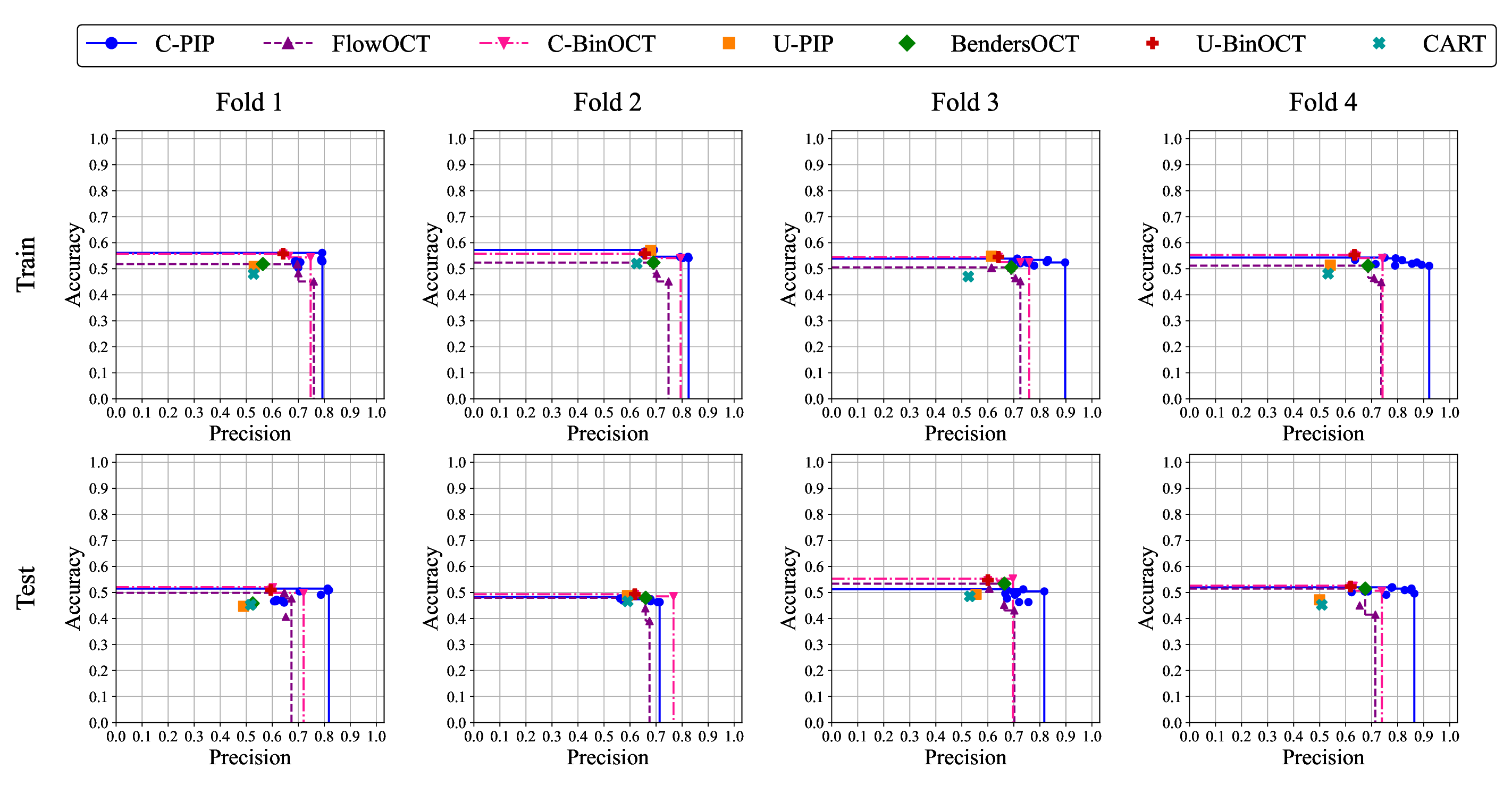}
\caption{Accuracy-precision plots and pareto curve, \textit{ctmc} dataset, depth-2.}
\label{fig:pareto-curve-ctmc-depth2}
\end{figure}

\begin{figure}[h]
\centering
\includegraphics[width=\textwidth]{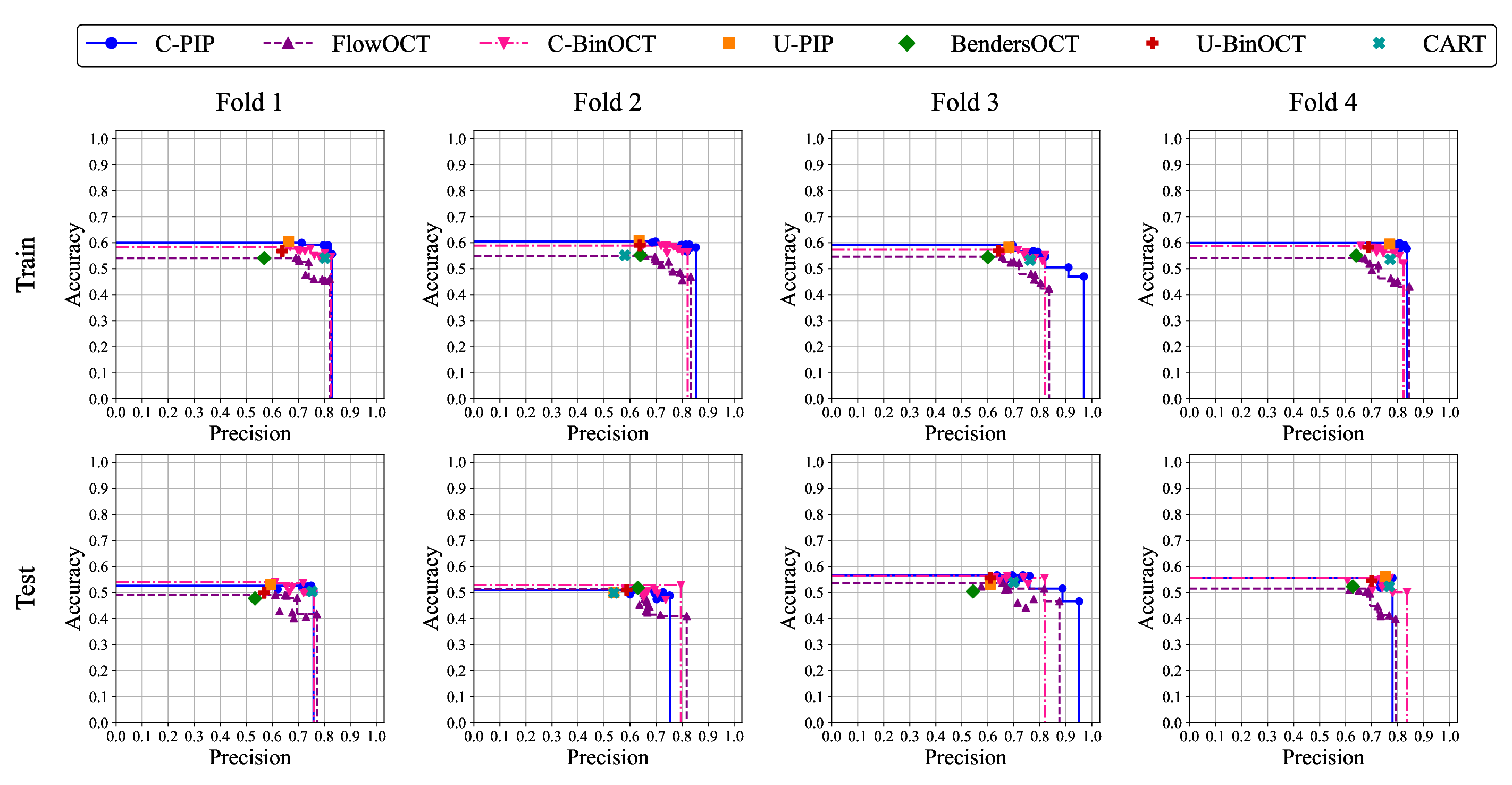}
\caption{Accuracy-precision plots and pareto curve, \textit{ctmc} dataset, depth-3.}
\label{fig:pareto-curve-ctmc-depth3}
\end{figure}

\begin{figure}[h]
\centering
\includegraphics[width=\textwidth]{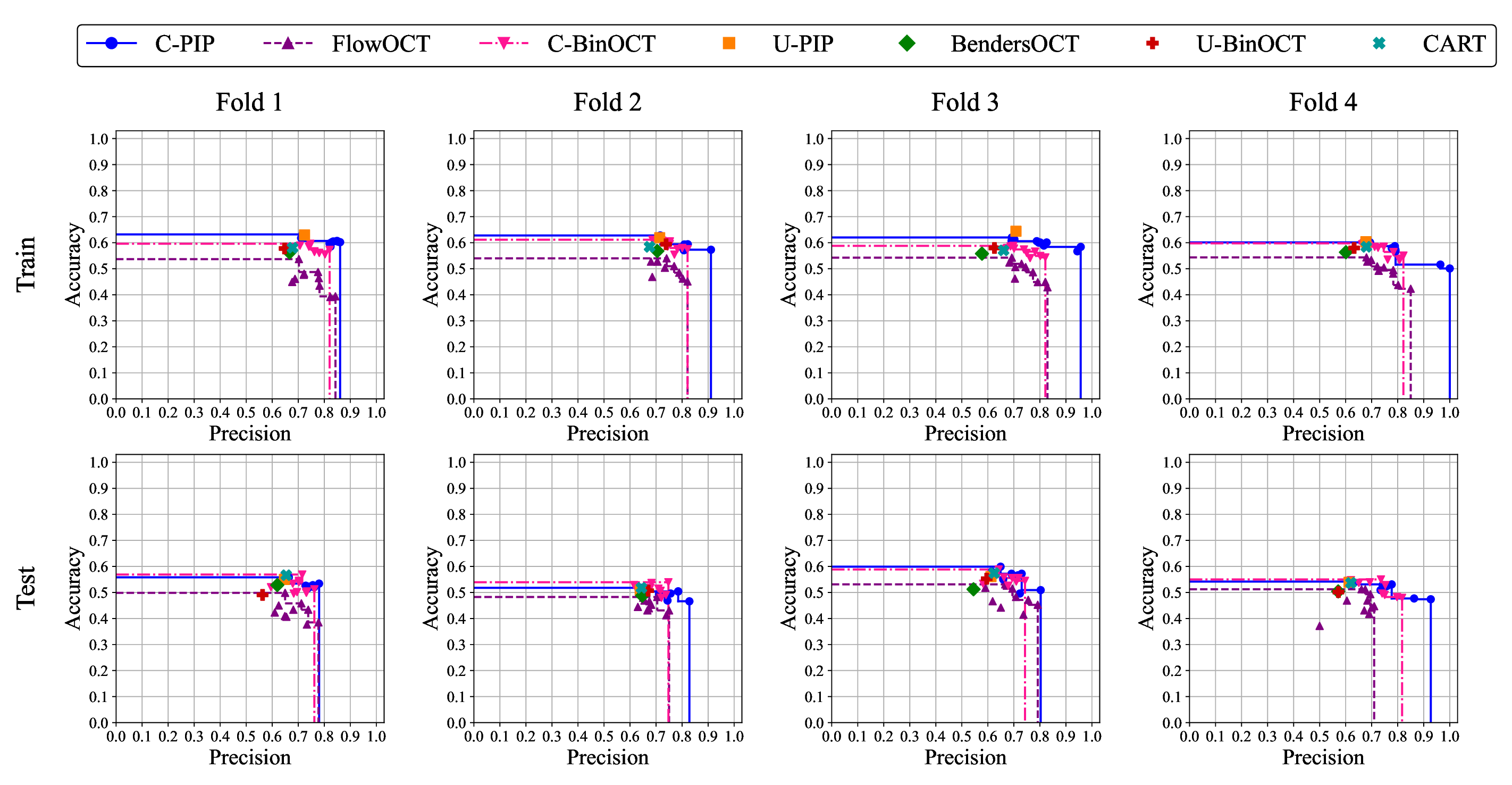}
\caption{Accuracy-precision plots and pareto curve, \textit{ctmc} dataset, depth-4.}
\label{fig:pareto-curve-ctmc-depth4}
\end{figure}

\end{appendices}


\bibliography{reference}

\end{document}